\documentclass[graybox]{SNmult}

\usepackage{type1cm}        % activate if the above 3 fonts are
\usepackage{makeidx}         % allows index generation
\usepackage{graphicx}        % standard LaTeX graphics tool
\usepackage{multicol}        % used for the two-column index
\usepackage[bottom]{footmisc}% places footnotes at page bottom

\usepackage{newtxtext}       % 
\usepackage[varvw]{newtxmath}       % selects Times Roman as basic font

\makeindex             % used for the subject index
\usepackage{hyperref}

\usepackage{tikz}

\usepackage{pgfplots}
\pgfplotsset{compat=1.18}

\newcommand\independent{\protect\mathpalette{\protect\independenT}{\perp}}
\def\independenT#1#2{\mathrel{\rlap{$#1#2$}\mkern2mu{#1#2}}}

\DeclareMathOperator{\pa}{pa}
\DeclareMathOperator{\ch}{ch}
\DeclareMathOperator{\an}{an}
\DeclareMathOperator{\de}{de}
\DeclareMathOperator{\nd}{nd}
\DeclareMathOperator{\nbd}{ne}
\DeclareMathOperator{\cnbd}{\overline{ne}}
\DeclareMathOperator{\CI}{CI}

\begin{document}
\motto{For my son, Declan, whom I can't wait to meet.}
\title*{Algebraic Implicitization Techniques in the Graphical Models Program}
\author{Liam Solus}
\institute{Liam Solus \at KTH Royal Institute of Technology, Stockholm, Sweden \email{solus@kth.se}\\ This is a preprint of a chapter for publication in a Springer Lecture Notes in Mathematics Series.}
\maketitle
\abstract*{Much of our daily lives now rely on technology driven by inference techniques for analyzing multivariate data. 
Graphical models are a general framework for deriving these techniques.
However, different data types require different graphical models, each with their own basic mathematical theory that helps us obtain the desired statistical inference methods. 
Developing the basic theory for a family of graphical models requires solutions to fundamental questions, where emerging mathematical techniques rely on algebraic and geometric methods.  
These notes provide an introduction to the theory of graphical models and the algebraic implicitization techniques now being used to provide the basic mathematical theory for families of graphical models.
}

\abstract{
Many aspects of daily life now rely on technology driven by advanced techniques for analyzing multivariate data. 
Graphical models provide the general framework within which much of these analyses take place.
However, different data types require different graphical models, with each requiring its own basic mathematical theory to guide sound statistical inference. 
Developing this theory for a family of graphical models requires solutions to fundamental questions, where emerging mathematical approaches are utilizing what can collectively be called algebraic implicitization techniques.  
These notes provide an introduction to the basics of graphical models and the algebraic implicitization techniques now appearing in instances of the graphical models program. 
}

%---Introduction
\section{Introduction}
\label{sec: intro}
Following technological breakthroughs of the early 21st century, societal advancements are now largely driven by sophisticated methods for making inferences based on multivariate data. 
For example, in the early 2000s, the internet transformed into a searchable marketplace where users could easily locate homepages, create user profiles, interact, and even generate content of their own. 
This, in turn, afforded companies a new capability to collect data.  
For a single user visit to their webpage (or app), a company could record the time of the visit, the location in the world of the user, what products they viewed, what they `liked', and whether or not they clicked through to a checkout.  
This list of possible measurements is, in fact, conservative in its length relative to the actual number of metrics measured at each user interaction. 
Consequently, many companies now control massive data sets consisting of multivariate observations that carry valuable information about user engagement and interactions. 

More generally, large multivariate data sets are now ubiquitous in most scientific, commerical and industrial disciplines. 
The individual components (or metrics) being measured in these multivariate observations naturally follow a complex dependence structure that describes how the values of one metric in the joint observation relate to the others. 
Gaining insights into the dependence structure underlying the generation of this multivariate data is paramount to making inferences, and the dependence structure itself can help inform how to efficiently and accurately make these inferences. 

Graphical models provide a flexible framework for these tasks. 
The principle of graphical models is simple: Suppose that our multivariate data consists of joint samples from a collection of random variables $X = (X_i)_{i\in V}$.  
Given a graph $\mathcal G = (V,E)$ with node set $V$, each variable $X_i$ is combinatorially represented by a node $i\in V$.  
The edges $E$ encode stochastic dependencies between the individual variables, and hence the graph becomes a compact combinatorial representation of the dependence structure underlying the multivariate data. 

A key advantage of graphical models is the flexibility the framework is afforded by having such a simple guiding principle.  
While the edges of $\mathcal G$ are encoding stochastic dependencies, the user is free to choose precisely \emph{how} the edges represent these dependencies, as well as a parameterization that best fits the specific type of data.  
Once these choices are made, we obtain a set of distributions $\mathcal M (\mathcal G)$, called the \emph{graphical model} for $\mathcal G$. 
The model $\mathcal M(\mathcal G)$ is then a collection of candidate distributions for the unknown distribution of $X = (X_i)_{i\in V}$ where each distribution in the set generates samples in a well-defined manner according to the dependence structure represented by $\mathcal G$. 

\begin{example}
    \label{ex: markov chain}
A classic example is the Markov chain on $X =(X_i)_{i=1}^m$.
The Markov property $X_i \independent X_{\{1,\ldots, i-2\}} \mid X_{i-1}$ for $i = 1,\ldots, m-1$ is visually encoded in the graph
\begin{center}
\begin{tikzpicture}[thick, scale=0.6]
  % \tikzset{>={Stealth[length=3mm, round]}}
  \node[circle, draw, inner sep=1pt, minimum width=2pt] (1) at (0,0)  {$1$};
  \node[circle, draw, inner sep=1pt, minimum width=2pt] (2) at (2,0) {$2$};
  \node[circle, inner sep=1pt, minimum width=1pt] (3) at (4,0) {$\cdots$};
  \node[circle, draw, inner sep=2pt, minimum width=2pt] (4) at (6,0) {$m$};

  % \draw (1) edge [->, very thick, loop left] (1);
  % \draw (2) edge [->, very thick, loop left] (2);
  % \draw (3) edge [->, very thick, loop right] (3);

  \draw[->, very thick] (1) -- (2);
  \draw[->, very thick] (2) -- (3);
  \draw[->, very thick] (3) -- (4);
\end{tikzpicture}
\end{center}
The conditional independence relations associated to the graph $\mathcal G$ above are fulfilled, for instance, by any distribution for $X$ satisfying the noisy linear relations 
\begin{equation}
    \label{eqn: markov chain}
    X_i = \lambda_{i-1}X_{i-1} + \varepsilon_i, \qquad i = 1,\ldots, m
\end{equation} 
for $\lambda_1,\ldots, \lambda_{n-1}\in \mathbb R$ and a choice of mutually independent and normally distributed error terms $\varepsilon_i\sim \mathcal N(0, \omega_i)$ with mean $0$. 
By varying the $\lambda_i\in \mathbb R$ and $\omega_i\in\mathbb R_{>0}$ over all possible choices, we obtain a collection of distributions defined by the relations in~\eqref{eqn: markov chain}
that form the graphical model $\mathcal M(\mathcal G)$. 
\end{example}

Once a rule for constructing the graphical model $\mathcal M(\mathcal G)$ is specified, 
varying $\mathcal G$ over a collection of graphs $\mathbb G$ and constructing $\mathcal M(\mathcal G)$ via the specified rule produces a \emph{family} of graphical models $\mathcal F_\mathbb G = \{\mathcal M(\mathcal G): \mathcal G \in\mathbb G\}$. 
Oftentimes, the rule specifying $\mathcal M(\mathcal G)$ is a choice of parameterization, in which case $\mathcal F_{\mathbb G}$ is a \emph{parametric} graphical model family. 
By changing the parameterization rule we can produce multiple graphical model families for the same set of graphs $\mathbb G$, each of which may fit a different type of data.   
Alternatively, by changing $\mathbb G$, we can model different flavors of dependence structure. 

This flexibility of graphical models is one of the key factors underlying their recent, and notably broad, applications. 
However, making use of graphical models requires inference techniques fit to the chosen graphical model family. 
Since each family $\mathcal F_{\mathbb G}$ is defined according to a different set of assumptions, we are tasked with addressing several fundamental challenges for \emph{each} family.  
These challenges are collectively referred to as the \emph{graphical models program}.

\begin{trailer}{The Graphical Models Program}
For a graphical model family $\mathcal{F}_{\mathbb G}$:
\begin{enumerate}
    \item (Model distinguishability) Determine when $\mathcal M(\mathcal G) = \mathcal M(\mathcal H)$ for $\mathcal G,\mathcal H\in \mathbb G$. 
    \item (Parameter identifiability and inference) Deduce if the value of a given parameter can be uniquely determined for a typical distribution belonging to $\mathcal M(\mathcal G)\in \mathcal F_{\mathbb G}$. Provide sound estimators for the parameter based on the observed data. 
    \item (Model selection) Determine an optimal graph $\mathcal G\in \mathbb G$ and distribution in $\mathcal M(\mathcal G)\in \mathcal F_{\mathbb G}$ that best fits the observed multivariate sample. 
\end{enumerate}
\end{trailer}

The three tasks in the graphical models program are related, and adequate answers to each are required for making sound inferences with the graphical model family $\mathcal F_{\mathbb G}$. 

Task~(2) is familiar from a basic statistics course, where the focus is on developing sound estimators for parameters of interest. 
The concept of parameter identifiability here is critical to sound interpretations of these estimators.  
For example, if multiple values $\lambda_i\neq \lambda_i'$ could be used in Example~\ref{ex: markov chain} to parameterize the same distribution in $\mathcal M(\mathcal G)$, we would need to ask ourselves which of the two values our estimators are approximating and what exactly this means for our model. 

Task~(1) also relates to sound interpretability of parameter estimates. 
For example, if a graph $\mathcal H$ satisfies $\mathcal M(\mathcal G) = \mathcal M(\mathcal H)$ for the Markov chain $\mathcal G$ in Example~\ref{ex: markov chain}, but $\mathcal H$ lacks an edge $i\to i + 1$ belonging to $\mathcal G$ then we need to be careful when interpreting an estimate of the parameter $\lambda_i$. 

Task~(3) is a process commonly referred to as \emph{structure learning}, and it is closely related to tasks~(1) and~(2).  
Before we can begin making inferences about parameters based on a dependence structure $\mathcal G$, we must first have a good way of deciding \emph{which} dependence structure $\mathcal G\in \mathbb G$ is most representative of our data. 

Interestingly, for many graphical model families there is a fundamental geometric and algebraic component to completing the graphical models program.  
For example, the Markov chain model $\mathcal M(\mathcal G)$ in Example~\ref{ex: markov chain} consists of multivariate Gaussian distributions that are completely described by their covariance matrices $\Sigma \in \textrm{Sym}(\mathbb R^m)$. 
Hence, $\mathcal M (\mathcal G)$ admits a `geometric' realization as a subset of covariance matrices living in the space of real symmetric $m\times m$ matrices $\textrm{Sym}(\mathbb R^m)$. 
This subset of $\textrm{Sym}(\mathbb R^m)$ is specified \emph{parametrically} via the equations in~\eqref{eqn: markov chain}, but it turns out that $\mathcal M(G)$ is equivalently defined as the set of points in $\textrm{Sym}(\mathbb R^m)$ satisfying a collection of equations $f_1(\Sigma) = 0,\ldots, f_s(\Sigma) = 0$ and inequalities $g_1(\Sigma)\geq 0, \ldots, g_t(\Sigma)\geq 0$ for a finite set of polynomials $f_1,\ldots, f_s, g_1,\ldots, g_t$ in the entries $\sigma_{ij}$ of the matrix $\Sigma = (\sigma_{ij})_{i,j=1}^m$. 

Identifying these polynomial constraints amounts to finding a so-called \emph{implicit representation} of the model $\mathcal M(\mathcal G)$, which is a useful tool when completing the graphical models program for the family $\mathcal F_{\mathbb G}$. 
Until recently, the theory underlying implicitization techniques and the theory underlying graphical models have developed largely in parallel. 
This is because the former finds its basis in the field of algebraic geometry, while the latter has roots in mathematical statistics. 
The application of algebraic implicitization in the graphical models program is new territory, with first explorations lead by statisticians including Mathias Drton and Steffen Lauritzen \cite{drton2018,Lauritzen1996}, and mathematicians including Bernd Sturmfels and Seth Sullivant \cite{DSS, sullivant2023}. 

The purpose of these notes is to place side-by-side the basic principles of the graphical models program and algebraic implicitization techniques. 
In an effort to bridge disciplines, we describe the basic theory of each, written with the intent of being accessible to statisticians unfamiliar with algebra as well as to algebraists unfamiliar with statistics.
To be relatively self-contained, we include proofs of some fundamental results. 
However, many of these results, along with additional details in each topic, can be found in classical texts, such as \cite{KF, Lauritzen1996, murphy, pearl2009} for graphical models and \cite{cox, DSS, sullivant2023} for algebraic implicitization. 
To demonstrate the value of algebraic implicitization techniques in the graphical models program, these notes culminate with an example of a parametric graphical model family for which the only known solution to the graphical models program uses the methods of algebraic implicitization. 
The complete details of this family can be found in the recent paper of Boege et al. \cite{boege}.

%---Probability and Graphs
\section{Probability and Graphs}
\label{sec: graphical models}

Our first task is to formally define graphical models and describe their fundamental properties. 
Since the goal of graphical models is to model the dependence relations underlying observed multivariate data,  we need to establish some basic probabilistic assumptions that we will use throughout. 

\subsection{Probabilistic assumptions}
\label{subsec: probability}
% Throughout these notes, we will use the following notation and assumptions. 
For a positive integer $m$, we let $[m] = \{1,\ldots, m\}$. 
Let $X = (X_i)_{i\in[m]}$ be a collection of random variables where $X_i$ takes values in a set $\mathcal X_i$. 
For a subset $A\subseteq [m]$, $X_A = (X_i)_{i\in A}$ will denote the subcollection of variables in $X$ with indices in $A$ and $\mathcal X_A = \prod_{i\in A}\mathcal X_i$ will denote its induced sample space.  
Note that $X = X_A$ in the case that $A = [m]$. 
Analogously, we write $\mathcal X$ for $\mathcal X_{[m]}$. 
A realization of $X_A$, i.e.~an element of $\mathcal X_A$, will be denoted $x_A = (x_i)_{i\in A}$. 

The set $\mathcal X$ is the collection of all multivariate observations that we can potentially observe in our data-generating scenario. 
Since the data-generating process is unknown to us, we will assume it follows some (unknown) joint distribution $\mathbb P$ for the variables $X$. 
To isolate a relatively general yet self-contained scenario, we assume the joint distributions $\mathbb P$ are absolutely continuous with respect to some product measure on $\mathcal X = \prod_{i\in[m]}\mathcal X_i$. 
Hence, each distribution $\mathbb P$ has a density $f_X(x) = f_X(x_1,\ldots,x_m)$, which we assume satisfies the positivity assumption $f_X(x)>0$ for all $x\in \mathcal X$. 

\begin{definition}
    \label{def: distribution set}
    Let $X=(X_i)_{i\in[m]}$ be a collection of random variables taking values in the joint sample space $\mathcal X = \prod_{i=1}^m\mathcal X_i$.
    We let $\mathbb D_X$ denote the set of all distributions $\mathbb P$ for $X$ that are absolutely continuous with respect to some product measure $\mu$ on $\mathcal X$ and have density function $f_X(x)$  satisfying $f_X(x)>0$ for all $x\in \mathcal X$.
\end{definition}

For instance, most of our examples will either have all $X_i$ continuous (with $\mathcal X_i = \mathbb R$ and $\mu$ the Lebesgue measure), or all $X_i$ discrete, in which case $\mathcal X_i = [d_i]$ for some positive integers $d_1,\ldots, d_m$. 
In this context, $f_X(x)$ is the familiar the probability density (resp. mass) function of $\mathbb P$. 
That is, for $(b_1,\ldots, b_m)\in \mathcal X$
\[
\begin{split}
    \textrm{Pr}(X_1\leq b_1,\ldots, X_m\leq b_m) &= \int_{-\infty}^{b_1}\cdots \int_{-\infty}^{b_m}f_X(x_1,\ldots, x_m)dx_1\cdots dx_m \quad \textrm{(continuous)}\\
    \textrm{Pr}(X_1= b_1,\ldots, X_m = b_m) &= f_X(b_1,\ldots, b_m) \qquad \qquad \qquad \qquad \qquad \quad \, \,   \textrm{(discrete)}
\end{split}
\]
where, in particular, $f_X(X)$ integrates to $1$ with respect to the base measure $\mu$ over the sample space $\mathcal X$ and $\mu$-measure zero sets all have probability $0$. 

When $\mathcal X_i = [d_i]$ for $i = 1,\ldots, m$ and $\mu$ is the counting measure, the set $\mathbb D_X$ defined in Definition~\ref{def: distribution set} restricts to the set of distributions $\mathbb P$ with a mass function $f_X(x)$ specifying a point in the \emph{(open) probability simplex}
\[
\Delta_{\mathcal X}^\circ = \left\{p = (p_x)_{x\in \mathcal X}\in \mathbb R^{\mathcal X} : p_x > 0 \textrm{ for all $x\in \mathcal X$ and } \sum_{x\in \mathcal X}p_x = 1\right\}; 
\]
that is, $(f_X(x))_{x\in \mathcal X}\in \Delta_{\mathcal X}^\circ$ if and only if $\mathbb P \in \mathbb D_X$. 

On the other hand, when the $X_i$ are continuous, we may consider the set of all multivariate normal distributions for $X = (X_i)_{i\in[m]}$ with mean vector $0$, denoted $\mathcal N_{0,m}$.  
In this case, we have $\mathcal N_{0,m}\subseteq \mathbb D_X$, but this containment is strict.

While much of the theory we will discuss applies beyond the assumptions defining $\mathbb D_X$ (see \cite{Lauritzen1996}), we will be working with conditional probabilities and therefore care needs to be taken \cite{Dawid1980}.  
The assumptions defining $\mathbb D_X$ in Definition~\ref{def: distribution set} assure that the techniques in these notes can be applied in many modern data science scenarios without fear.  

Specifically, for $\mathbb P\in \mathbb D_X$, we obtain the marginal density $f_{X_A}(x_A)$ by integrating $f_X(x)$ over $\mathcal X_{[m]\setminus A}$ with respect to the base measure $\mu$.
Given $x_C\in \mathcal X_C$ with $f_{X_C}(x_C)>0$, we can then define the conditional density $f_{X_A \mid X_C}(x_A\mid x_C)$ for $X_A$ given $X_C = x_C$ as
\[
f_{X_A \mid X_C}(x_A \mid x_C) = \frac{f_{X_{A\cup C}}(x_A, x_C)}{f_{X_C}(x_C)} 
\]
for all $x_A\in \mathcal X_A$. 

The guiding principle of graphical models is that a graph may be used to combinatorially represent potential dependence relations in a distribution $\mathbb P$ for $X$. 
Thinking conversely, a graphical model specifies independence information via the \emph{absence} of edges in a graph. 

The most basic form of stochastic independence is, of course, conditional independence. 
We say that $X_A$ is \emph{conditionally independent} of $X_B$ given $X_C$ if 
\[
f_{X_{A\cup B} \mid X_C}(x_A,x_B\mid x_C) = f_{X_A \mid X_C}(x_A \mid x_C)f_{X_B \mid X_C}(x_B \mid x_c)  
\]
for all $x_A\in \mathcal X_A, x_B\in \mathcal X_B$ whenever $f_{X_C}(x_C)>0$.
When $X_A$ is conditionally independent of $X_B$ given $X_C$, we write $X_A\independent X_B \mid X_C$. 
Otherwise, we say $X_A$ is (conditionally) \emph{dependent} on $X_B$ given $X_C$, which we denote by $X_A \not\independent X_B \mid X_C$. 

Graphical models are diverse, with different types of graphical models being used to model different types of dependence relations that may underlie the generation of multivariate data \cite{collazo, duarte, Lauritzen1996, richardson, peters2017, shimizu2006, varando}.
In fact, the type of dependence driving the definition of a certain graphical model need not be conditional independence (see for instance \cite{varando}). 
However, conditional independencies typically appear in one way or another when working with graphical models. 
The positivity assumption on $\mathbb P\in \mathbb D_X$ in Definition~\ref{def: distribution set} ensures that we can always accommodate any conditional independencies that may arise. 
Specifically, it guarantees that $f_{X_C}(x_C)>0$ for all $C\subseteq[m]$ and all $x_C\in \mathcal X_C$, allowing us to freely consider conditional distributions. 
Finally, we also operate under the basic topological assumption that all discrete functions are continuous.  
In particular, when all $X_i$ are discrete, the continuity assumption is always satisfied. 

\subsection{Graphs}
\label{subsec: graphs}

A graph $\mathcal G = (V,E)$ is an ordered pair of sets, where $V$ is the set of \emph{vertices} (or \emph{nodes}) of $\mathcal G$, and $E$ is its set of edges. 
We will most often let $V = [m]$ for a positive integer $m$. 
Different types of edges are used to denote different types of dependences, matching basic human intuition.  
For instance, an undirected edge $i - j$ is used to represent the possibility that $X_i$ and $X_j$ covary, while a directed edge $i\to j$ is more suggestive of a directional relation, e.g., that the observed outcomes of $X_i$ are effecting the outcomes we observe for $X_j$. 
A third type of edge serves as an intuitive representation of two nodes $i,j$ each being effected by some ``unseen'' node $k$.  
This type of edge is \emph{bidirected} $i\leftrightarrow j$ which represents the pair of directed edges $i\leftarrow k \rightarrow j$ where the node $k$ is \emph{unobserved} (i.e., $k$ does not correspond to a measured variable $X_k$). 

The edge set can consist of lots of different types of edges (e.g., undirected, directed or bidirected). 
However, unless otherwise stated, we will assume that $E$ contains at most one edge between any two nodes $i,j$; that is, we do not allow, for example, that both $i\leftrightarrow j, i\to j\in E$. 
We also assume that $E$ does not contain any \emph{self-loops}; i.e., there is no edge between $i$ and itself for any node $i\in V$. 

If the edge set $E$ contains a mixture of undirected, directed and bidirected edges then we call $\mathcal G$ a \emph{mixed graph}. 
If the edge set $E$ consists of only undirected edges, then $\mathcal G$ is an \emph{undirected graph}.  
Similarly, if $E$ contains only directed edges, then $\mathcal G$ is a \emph{directed graph}. 
A \emph{directed cycle} in $\mathcal G$ is a subset of directed edges $S$ of $E$ such that 
\[
S = \{ i_0 \to i_1,i_1\to i_2,\ldots, i_{k-1}\to i_k,i_k \to i_0\}. 
\]
If we draw a picture of a directed cycle, we see it is intuitively describing a feedback loop in the system: 
\begin{center}
\begin{tikzpicture}[thick, scale=0.6]
    \node[circle, draw, inner sep=1pt, minimum width=1pt] (1) at (0,0)  {$1$};
    \node[circle, draw, inner sep=1pt, minimum width=1pt] (2) at (2,0) {$2$};
    \node[circle, draw, inner sep=1pt, minimum width=1pt] (3) at (0,-2) {$3$};
    \node[circle, draw, inner sep=1pt, minimum width=1pt] (4) at (2,-2) {$4$};
    
    \draw[->, very thick] (1) -- (2);
    \draw[<-, very thick] (1) -- (3);
    \draw[->, very thick] (2) -- (4);
    \draw[<-, very thick] (3) -- (4);
\end{tikzpicture}
\end{center}
If $\mathcal G$ is directed and also contains no \emph{directed cycles} then we call $\mathcal G$ a \emph{Directed Acyclic Graph} (DAG).  
DAGs capture the intuition of a system of directed dependencies without feedback loops. 
Here are examples (from left-to-right) of an undirected graph, a DAG and a mixed graph. 
\begin{center}
\begin{center}
\begin{tikzpicture}[thick, scale=0.6]
    \node[circle, draw, inner sep=1pt, minimum width=1pt] (1) at (0,0)  {$1$};
    \node[circle, draw, inner sep=1pt, minimum width=1pt] (2) at (2,0) {$2$};
    \node[circle, draw, inner sep=1pt, minimum width=1pt] (3) at (0,-2) {$3$};
    \node[circle, draw, inner sep=1pt, minimum width=1pt] (4) at (2,-2) {$4$};
    
    \draw[-, very thick] (1) -- (2);
    \draw[-, very thick] (1) -- (3);
    \draw[-, very thick] (2) -- (4);
    \draw[-, very thick] (3) -- (4);

    \node[circle, draw, inner sep=1pt, minimum width=1pt] (d1) at (0 + 6,0)  {$1$};
    \node[circle, draw, inner sep=1pt, minimum width=1pt] (d2) at (2 + 6,0) {$2$};
    \node[circle, draw, inner sep=1pt, minimum width=1pt] (d3) at (0 + 6,-2) {$3$};
    \node[circle, draw, inner sep=1pt, minimum width=1pt] (d4) at (2 + 6,-2) {$4$};
    
    \draw[->, very thick] (d1) -- (d2);
    \draw[->, very thick] (d1) -- (d3);
    \draw[->, very thick] (d2) -- (d4);
    \draw[->, very thick] (d3) -- (d4);

    \node[circle, draw, inner sep=1pt, minimum width=1pt] (b1) at (0 + 6 + 6,0)  {$1$};
    \node[circle, draw, inner sep=1pt, minimum width=1pt] (b2) at (2 + 6 + 6,0) {$2$};
    \node[circle, draw, inner sep=1pt, minimum width=1pt] (b3) at (0 + 6 + 6,-2) {$3$};
    \node[circle, draw, inner sep=1pt, minimum width=1pt] (b4) at (2 + 6 + 6,-2) {$4$};
    
    \draw[->, very thick] (b1) -- (b2);
    \draw[-, very thick] (b1) -- (b3);
    \draw[<->, very thick] (b2) -- (b4);
    \draw[->, very thick] (b3) -- (b4);
\end{tikzpicture}
\end{center}
\end{center}

If the edge set $E$ contains only directed and bidirected edges then we call $\mathcal G$ a \emph{directed mixed graph}. 
Thinking back to our description of a bidirected edge, we see that directed mixed graphs represent systems of directed dependencies where some of the vertices that could be impacting others are not observed.  
Similarly, \emph{acyclic directed mixed graphs} (ADMGs) do not contain directed cycles, and hence describe such systems where there are no feedback loops. 

\begin{remark}[On the use of bidirected and undirected edges]
Since the literature on graphical models is - relativistically speaking -  still in its early years, it is not unusual to see bidirected edges used in one paper or software package to carry a meaning that is assigned to undirected edges in another.  
This is a foreseeable consequence of the edges both intutively representing a sort of `unspecifiable' covariation.  

The distinction between the edge types we outlined above is one typically abided by in the graphical models literature pertaining to causality.
In this context, the types of edges are meant to distinguish between \emph{causal relations} $i\to j$, \emph{confounding} $i \leftrightarrow j$ (typically by some unobserved common cause) and correlation without a discernible causal interpretation ($i-j$).  

It is also possible to imagine scenarios where a relation is confounded in such a way that one can simply not distinguish the direction of an edge even though one may deduce that there is a directed edge between two nodes (i.e., we know either $i\to j$ or $j\to i$ exists but we do not know which one). 
Hence, in some literature and software packages, the bidirected edge is also used to represent this type of unspecifiability.  
However, since this situation most frequently arises due to a lack of identifiability as opposed to some latent effect, much of the literature uses the undirected edge $i - j$ to represent such situations. 
We will follow this convention. 
\end{remark}

To solidify this intuition, let's take a look at some examples of graphs from the literature that are used to model complex systems with edges interpreted as above. 

\begin{example}[Molecular Biology]
\label{ex: sachs network}
Graphical models are often applied in the fields of molecular biology and genomics, with a classic example being the protein signaling network studied by Sachs et al. \cite{sachs}. 
Here, the multivariate data set consist of joint samples of the expression levels of $11$ different molecules, with each sample being a simultaneous measurement of the abundance of each molecule in a single cell. 
The commonly accepted ground truth network representing how the individual molecule's expression levels are effecting one another is represented using the following graph with directed and bidirected edges. 
\begin{center}
\begin{tikzpicture}[thick]
    \def\n{11} 
    \def\radius{2.5cm}

    \foreach \name [count=\i] in {PKA, PKC, PIP3, Mek, Raf, PIP2, Plcg, Jnk, P38, Akt, Erk}{
        
        % Calculate angle evenly divided by the number of nodes (6 nodes total)
        \pgfmathsetmacro{\angle}{(\i-1) * (360/\n)}
        
        % Draw the node using polar coordinates: (angle:radius)
        % (node_id) is set to lower-case index-based string or name reference
        \node[circle, inner sep=1pt, minimum width=1pt] (\name) at (\angle:\radius) {\name};
    }

    % Example: You can now reference the specific names directly to draw paths!
    \draw[->, thick] (PKA) -- (Jnk);
    \draw[->, thick] (PKC) -- (Jnk);
    \draw[->, thick] (PKA) -- (P38);
    \draw[->, thick] (PKC) -- (P38);
    \draw[->, thick] (PKA) -- (Akt);
    \draw[->, thick] (PIP3) -- (Akt);
    \draw[->, thick] (PKA) -- (Erk);
    \draw[->, thick] (Mek) -- (Erk);
    \draw[->, thick] (PKA) -- (Mek);
    \draw[->, thick] (PKC) -- (Mek);
    \draw[->, thick] (Raf) -- (Mek);
    \draw[->, thick] (PKA) -- (Raf);
    \draw[->, thick] (PKC) -- (Raf);
    \draw[->, thick] (PIP2) -- (PKC);
    \draw[->, thick] (Plcg) -- (PKC);
    \draw[->, thick] (PIP3) -- (PIP2);
    \draw[->, thick] (Plcg) -- (PIP2);
    \draw[->, thick] (PIP3) -- (Plcg);
    \draw[->, thick] (PIP3) -- (PKA);
    \draw[->, thick] (PKA) -- (PIP3);

\end{tikzpicture}
\end{center}
\end{example}

\begin{example}[Human Behavior]
    \label{ex: traffic network}
    Other examples of graphical models in practice arise in economics, sociology, anthropology and other fields studying human behavior.  For instance, \cite{transport} recently used graphical models to represent and make inferences regarding the behavior of users of the Hong Kong subway system and how their behavior is effected by incentive programs aimed at encouraging people to travel before peak rush hour traffic. In their paper, they modeled the underlying measured features -  such as home location, work location, trip start location, trip end location, trip start-time and trip end-time -- with the following graph: 
    \begin{center}
\begin{tikzpicture}[scale=1.6, font=\fontsize{8}{0}\selectfont]
			\draw
			(0.0:2) node (flexibility){flexibility}
			(40.0:2) node (trip fare){trip fare}
			(80.0:2) node (trip end time){trip end time}
			(120.0:2) node (trip start time){trip start time}
			(160.0:2) node (avg trip travel time){avg trip travel time}
			(200.0:2) node (trip destination){trip destination}
			(240.0:2) node (mean time shift){mean time shift}
			(280.0:2) node (promo){promo}
			(320.0:2) node (trip origin){trip origin};
			\begin{scope}[-, thick]
				\draw (flexibility) to (trip origin);
				\draw (flexibility) to (mean time shift);
				\draw (flexibility) to (avg trip travel time);
				\draw (flexibility) to (trip destination);
				\draw (trip end time) to (trip start time);
				\draw (trip end time) to (trip fare);
				\draw (trip start time) to (trip fare);
				\draw (avg trip travel time) to (trip destination);
				\draw (avg trip travel time) to (mean time shift);
				\draw (avg trip travel time) to (trip origin);
				\draw (trip destination) to (trip origin);
				\draw (mean time shift) to (trip origin);
				\draw (mean time shift) to (trip destination);
			\end{scope}
			\begin{scope}[->, thick]
				\draw (flexibility) to (trip end time);
				\draw (flexibility) to (trip fare);
				\draw (flexibility) to (trip start time);
				\draw (avg trip travel time) to (trip start time);
				\draw (avg trip travel time) to (trip end time);
				\draw (avg trip travel time) to (trip fare);
				\draw (trip destination) to (trip start time);
				\draw (trip destination) to (trip end time);
				\draw (trip destination) to (trip fare);
				\draw (mean time shift) to (trip end time);
				\draw (mean time shift) to (trip fare);
				\draw (mean time shift) to (trip start time);
				\draw (trip origin) to (trip start time);
				\draw (trip origin) to (trip fare);
				\draw (trip origin) to (trip end time);
				\draw (promo) to (trip end time);
				\draw (promo) to (trip fare);
				\draw (promo) to (trip start time);
			\end{scope}
		\end{tikzpicture}
\end{center}
\end{example}

Graph theory is well-known for its reliance on numerous definitions and notation, and the graphical models literature is no exception. 
The following is a (nonexhaustive) list of common terminology and notation that appears in the literature when working with graphical models:

\begin{trailer}{Collected Graph Theory Terminology and Notation}
Let $G = (V,E)$ be a graph, $i, j\in V$ and $C\subseteq V$.  Recall that we assume $\mathcal G$ has at most one edge between each pair of nodes and contains no self-loops.
\begin{itemize}
    \item \emph{walk} (between $i$ and $j$): A sequence of vertices $(v_k)_{k\in[s]}$ with $v_1 = i$ and $v_s = j$ for which there exists an edge between $v_i$ and $v_{i+1}$ for all $i = 1,\ldots, s-1$.
    \item \emph{path} (between $i$ and $j$): A walk $(v_k)_{k\in[s]}$ with $v_1 = i$ and $v_s = j$ in which no vertex in the sequence is repeated.
    \item \emph{directed path} (from $i$ to $j$): A path $(v_k)_{k\in[s]}$ with $v_1 = i$ and $v_s = j$ for which the edges are directed as $v_i\to v_{i+1}$. 
    \item \emph{cycle}: A walk $(v_k)_{k\in[s]}$ with $v_1 = v_s$ in which no other vertices are repeated. 
    \item \emph{directed cycle}: A cycle $(v_k)_{k\in[s]}$ in which the edges are directed as $v_i\to v_{i+1}$.
    \item \emph{(induced) subgraph}: $G_C = (C, E_C)$ where 
    $
    E_C = \{ e \in E: \textrm{ $e$ connects } i,j\in A\}.
    $
    \item \emph{parents} of $i$: $\pa(i) = \{j\in V : j\to i\in E\}$.  
    \item \emph{children} of $i$: $\ch(i) = \{j\in V: i\to j\in E\}$. 
    \item \emph{descendants} of $i$: $\de(i) = \{j\in V: \textrm{ there exists a directed path from $i$ to $j$}\}$.
    \item \emph{nondescendants} of $i$: $\nd(i) = V\setminus \de(i)$. 
    \item \emph{ancestors} of $i$: $\an(i) = \{j\in V: \textrm{ there exists a directed path from $j$ to $i$}\}$.
    \item \emph{ancestors} of $C$: $\an(C) = \bigcup_{i\in C}\an(i)$.
    \item \emph{adjacent}: $i$ and $j$ are \emph{adjacent} if there is an edge between $i$ and $j$. 
    \item \emph{neighbors (neighborhood)} of $i$: $\nbd(i) = \{j\in V: j \textrm{ and } i \textrm{ are adjacent}\}$.
    \item \emph{closed neighborhood}: $\cnbd(i) = \nbd(i)\cup\{i\}$.
\end{itemize}
\end{trailer}

We also collect our notation for specific families of graphs that will appear in these notes. 

\begin{trailer}{Some families of graphs}
    \begin{itemize}
        \item $\mathbb{UG}$: The set of all undirected graphs
        \item $\mathbb{DAG}$: The set of all DAGs
        \item $\mathbb{DAN}$: The set of all directed ancestral graphs 
        \item $\textrm{c-}\mathbb{UG}$: The set of all colored undirected graphs
        \item $\textrm{c-}\mathbb{DAG}$: The set of all colored DAGs
        \item $\mathbb{BPEC}$: The set of all BPEC DAGs %(see Definition~\ref{def: BPEC DAG})
    \end{itemize}
\end{trailer}
We place these definitions and notation here for convenience, but we will also introduce again them as needed.

\section{Conditional Independence Graphical Models}
\label{subsec: CI models}

The discussion in Section~\ref{subsec: graphs} highlights how different types of graphs can carry different meanings for the dependencies they encode (e.g. correlative or causal).  
In practice, one typically states how they intend to interpret the edges of their graphs before any estimation with data is performed. 
Once a type of graph is chosen (and how it will be interpreted is asserted), the task is then to associate a \emph{statistical model} (e.g. a candidate set of distributions for the unknown data-generating process) to the graph.  
As we will describe, there are many reasonable ways to go about doing this.  
Since the most ubiquitous form of stochastic (in)dependence is conditional (in)dependence, it is natural to start with \emph{conditional independence} (CI) \emph{graphical models}. 
The graphical models program for CI models is now, more or less, complete thanks largely to the work of researchers including Steffen Lauritzen \cite{Lauritzen1996}, Judea Pearl \cite{pearl2009} and Thomas Richardson \cite{richardson}.  
These results form the basis for much of the ongoing research in graphical models. 

For the sake of a completeness, this section includes several classic proofs of fundamental theorems.  
A more complete version of this story, including many of these proofs and detailed extensions of how the results extend beyond the assumptions in Definition~\ref{def: distribution set}, is given in the classic book \emph{Graphical Models} by Steffen Lauritzen \cite{Lauritzen1996}. 

\subsection{The logic of Markov properties}
\label{subsec: markov properties}
The first step in constructing a graphical model family is to specify a collection of graphs $\mathbb G$.  
Typical choices include the set of all undirected graphs, denoted $\mathbb{UG}$, or the set of all directed acyclic graphs (DAGs), denoted $\mathbb{DAG}$.  
However, major advances in the graphical models program for CI models based on more complex families of mixed graphs have also been made in recent years (see, for example, Section~\ref{subsubsec: ancestral models} below). 

The types of edges allowed in the graphs $\mathcal G\in \mathbb G$ determine the conditional independence relations represented by the graphs, based on the intuitive interpretations discussed in Section~\ref{subsec: graphs}.  
To define a conditional independence graphical model for the graph $\mathcal G$, we need a combinatorial rule for assigning a collection of CI relations to $\mathcal G$ based on its edge structure.  Such a rule is called a \emph{Markov property} for $\mathcal G$. 

\begin{definition}
    \label{def: markov property}
    Let $\mathcal G$ be a graph.  
    A \emph{Markov property} for $\mathcal G$ is a rule for assigning a collection $\CI(\mathcal G)$ of conditional independence relations to $\mathcal G$. 
\end{definition}

There are several common principles by which to define a Markov property for $\mathcal G$ that are generally agnostic to the type of edges allowed in $\mathcal G$:
\begin{trailer}{Traditional principles by which to assign a Markov property to a graph}
\label{trail: markov properties}
    For a graph $\mathcal G=(V,E)$, one may specify the set of CI relations $\CI(\mathcal G)$ according to
    \begin{enumerate}
    \item the individual missing edges in $\mathcal G$; i.e., the elements $e\notin E$,
    \item the adjacencies in $\mathcal G$ of individual nodes $v\in V$, or
    \item a notion of `connectedness' for subsets of nodes within the graph $\mathcal G$. 
    \end{enumerate}
\end{trailer}

Each of the three choices above can be rigorously formalized for most types of graphs, and each captures a different intuition about how a graph  encodes (in)dependence structure.  
Method (1) focuses on pairwise interactions, specifying that the absence of an edge between two nodes should imply some conditional independence between the pair, while method (2) focuses on the intuition that each node should be conditionally independent from some others when its immediate neighbors in the graph are considered.  
Alternatively, method (3) takes the perspective that if there is some (rigorous) way in which two subsets of nodes in the graph can be separated, then this separation should constitute a conditional independence.

Each of the three methods for specifying a Markov property will define a model $\mathcal M(\mathcal G)$ for each $\mathcal G\in \mathbb G$.
These three models for the graph $\mathcal G$ are naturally related.  
In the next section, we describe this relation when $\mathbb G$ is the set of all undirected graphs.

\subsection{Conditional independence undirected graphical models}
\label{subsubsec: UG models}

In this section, $\mathcal G = ([m], E)\in \mathbb{UG}$ will be an undirected graph with vertex set $[m]$. %, and $\mathbb G$ will denote the set of all undirected graphs. 
To construct conditional independence models for an undirected graph $\mathcal G$, we use the standard methods for specifying a Markov property for $\mathcal G$ outlined in Section~\ref{subsec: markov properties}.  
This will give us three different CI graphical models for the graph $\mathcal G$. 

\begin{definition}[Markov Properties for Undirected Graphs]
\label{def: undirected MPs}
Let $\mathcal G=([m], E)$ be an undirected graph. 
\begin{enumerate}
    \item The \emph{pairwise Markov property} for $\mathcal G$ is the set of CI relations
    \[
    \CI_P(\mathcal G) = \{X_i\independent X_j \mid X_{[m]\setminus \{i,j\}} : i-j\notin E\}. 
    \]
    \item The \emph{local Markov property} for $\mathcal G$ is the set of CI relations
    \[
    \CI_L(\mathcal G) = \{X_i\independent X_{[m]\setminus \cnbd(i)} \mid X_{\nbd(i)} : i\in[m]\}.
    \]
    \item The \emph{global Markov property} for $\mathcal G$ is the set of CI relations
    \[
    \CI_G(\mathcal G) = \{X_A\independent X_B \mid X_C : A,B,C\subseteq[m] \textrm{ where $A$ and $B$ are separated in $\mathcal G$ by $C$}\}.
    \]
\end{enumerate}
\end{definition}

Note that the pairwise Markov property is a rigorous way of specifying a Markov property based on individual missing edges in undirected graphs. 
This was our first standard way of defining a Markov property in Section~\ref{subsec: markov properties}. 
The local Markov property is a rigorous formulation of the second method, and the global Markov property corresponds to the third. 
For the global Markov property we need to formally state what it means for two subsets of nodes in the graph to be separated by a third.  

To this end, let $s$ be a nonnegative integer and consider a sequence of vertices $(v_i)_{i\in [s+1]}$ belonging to a graph $\mathcal G$.  
We define a \emph{path} (between $v_1$ and $v_{s+1}$) as a sequence of edges $(e_i)_{i\in [s]}$ for some nonnegative integer $s$, where the edge $e_i$ has endpoints $v_i$ and $v_{i+1}$.  
When $s= 0$, the edge sequence is empty, which corresponds to the trivial path between $v_1$ and itself (e.g.~the path using no edges). 

\begin{definition}
    \label{def: undirected connectedness}
    Let $\mathcal G = ([m],E)\in \mathbb{UG}$, and let $A,B,C\subseteq [m]$. 
    We say that $A$ and $B$ are \emph{connected} given $C$ if there exist $a\in A$, $b\in B$ and a path between $a$ and $b$ in $\mathcal G$ that does not use nodes in $C$. 
    Otherwise, we say $A$ and $B$ are \emph{separated} given $C$ in $\mathcal G$, which we denote $A \perp_\mathcal G B \mid C$. 
\end{definition}

\begin{example}
    \label{ex: undirected grid}
    Let $\mathcal G = (V, E)$ be the following graph on node set $V = \{1,\ldots, 9\}$.
    \begin{center}
    \begin{tikzpicture}[thick, scale=0.6]
    % \tikzset{>={Stealth[length=3mm, round]}}
    \node[circle, draw, inner sep=1pt, minimum width=1pt] (1) at (0,0)  {$1$};
    \node[circle, draw, inner sep=1pt, minimum width=1pt] (2) at (1,1) {$2$};
    \node[circle, draw, inner sep=1pt, minimum width=1pt] (3) at (1,-1) {$3$};
    \node[circle, draw, inner sep=1pt, minimum width=1pt] (4) at (2,0) {$4$};
    \node[circle, draw, inner sep=1pt, minimum width=1pt] (5) at (2,2) {$5$};
    \node[circle, draw, inner sep=1pt, minimum width=1pt] (6) at (3,1) {$6$};
    \node[circle, draw, inner sep=1pt, minimum width=1pt] (7) at (2,-2) {$7$};
    \node[circle, draw, inner sep=1pt, minimum width=1pt] (8) at (3,-1) {$8$};
    \node[circle, draw, inner sep=1pt, minimum width=1pt] (9) at (4,0) {$9$};
    
    % \draw (1) edge [->, very thick, loop left] (1);
    % \draw (2) edge [->, very thick, loop left] (2);
    % \draw (3) edge [->, very thick, loop right] (3);
    
    \draw[-, very thick] (1) -- (2);
    \draw[-, very thick] (1) -- (3);
    \draw[-, very thick] (2) -- (4);
    \draw[-, very thick] (2) -- (5);
    \draw[-, very thick] (3) -- (4);
    \draw[-, very thick] (3) -- (7);
    \draw[-, very thick] (4) -- (6);
    \draw[-, very thick] (4) -- (8);
    \draw[-, very thick] (5) -- (6);
    \draw[-, very thick] (7) -- (8);
    \draw[-, very thick] (6) -- (9);
    \draw[-, very thick] (8) -- (9);
    \end{tikzpicture} 
    \end{center}
    Grid graphs such as $\mathcal G$ are the basis of the classic \emph{Ising model} \cite{KF}, which provide a simple model for the spread of infectious disease.  Imagine that each node in the graph is a plant in a field.  We model the infection status of plant $i$ with a binary random variable where $X_i = 1$ if the plant is infected and $X_i = 0$ otherwise.  
    A simple hypothesis is that the disease likely spreads via its nearest neighbors in the grid formation. 
    So the edges in the graph represent that there is a possible correlation between the values of $X_i$ and $X_j$ when plants $i$ and $j$ are neighbors in the field. 

    The Markov properties in Definition~\ref{def: undirected MPs} each encode intuitive information regarding the probability that certain plants are infected given knowledge of the status of other plants. 
    The CI relation $X_4 \independent X_{\{1,5,7,9\}} \mid X_{\{2,3,4,8\}}\in \CI_L(\mathcal G)$ captures the intuition that, provided with the infection status of the plants nearest to $4$, now new information can be gained about the probability of plant $4$ being infected when told the status of plants further away in the field. 
    Similarly, the pairwise relation $X_1 \independent X_9 \mid X_{\{2,3,4,5,6,7,8\}}\in \CI_P(\mathcal G)$ encodes the idea that the status of two non-neighboring plants in the field should not inform one another beyond the information provided via the infection status of all other plants. 
    The global relations, such as $X_{\{2,3\}} \independent X_{\{6,8,9\}} \mid X_{\{4,5,7\}} \in \CI_G(\mathcal G)$ encode the intuition that no new information should be obtainable about the infection status on one side of the field from the infection-status on the other side, provided we already understand the infection status at the middle of the field. 
\end{example}

Example~\ref{ex: undirected grid} gives some basic intuition about how we may use an undirected graph to describe (in)dependence relations in a data problem.  
Depending on the available intuition for the dependence structure underlying the data problem, one Markov property may be more natural to work with than another. 
The task then is to describe the statistical model (e.g. set of distributions) we obtain from our choice of Markov property, and how these different models relate. 
For each of the three Markov properties in Definition~\ref{def: undirected MPs} we may define a conditional independence model associated to the undirected graph $\mathcal G$ as follows. 

\pagebreak

\begin{definition}
    \label{def: undirected GMs}
    Let $G = ([m], E)$ be an undirected graph. 
    \begin{enumerate}
        \item The \emph{pairwise undirected CI graphical model} for $\mathcal G$ is 
        \[
        \mathcal M_P(\mathcal G) = \{ \mathbb P \in \mathbb D_X: \textrm{ $\mathbb P$ satisfies all CI relations in $\CI_P(\mathcal G)$}\}. 
        \]
        \item The \emph{local undirected CI graphical model} for $\mathcal G$ is 
        \[
        \mathcal M_L(\mathcal G) = \{ \mathbb P \in \mathbb D_X: \textrm{ $\mathbb P$ satisfies all CI relations in $\CI_L(\mathcal G)$}\}. 
        \]
        \item The \emph{global undirected CI graphical model} for $\mathcal G$ is 
        \[
        \mathcal M_G(\mathcal G) = \{ \mathbb P \in \mathbb D_X: \textrm{ $\mathbb P$ satisfies all CI relations in $\CI_G(\mathcal G)$}\}. 
        \]
    \end{enumerate}
\end{definition}

When considering the models defined in Definition~\ref{def: undirected GMs}, it is important to remember our basic assumptions on the set of distributions $\mathbb D_X$ in Definition~\ref{def: distribution set}. 
In particular, the models in Definition~\ref{def: undirected GMs} are well-defined since all necessary conditional distributions exist. 
It is not difficult to see that the three models in Definition~\ref{def: undirected GMs} are related. 
To rigorously describe their relationship, we use a fundamental calculus that holds for conditional independence relations. 

\begin{proposition}
    \label{prop: CI axioms}
    Let $\mathbb P\in \mathbb D_X$ be a distribution for $X = (X_i)_{i\in [m]}$, and let $A,B,C,D\subseteq [m]$.  
    The following implications hold for $\mathbb P$.
    \begin{itemize}
        \item (Symmetry) If $X_A \independent X_B \mid X_C$ then $X_B\independent X_A \mid X_C$. 
        \item (Decomposition) If $X_A \independent X_{B\cup D} \mid X_C$ then $X_A \independent X_B \mid X_C$. 
        \item (Weak Union) If $X_A \independent X_{B\cup D} \mid X_C$ then $X_A \independent X_B \mid X_{C\cup D}$.
        \item (Contraction) If $X_A \independent X_B \mid X_{C\cup D}$ and $X_A \independent X_D \mid X_C$ then $X_A \independent X_{B\cup D} \mid X_C$.
        \item (Intersection) If $X_A \independent X_B \mid X_{C\cup D}$ and $X_A \independent X_D \mid X_{B\cup C}$ then $X_A \independent X_{B\cup D} \mid X_C$. 
    \end{itemize}
\end{proposition}

The implications listed in Proposition~\ref{prop: CI axioms} are sometimes referred to as the \emph{conditional independence axioms}, and their proof amounts to straightforward manipulations of the density $f_X(x)$ for $\mathbb P$. 
It is also not hard to show that the same set of implications hold for our notion of undirected graph separation in Definition~\ref{def: undirected connectedness}. 

\begin{proposition}
    \label{prop: UG axioms}
    Let $\mathcal G = ([m], E)$ be an undirected graph, and let $A,B,C,D\subseteq [m]$. 
    The following implications hold 
    \begin{itemize}
        \item (Symmetry) If $A \perp_{\mathcal G} B \mid C$ then $B \perp_{\mathcal G} A \mid C$. 
        \item (Decomposition) If $A \perp_{\mathcal G} {B\cup D} \mid C$ then $A \perp_{\mathcal G} B \mid C$. 
        \item (Weak union) If $A \perp_{\mathcal G} {B\cup D} \mid C$ then $A \perp_{\mathcal G} B \mid {C\cup D}$.
        \item (Contraction) If $A \perp_{\mathcal G} B \mid {C\cup D}$ and $A \perp_{\mathcal G} D \mid C$ then $A \perp_{\mathcal G} {B\cup D} \mid C$.
        \item (Intersection) If $A \perp_{\mathcal G} B \mid {C\cup D}$ and $A \perp_{\mathcal G} D \mid {B\cup C}$ then $A \perp_{\mathcal G} {B\cup D} \mid C$. 
    \end{itemize}
\end{proposition}

One consequence of the conditional independence calculus in Proposition~\ref{prop: CI axioms}, and its graph-theoretic analog in Proposition~\ref{prop: UG axioms}, is that the three Markov properties for $\mathcal G$ given in Definition~\ref{def: undirected MPs} define the exact same subsets of $\mathbb D_X$. 
To the best of the author's knowledge, the following is a combination of results first published by Lauritzen \cite{Lauritzen1996} and Pearl and Paz \cite{pearlpaz}. 

\begin{theorem}
    \label{thm: undirected positive equality}
    Let $\mathcal G = ([m],E)$ be an undirected graph. Then 
    \[
    \mathcal M_P(\mathcal G) = \mathcal M_L(\mathcal G) = \mathcal M_G(\mathcal G).
    \]
\end{theorem}

\begin{proof}
    If $\mathbb P \in \mathcal M_G(\mathcal G)$ then $X_A\independent X_B \mid X_C$ whenever $A$ and $B$ are separated given $C$ in $\mathcal G$. 
    Consider the choice of sets $A = \{i\}$, $B = [m]\setminus \cnbd(i)$ and $C = \nbd(i)$. 
    Then any path from $A$ to $B$ is a sequence of edges $(i_j - i_{j+1})_{j\in[s]}$ such that $i_1 = i$ and $i_{s+1}\in B$. 
    In particular, the first edge $i_1 - i_2$ is between $i$ and a neighbor of $i$; that is, $i_2\in \nbd(i) = C$. 
    Hence, there is no path from $A$ to $B$ in $G$ that does not use nodes in $C$. 
    Thus, $A$ and $B$ are separated given $C$ in $\mathcal G$. 
    It then follows from Definition~\ref{def: undirected GMs}~(3) that $X_A \independent X_B \mid X_C$. 
    Therefore, for every $i\in[m]$ we have $X_i\independent X_{[m]\setminus \cnbd(i)}\mid X_{\nbd(i)}$, which implies $\mathbb P\in \mathcal M_L(\mathcal G)$. 
    Thus, $\mathcal M_G(\mathcal G)\subseteq \mathcal M_L(\mathcal G)$.

    Now let $\mathbb P \in \mathcal M_L(\mathcal G)$, and let $i- j\notin E$. 
    To show that $\mathbb P \in \mathcal M_P(\mathcal G)$, we must show that $X_i\independent X_j\mid X_{[m]\setminus\{i,j\}}$. 
    Since $\mathbb P \in \mathcal M_L(\mathcal G)$, we know that $X_i\independent X_B \mid X_C$ where $B = [m]\setminus \cnbd(i)$ and $C = \nbd(i)$. 
    Note that $j\in B$ since $i - j\notin E$ is equivalent to $j\notin \nbd(i)$.  
    Thus, we may rewrite $B$ as the disjoint union of two sets $B = \{j\} \cup D$ where $D = [m]\setminus (\cnbd(i)\cup \{j\})$. 
    Hence, $X_i \independent X_{\{j\} \cup D} \mid C$, which implies $X_i\independent X_j \mid X_{C\cup D}$ by the weak union property in Proposition~\ref{prop: CI axioms}. 
    Since $C\cup D = \nbd(i)\cup [m]\setminus(\cnbd(i)\cup\{j\}) = [m]\setminus \{i,j\}$, it follows that $\mathcal M_L(\mathcal G)\subseteq \mathcal M_P(\mathcal G)$. 

    The above argument shows that $\mathcal M_G(\mathcal G)\subseteq \mathcal M_L(\mathcal G) \subseteq \mathcal M_P(\mathcal G)$. 
    To complete the proof, we show that $\mathcal M_P(\mathcal G)\subseteq \mathcal M_G(\mathcal G)$.  
    Let $\mathbb P\in \mathcal M_P(\mathcal G)$, and suppose that $A$ and $B$ are separated given $C$ in $\mathcal G$. 
    To prove the result, we must show that $X_A \independent X_B \mid X_C$ in $\mathbb P$. 
    Note that the CI relation $X_A \independent X_B \mid X_C$ is vacuously true whenever $A$ or $B$ is the empty set, so we assume that $A, B\neq \emptyset$. 
    
    Suppose first that $|C| = m - 2$. 
    It follows that $A = \{a\}$ and $B = \{b\}$ are both singleton sets, and $C$ must be equal to the set $[m]\setminus \{a, b\}$. 
    Hence, $X_A \independent X_B \mid X_C$ is the same statement as $X_a \independent X_b \mid X_{[m]\setminus \{a,b\}}$, which is one of the defining CI relations of the pairwise undirected Markov property for $\mathcal G$. 
    Thus, $X_A \independent X_B \mid X_C$ holds in $\mathbb P$. 

    Working by induction on $|C|$, we now assume that $X_A \independent X_B \mid X_C$ holds in $\mathbb P$ whenever $|C| > n$ and consider a choice of the sets $A,B,C$ where $|C| = n < m - 2$. 
    There are then two cases to consider: 

    Suppose first that $[m] = A\cup B \cup C$. 
    Since $|C| < m-2$, then without loss of generality, we may assume that $|A| > 1$. 
    Choose $a\in A$. 
    By the weak union property for undirected graphs in Proposition~\ref{prop: UG axioms}, we have that $A\setminus \{a\} \cup \{a\} \perp_{\mathcal G} B \mid C$ implies $A\setminus \{a\} \perp_{\mathcal G} B \mid C\cup \{a\}$ and $\{a\} \perp_{\mathcal G} B \mid C\cup A\setminus\{a\}$.  
    Since $|A| >1$ and $|C| = n$, we have that $|C\cup A \setminus\{a\}|, |C\cup \{a\}| > n$. 
    Hence, by the inductive hypothesis, $X_{A\setminus \{a\}} \independent X_B \mid X_{C\cup \{a\}}$ and $X_{a} \independent X_B \mid X_{C\cup A \setminus\{a\}}$ both hold in $\mathbb P$. 
    By the intersection axiom in Proposition~\ref{prop: CI axioms}, it follows that $X_A \independent X_B \mid X_C$ also holds in $\mathbb P$. 

    For the second case, suppose that there exists some $v\in [m] \setminus (A\cup B \cup C)$. 
    Since, by assumption, $A$ and $B$ are separated given $C$ in $\mathcal G$, then every path in $\mathcal G$ between a node in $A$ and a node in $B$ uses a node in $C$.  
    Therefore, every path between $A$ and $B$ in $\mathcal G$ uses a node in $C\cup \{v\}$. 
    Hence, $A$ and $B$ are separated given $C\cup\{v\}$ in $\mathcal G$. 
    Since $|C\cup \{v\}| > n$, then the inductive hypothesis implies that $X_A \independent X_B \mid X_{C\cup \{v\}}$ holds in $\mathbb P$. 

    To finish the proof, observe that one of the following separation statements must hold in $\mathcal G$: either $\{v\} \perp_{\mathcal G} B \mid A \cup C$ or $\{v\} \perp_{\mathcal G} A \mid B\cup C$.  
    Otherwise, there would be a path between $A$ and $B$ passing through $v$ that does not use any nodes in $C$, which would contradict our assumption $A \perp_{\mathcal G} B \mid C$. 
    We suppose, without loss of generality, that $\{v\} \perp_{\mathcal G} B \mid A \cup C$ holds. 
    Since$|A|,|B| \geq 1$ then  $|C \cup A| > n$, and the inductive hypothesis implies that $X_v \independent X_B \mid X_{A\cup C}$ holds in $\mathbb P$. 
    Since $X_v \independent X_B \mid X_{A\cup C}$ and $X_A \independent X_B \mid X_{C\cup \{v\}}$ hold in $\mathbb P$, then the intersection property in Proposition~\ref{prop: CI axioms} implies that $X_{A\cup \{ v\}} \independent X_B \mid X_C$ also holds in $\mathbb P$. 
    Applying the decomposition property in Proposition~\ref{prop: CI axioms} to this statement shows that $X_{A} \independent X_B \mid X_C$ in $\mathbb P$, which completes the proof. 
\end{proof}

\begin{remark}
    \label{rem: weakening of assumptions}
    Although we are operating under the assumption that our distributions $\mathbb P$ belong to  $\mathbb D_X$, it is not hard to see that Theorem~\ref{thm: undirected positive equality} can hold under weaker assumptions. 
    In particular, since the proof of Theorem~\ref{thm: undirected positive equality} only relies on the validity of the conditional independence calculus in Propositions~\ref{prop: CI axioms} and~\ref{prop: UG axioms}, all that matters is that the implied CI relations are well-defined in $\mathbb{P}$.  
    For a very general and rigorous formulation of when this can be done, see \cite{Dawid1980}.  

    For example, the positivity assumption $f_X(x)> 0$ for all $x\in \mathcal X$ in Definition~\ref{def: distribution set} for $\mathbb D_X$ is only needed to prove the intersection property in Proposition~\ref{prop: CI axioms}. 
    Hence, the chain of containments $\mathcal M_G(\mathcal G)\subseteq \mathcal M_L(\mathcal G) \subseteq \mathcal M_P(\mathcal G)$ holds even if $\mathbb D_X$ is relaxed to include distributions that are not positive on $\mathcal X$. 
\end{remark}

According to Theorem~\ref{thm: undirected positive equality}, it is now safe for us to refer to \emph{the} undirected conditional independence graphical model for $\mathcal G$.  

\begin{definition}
    \label{def: UG model}
    Let $\mathcal G = ([m],E)$ be an undirected graph. 
    The \emph{undirected conditional independence (CI) graphical model} for $\mathcal G$ is
    \[
    \mathcal M(\mathcal G) = \{\mathbb P\in \mathbb D_X : \mathbb P \textrm{ satisfies one of the three Markov properties in Definition~\ref{def: undirected MPs}}\}.
    \]
    The \emph{undirected CI graphical model family} is 
    \[
    \mathcal F_{\mathbb{UG}} = \{\mathcal M(\mathcal G) : \mathcal G \in \mathbb{UG}\}.
    \]
\end{definition}

It is important to note that the undirected CI graphical model $\mathcal M(\mathcal G)$ is specified with respect to a chosen sample space $\mathcal X = \prod_{i\in [m]}\mathcal X_i$, where the component sample spaces $\mathcal X_i$ are chosen to reasonably fit the data problem.  
In the framework we have provided here, $\mathcal M(\mathcal G)\subseteq \mathbb D_X$, where $\mathbb D_X$ could contain distributions defined with respect to different base measures $\mu$ on $\mathcal X$. 
In practice, it is most common that $\mathcal M(\mathcal G)$ is restricted to a subset of $\mathbb D_X$ for a fixed base measure $\mu$, and perhap even additional assumptions on the type of distributions. 
Hence, it is important to pay attention to the modeling assumptions when reading the literature on CI graphical models. 
A classic example are the undirected Gaussian graphical models. 

\begin{example}[Undirected Gaussian graphical models]
    \label{ex: UGGMs}
    Given an undirected graph $\mathcal G = ([m], E)$,
    let $\textrm{PD}_m$ denote the set of all $m\times m$ positive definite matrices, and let 
    \[
    \mathcal L_{\mathcal G} = \{K = (k_{ij})_{i,j=1}^m\in \textrm{PD}_m : k_{ij} = 0 \textrm{ if } i - j \notin E\}. 
    \]
    The \emph{undirected Gaussian graphical model} for $\mathcal G$ is the set of all multivariate normal distributions $\mathcal N(0, \Sigma)$ with mean $0\in \mathbb R^m$ and covariance matrix $\Sigma \in \textrm{PD}_m$ with concentration matrix belonging to $\mathcal {\mathcal G}$; that is, 
    \[
    \mathcal M_{\textrm{Gauss}}(\mathcal G) = \{\mathcal N(0,\Sigma) : \Sigma^{-1} \in \mathcal L_{\mathcal G}\}. 
    \]
    It is an exercise to show that $\mathcal N(0,\Sigma)\in \mathcal M_{\textrm{Gauss}}(\mathcal G)$ if and only if the distribution $\mathcal N(0,\Sigma)$ for $X = (X_i)_{i\in[m]}$ satisfies $X_i \independent X_j \mid X_{[m]\setminus \{i,j\}}$ for all $i - j \notin E$. 
    In particular, $\mathcal M_{\textrm{Gauss}}(\mathcal G) \subset \mathcal M(\mathcal G)$.  
    It is also common in the literature to drop the subscript notation when it is perceived to be clear that the distributions being considered are Gaussian. 

    We have now hit a point where we can explore real data examples.  
    For instance, we can consider the data set used to model the protein signaling network in Example~\ref{ex: sachs network}. 
    If we want to model the unknown data-generating distribution with a multivariate Gaussian distribution, and we are willing to assume that this unknown distribution abides by the pairwise Markov property in Definition~\ref{def: undirected MPs}, then performing Pearson correlation tests to evaluate the null hypotheses $H_0: X_i \independent X_j \mid X_{[m]\setminus\{i,j\}}$ with threshold $\alpha = 0.05$ suggests that the dependence structure of the distribution is represented by the following graph. 
    \begin{center}
    \begin{tikzpicture}[thick]
    \def\n{11} 
    \def\radius{2.5cm}

    \foreach \name [count=\i] in {PKA, PKC, PIP3, Mek, Raf, PIP2, Plcg, Jnk, P38, Akt, Erk}{
        
        % Calculate angle evenly divided by the number of nodes (6 nodes total)
        \pgfmathsetmacro{\angle}{(\i-1) * (360/\n)}
        
        % Draw the node using polar coordinates: (angle:radius)
        % (node_id) is set to lower-case index-based string or name reference
        \node[circle, inner sep=1pt, minimum width=1pt] (\name) at (\angle:\radius) {\name};
    }

    % Example: You can now reference the specific names directly to draw paths!
    \draw[-, thick] (PKA) -- (PIP3);
    \draw[-, thick] (PKA) -- (PIP2);
    \draw[-, thick] (PKA) -- (Jnk);

    \draw[-, thick] (PKC) -- (Akt);
    \draw[-, thick] (PKC) -- (Erk);
    \draw[-, thick] (PKC) -- (PIP3);
    \draw[-, thick] (PKC) -- (PIP2);
    \draw[-, thick] (PKC) -- (Mek);
    \draw[-, thick] (PKC) -- (Raf);

    \draw[-, thick] (P38) -- (Raf);
    \draw[-, thick] (P38) -- (Mek);
    \draw[-, thick] (P38) -- (PIP2);
    \draw[-, thick] (P38) -- (PIP3);

    \draw[-, thick] (Jnk) -- (Akt);
    \draw[-, thick] (Jnk) -- (Erk);
    \draw[-, thick] (Jnk) -- (PIP2);

    \draw[-, thick] (Raf) -- (PIP2);

    \draw[-, thick] (Mek) -- (PIP2);

    \draw[-, thick] (PIP2) -- (Akt);
    \draw[-, thick] (PIP2) -- (Erk);

    \draw[-, thick] (PIP3) -- (Erk);

    \end{tikzpicture}
    \end{center}
    While we have used the pairwise Markov property to infer this dependence structure from the data, Theorem~\ref{thm: undirected positive equality} implies that
    \[
    \mathcal M_{\textrm{Gauss}}(\mathcal G) = \mathcal M_P^{\textrm{Gauss}}(\mathcal G) = \mathcal M_L^{\textrm{Gauss}}(\mathcal G) = \mathcal M_G^{\textrm{Gauss}}(\mathcal G), 
    \]
    Where the latter sets denote the, respective, subsets of $\mathcal M_P(\mathcal G), \mathcal M_L(\mathcal G)$ and $\mathcal M_G(\mathcal G)$, containing only Gaussian distributions.
    In particular, our model implies we can read more complex conditional independence relations from the graph that should also hold for the data-generating distribution. 
    For example, the global Markov property implies that the joint abundance levels of molecules Akt and PKA are independent of those for Erk and Mek given measurements for PKC, JNK, PIP2 and PIP3.
\end{example}

Comparing the estimated graph in Example~\ref{ex: UGGMs} with the biologically accepted ground-truth network in Example~\ref{ex: sachs network}, we see that our model is perhaps not the best.  
This is reasonable, since it is unlikely that the Gaussianity assumption is a good match to the unknown distribution.  
At a deeper level, the dynamics of dependencies in a protein network are, more often than not, directed in a causal sense. 
The coming sections will provide us with a theoretical framework for parsing out models where directed dependence relations are a better fit to the data. 

This small example should, however, not be interpreted as suggesting a limited applicability of the family of undirected CI graphical models $\mathcal F_{\mathbb{UG}}$. 
In fact, Gaussian undirected graphical models are one of the most popular models in high-dimensional statistics, thanks to the success of methods such as the \emph{graphical lasso} \cite[Section 26.7.2]{murphy}. 

One final feature to note about the undirected CI graphical model $\mathcal M(\mathcal G)$ is that it is defined using only conditional independence constraints. 
In particular, the model is specified \emph{implicitly}, without using a parameterization.  
While lacking a parameterization, the CI graphical model $\mathcal M(\mathcal G)$ does allow for a factorization of its densities using a set of \emph{potential functions}. 
This factorization arises in the following way. 

For an undirected graph $\mathcal G = ([m], E)$, we define a \emph{clique} in $\mathcal G$ to be a subset $C\subseteq [m]$ satisfying $i - j\in E$ for all pairs $i,j\in C$. 
We say that a clique $C$ is \emph{maximal} in $\mathcal G$ if there is no clique $C'$ in $\mathcal G$ satisfying $C\subsetneq C'$. 
Let $\mathcal C(\mathcal G)$ denote the set of all maximal cliques in $\mathcal G$. 

\begin{example}
    [Cliques in an undirected graph]
    Let $\mathcal G$ be the following undirected graph. 
    \begin{center}
    \begin{tikzpicture}[thick, scale=0.6]
    \node[circle, draw, inner sep=1pt, minimum width=1pt] (1) at (0,0)  {$1$};
    \node[circle, draw, inner sep=1pt, minimum width=1pt] (2) at (2,0) {$2$};
    \node[circle, draw, inner sep=1pt, minimum width=1pt] (3) at (0,-2) {$3$};
    \node[circle, draw, inner sep=1pt, minimum width=1pt] (4) at (2,-2) {$4$};
    
    \draw[-, very thick] (1) -- (2);
    \draw[-, very thick] (1) -- (3);
    \draw[-, very thick] (2) -- (4);
    \draw[-, very thick] (3) -- (4);
    \draw[-, very thick] (3) -- (2);
    \end{tikzpicture}
    \end{center}
    Examples of cliques in $\mathcal G$ include
    \[
    C_1 = \{1\}, \qquad C_2 = \{2,4\}, \qquad C_3 = \{1,2,3\}, \qquad C_4 = \{2,3,4\}. 
    \]
    since $\mathcal G$ contains the subgraphs
    \begin{center}
    \begin{tikzpicture}[thick, scale=0.6]
    \node[circle, draw, inner sep=1pt, minimum width=1pt] (1) at (0+4,0-1)  {$1$};

    \node[circle, draw, inner sep=1pt, minimum width=1pt] (a2) at (2 + 6,0) {$2$};
    \node[circle, draw, inner sep=1pt, minimum width=1pt] (a4) at (2 + 6,-2) {$4$};
    
    \draw[-, very thick] (a2) -- (a4);

    \node[circle, draw, inner sep=1pt, minimum width=1pt] (b1) at (0 + 6 + 6,0)  {$1$};
    \node[circle, draw, inner sep=1pt, minimum width=1pt] (b2) at (2 + 6 + 6,0) {$2$};
    \node[circle, draw, inner sep=1pt, minimum width=1pt] (b3) at (0 + 6 + 6,-2) {$3$};
    
    \draw[-, very thick] (b1) -- (b2);
    \draw[-, very thick] (b1) -- (b3);
    \draw[-, very thick] (b3) -- (b2);

    \node[circle, draw, inner sep=1pt, minimum width=1pt] (c2) at (2 + 6 + 6 + 4,0) {$2$};
    \node[circle, draw, inner sep=1pt, minimum width=1pt] (c3) at (0 + 6 + 6 + 4,-2) {$3$};
    \node[circle, draw, inner sep=1pt, minimum width=1pt] (c4) at (2 + 6 + 6 + 4,-2) {$4$};
    
    \draw[-, very thick] (c2) -- (c4);
    \draw[-, very thick] (c3) -- (c4);
    \draw[-, very thick] (c3) -- (c2);
    \end{tikzpicture}
    \end{center}
    in which each pair of nodes is connected by an edge. 
    Note that $\mathcal G$ contains exactly two maximal cliques, which are $C_3$ and $C_4$. 
\end{example}

\begin{definition}
    \label{def: potential function}
    Let $m$ be a positive integer and $C\subseteq [m]$. 
    A function $\phi_C: \mathcal X_C\to \mathbb R$ is called a \emph{potential function} if $\phi_C(x_c)\geq 0$ for all $x_C\in \mathcal X_C$. 
\end{definition}

The idea behind  potential functions is to represent local contributions of the nodes in a clique $C$ -- which are mutually dependent upon each other -- to the density $f_X(x)$ of the distribution $\mathbb P$ on $X = (X_i)_{i\in[m]}$. 
Doing so can afford us several statistical advantages, especially when $\mathcal G$ is relatively sparse.  
These include, for example, improved efficiency in estimating the density $f_X(x)$ via local estimations of the potential functions $\phi_C(x_C)$. 
For more on this in the case of discrete variables $X_i$, see \cite[Chapters 9,10]{KF}. 
We now make a formal definition capturing when the density $f_X(x)$ can be reduced to a considering less cumbersome potential functions on cliques. 

\begin{definition}
    \label{def: undirected factorization}
    Let $\mathcal G = ([m],E)$ be an undirected graph. 
    A distribution $\mathbb P\in \mathbb D_X$ with density $f_X(x)$ \emph{factorizes} according to $\mathcal G$ if there exist potential functions $\phi_C(x_C)$ for all $C\in \mathcal C(\mathcal G)$ such that 
    \[
    f_X(x) = \frac{1}{Z}\prod_{C\in \mathcal C(\mathcal G)}\phi_C(x_C) \qquad \textrm{ for all } x\in \mathcal X, 
    \]
    where 
    \[
    Z = \int_{\mathcal X}\prod_{C\in \mathcal C(\mathcal G)}\phi_C(x_C)dx.
    \]
    and the integral is taken with respect to the appropriate base measure for $\mathbb P$. 
\end{definition}

\begin{remark}
    \label{rem: nonmaximal cliques}
    Let $\mathcal C'(\mathcal G)$ denote a subset of (possibly nonmaximal) cliques in $\mathcal G$ with the property that for every edge $i - j$ of $\mathcal G$ there is at least one clique $C \in \mathcal C'$ such that $i,j\in C$. 
    Note that a distribution $\mathbb P \in \mathbb D_X$ with density $f_X(x)$ factorizes according to $\mathcal G$ if and only if there exist potential functions $\phi'_C(x_C)$ for all $C\in \mathcal C'(\mathcal G)$ such that 
    \[
    f_X(x) = \frac{1}{Z'}\prod_{C\in \mathcal C'(\mathcal G)}\phi_C(x_C) \quad \textrm{ for all } x\in \mathcal X,
    \textrm{where}
    \quad
    Z' = \int_{\mathcal X}\prod_{C\in \mathcal C'(\mathcal G)}\phi_C(x_C)dx.
    \]
    In other words, Definition~\ref{def: undirected factorization} can be equivalently stated using all cliques in $\mathcal G$, but for most purposes it is more efficient to use only the maximal cliques. 
\end{remark}

Using the very general framework of potential functions, it is possible to provide a factorization formula for the density of any distribution in any undirected CI graphical model $\mathcal M(\mathcal G)$. 
Namely, the following famous theorem states that all distributions in our undirected CI graphical model for $\mathcal G$ can be factorized using potential functions. 

\begin{theorem}
    [Hammersley-Clifford Theorem]
    \label{thm: hammersley-clifford}
    Let $\mathcal G = ([m], E)$ be an undirected graph, and let $\mathbb P\in \mathbb D_X$.  Then $\mathbb P$ factorizes according to $\mathcal G$ if and only if $P \in \mathcal M(\mathcal G)$.
\end{theorem}

A proof of Theorem~\ref{thm: hammersley-clifford} can be found in the book \emph{Graphical Models} by Lauritzen \cite{Lauritzen1996}. 

\begin{example}[Zero field Ising model]
    \label{ex: Ising model parameterization}
    Perhaps the most fundamental example of an undirected CI graphical model is the Ising model from Example~\ref{ex: undirected grid}.  
    In its original, physical context, we attach binary random variables $X_i$ to each node $i$ in the grid graph where $X_i$ takes values in the set $\{+1,-1\}$ representing the possibilities of positive or negative atomic spins. 
    For the grid graph $\mathcal G = (V, E)$ in Example~\ref{ex: undirected grid}, it is easy to see that the maximal cliques are given by the edges of $\mathcal G$:
    \[
    \mathcal C(\mathcal G) = \{\{i,j\} : i - j\in E\}. 
    \]
    The (zero field) Ising model then associates a coupling constant $J_{ij}\in \mathbb R$ to each maximal clique $\{i,j\}$, creating a potential function $\phi_{\{i,j\}}(x_i,x_j) = e^{-J_{ij}x_ix_j}$. 
    The behavior of the atomic system is described via the Boltzmann distribution
    \[
    f_X(x) = \frac{1}{Z}\prod_{\{i,j\}\in\mathcal C(\mathcal G)}\phi_{\{i,j\}}(x_i,x_j),
    \]
    where $Z = \sum_{x\in\{+1,-1\}^V}\prod_{\{i,j\}\in\mathcal C(\mathcal G)}\phi_{\{i,j\}}(x_i,x_j)$ is the \emph{partition function}.
    Since this distribution factorizes according to $\mathcal G$, Theorem~\ref{thm: hammersley-clifford} implies that CI relations holding among the variables can be easily read from the combinatorial structure of the graph $\mathcal G$.  
\end{example}

The Hammersley-Clifford Theorem admits several generalizations that apply to distributions that may not belong to the family $\mathbb D_X$. 
However, the statement in its classical form (as given in Theorem~\ref{thm: hammersley-clifford}), is often sufficient to cover most modern multivariate data scenarios.

\subsection{Conditional independence DAG models}
\label{subsubsec: DAG models}

One of the most popular types of CI graphical models are built upon \emph{directed acyclic graphs} (DAGs). 
This is because the directed edges (unlike undirected) admit an intuitive \emph{causal} interpretation, as discussed in Section~\ref{subsec: graphs}, and the acyclicity tends to allow for simple, recursive proofs of valuable theorems.  

Causal models are of increasing importance in data-driven society, largely due to access to new data-scenarios in which one can ethically perform Randomized Controlled Trials (RCTs). 
Historically, RCTs were difficult to perform since they require a test and control group, but treating one group with a potentially life-altering treatment while not treating the other group raises serious ethical questions.  
Now, as described in the introduction, there is a wealth of opportunities for performing low-risk RCTs within the internet-based market place since, for example, there are generally no major ethical qualms in showing one customer one version of your app and a different customer an alternative (test) version. 
Hence, the major technological developments of the early 21st century have lead to a renaissance for research in causal modeling. 
At the heart of these activities are the conditional independence DAG models, which we will now introduce. 

Similar to undirected CI graphical models, we will define CI models for DAGs by way of the standard methods for specifying a Markov property outlined in Section~\ref{subsec: markov properties}. 
In this case, we would like to build a causal intuition into how we consider the CI relations associated to the absence of directed edges, the direction of the edges connecting neighbors, and our global notion of (directed) connectedness within the graph. 
By doing this appropriately, the CI relations encoded by the DAG can be viewed as encoding rudimentary (probabilistic) properties that would naturally hold in a causal model. 
In this way, our CI DAG models will become a (probabilistic) basis for causal features in a data-generating process. 

\begin{definition}[Markov properties for DAGs]
    \label{def: DAG MPs}
    Let $\mathcal G = ([m], E)$ be a DAG. 
    \begin{enumerate}
        \item The \emph{directed pairwise Markov property} for $\mathcal G$  is the set of CI relations
        \[
        \CI_P(\mathcal G) = \{X_i \independent X_j \mid X_{\nd(i)\setminus\{j\}}: j\in \nd(i) \textrm{ and } j\to i, i\to j\notin E\}. 
        \]
        \item The \emph{directed local Markov property} for $\mathcal G$ is the set of CI relations
        \[
        \CI_L(\mathcal G) = \{X_i \independent X_{\nd(i)\setminus \pa(i)} \mid X_{\pa(i)}: i \in [m]\}. 
        \]
        \item The \emph{directed global Markov property} for $\mathcal G$ is the set of CI relations
        \[
        \CI_G(\mathcal G) = \{ X_A \independent X_B \mid X_C : A \textrm{ and } B \textrm{ are d-separated given } C \textrm{ in } \mathcal G\}. 
        \]
    \end{enumerate}
\end{definition}

Note that the CI relations in the directed local Markov property are a very natural probabilistic consequence of the DAG $\mathcal G$ being a causal model.  
The parents $j\in \pa(i)$ are the nodes in $\mathcal G$ for which there is an edge $j\to i$ in $\mathcal G$.  
Thinking causally, we would interpret the parents of $i$ as the \emph{direct effects} of $i$.  
Similarly, the nondescendants of $i$, $\nd(i)$, are all the nodes $j$ that are \emph{not} reachable from $i$ via a directed path $i\to \cdots \to j$. In other words, the nondescendants of $i$ are all nodes that are not indirectly effected by $i$.  
So the statement $X_i \independent X_{\nd(i)\setminus \pa(i)} \mid X_{\pa(i)}$ can be viewed as capturing a probabilistic consequence of the causal structure $\mathcal G$: \emph{When the direct effects of $X_i$ are measured, the distribution of $X_i$ is independent of all other variables that are not directly or indirectly effected by $X_i$.}

Similarly, the statement $X_i \independent X_j \mid X_{\nd(i)\setminus\{j\}}$ for $j\in \nd(i)$ in the pairwise local Markov property captures a different probabilistic consequence of the causal structure $\mathcal G$: \emph{If $X_j$ and $X_i$ are not related by a direct effect then their distributions should be independent when we are given measurements of all variables not effected by $X_i$.}

\begin{remark}[Probabilistic versus causal modeling]
We emphasis that a distribution $\mathbb P\in \mathbb D_X$ satisfying the Markov properties in Definition~\ref{def: DAG MPs} for a DAG $\mathcal G$ is \emph{not} a causal model.  Instead, it is a distribution that satisfies intuitive probabilistic consequences (e.g. certain CI relations) that would naturally be satisfied if the data-generating process has underlying causal dependence structure given exactly by the DAG $\mathcal G$. 
This includes the assumption that every variable that could be effecting any variable in the system is represented by a node in $\mathcal G$, an assumption referred to as \emph{causal sufficiency} \cite{pearl2009}. 
Typically, more work, or additional assumptions, are required to obtain a causal model based on $\mathcal G$. 
\end{remark}

The global Markov property can similarly be viewed as a set of CI relations that are probabilistic consequences of the data-generating process having underlying causal structure given by $\mathcal G$. 
Comparing Definition~\ref{def: DAG MPs} with the Markov properties for undirected graphs in Definition~\ref{def: undirected MPs}, we see that the global Markov property for DAGs substitutes the notion of separation in undirected graphs with a notion of \emph{directed} separation, or d-separation for short. 

\begin{definition}[d-separation in DAGs]
    \label{def: d-separation}
    Let $\mathcal G = ([m], E)$ be a DAG, $i,j\in [m]$ and $C\subseteq [m]$. 
    We say that $i$ and $j$ are \emph{d-connected} given $C$ in $\mathcal G$ if there exists a path $\rho = (v_k)_{k\in [s]}$ with $v_1 = i$ and $v_{s} = j$ such that
    \begin{enumerate}
        \item if $\rho$ contains the edges $v_{k-1} \to v_k \to v_{k+1}$, $v_{k-1} \leftarrow v_k \to v_{k+1}$ or $v_{k-1} \leftarrow v_k \leftarrow v_{k+1}$ for some $k\in[s]$ then $v_k\notin C$, and
        \item if $\rho$ contains $v_{k-1} \to v_k \leftarrow v_{k+1}$ for some $k\in [s]$ then $v_k\in an(C)$.
    \end{enumerate}
    Let $A,B,C\subseteq[m]$ be pairwise disjoint with $A,B\neq \emptyset$. 
    We say that $A$ and $B$ are \emph{d-separated} given $C$ in $\mathcal G$, denoted $A \perp_{\mathcal G} B \mid C$, if there is no $i \in A$ and $j\in B$ that are d-connected given $C$. 
    Otherwise, we say $A$ and $B$ are \emph{d-connected} given $C$ in $\mathcal G$, and denote it by $A \not\perp_{\mathcal G} B \mid C$. 
\end{definition}

\begin{example}
    \label{ex: d-sep genes}
    Let $\mathcal G = (V, E)$ be the following DAG on node set $V = \{1,2,3,4\}$. 
    \begin{center}
    \begin{tikzpicture}[thick, scale=0.6]
    \node[circle, draw, inner sep=1pt, minimum width=1pt] (1) at (0,0)  {$1$};
    \node[circle, draw, inner sep=1pt, minimum width=1pt] (2) at (2,0) {$2$};
    \node[circle, draw, inner sep=1pt, minimum width=1pt] (3) at (1,-1.5) {$3$};
    \node[circle, draw, inner sep=1pt, minimum width=1pt] (4) at (1,-3) {$4$};
    
    \draw[->, very thick] (1) -- (3);
    \draw[->, very thick] (2) -- (3);
    \draw[->, very thick] (3) -- (4);
    \end{tikzpicture} 
    \end{center}
    We imagine a model in which a variable $X_i$ records the genetic information of individual $i$, with $i=1$ being a (biological) mother, $i=2$ a (biological) father, $i=3$ their (biological) child and $i=4$ their (biological) grandchild. 
    The logic of d-connection/d-separation is that it captures how information could flow through this network via conditioning on available knowledge.  
    For example, provided with only genetic information $X_1$ about the mother, we should not be able to deduce any genetic information $X_2$ about the father.  This makes sense in the graph because the only path connecting the parents $1$ and $2$ flow downward, and intuitively we should not be able to pass information backwards along directed edges. 
    Indeed, $1$ and $2$ are d-separated in $\mathcal G$ when no additional information is provided (i.e. $C = \emptyset$). 

    On the other hand, given genetic information about the child in the form of an observation $X_3 = x_3$, when measuring genetic information about the mother $X_1$ we can compare our observations with the information in $X_3 = x_3$. 
    Features in $X_3 = x_3$ that are not apparent from the mother likely are coming from the father $X_2$.  
    Hence, $X_1$ and $X_2$ become dependent given observations from $X_3$. 
    This is combinatorially captured by the fact that $1$ and $2$ are d-connected given $C = \{3\}$ in $\mathcal G$. 

    Similarly, genetic information about the mother $X_1$ will naturally inform us about the genetics of the grandchild $X_4$, and $1$ and $4$ are d-connected given $C = \emptyset$. 
    However, provided with the genetic information of the child $X_3= x_3$, we can only learn as much about the grandchild $X_4$ from the mother $X_1$ that we already know from our observation $X_3 = x_3$ from the intermediary child. 
    Correspondingly, $1$ and $4$ are d-separated given $C = \{3\}$ in $\mathcal G$. 
\end{example}

\begin{example}
    \label{ex: d-sep 2}
    Let $\mathcal G = (V,E)$ be the following DAG on vertex set $V = \{1,2,3,4,5,6\}$. 
    \begin{center}
    \begin{tikzpicture}[thick, scale=0.6]
    \node[circle, draw, inner sep=1pt, minimum width=1pt] (1) at (0,0)  {$1$};
    \node[circle, draw, inner sep=1pt, minimum width=1pt] (2) at (2,0) {$2$};
    \node[circle, draw, inner sep=1pt, minimum width=1pt] (3) at (0,-2) {$3$};
    \node[circle, draw, inner sep=1pt, minimum width=1pt] (4) at (2,-2) {$4$};
    \node[circle, draw, inner sep=1pt, minimum width=1pt] (5) at (4,-2) {$5$};
    \node[circle, draw, inner sep=1pt, minimum width=1pt] (6) at (3,-3.5) {$6$};
    
    \draw[->, very thick] (1) -- (2);
    \draw[->, very thick] (1) -- (3);
    \draw[->, very thick] (2) -- (4);
    \draw[->, very thick] (3) -- (4);
    \draw[->, very thick] (4) -- (5);
    \draw[->, very thick] (4) -- (6);
    \draw[->, very thick] (5) -- (6);
    \end{tikzpicture}
    \end{center}
    The following d-separation holds in $\mathcal G$: 
    $
    3 \perp_{\mathcal G} 2 \mid 1,
    $
    as do the following d-connections
    \[
    3 \not\perp_{\mathcal G} 2, \qquad 3 \not\perp_{\mathcal G} 2 \mid \{1,4\}, \qquad 3 \not\perp_{\mathcal G} 2 \mid \{1,6\}.
    \]
    The final d-connection here follows from Definition~\ref{def: d-separation}~(2) since $4$ is an ancestor of $6$. 
    Similarly, we have
    \[
    \{1,2\} \perp_{\mathcal G} \{5,6\} \mid \{3,4\} \qquad \textrm{and} \qquad \{1,2\} \not\perp_{\mathcal G} \{5,6\} \mid 3. 
    \]    
\end{example}

Following the principle of Markov properties outlined in Section~\ref{subsec: markov properties}, we can define a CI graphical model for every Markov property in Definition~\ref{def: DAG MPs}. 

\begin{definition}
    \label{def: DAG GMs}
    Let $\mathcal G = ([m], E)$ be a DAG. 
    \begin{enumerate}
        \item The \emph{pairwise CI DAG model} for $\mathcal G$ is
        \[
        \mathcal M_P(\mathcal G) = \{ \mathbb P \in \mathbb D_X : \mathbb P \textrm{ satisfies all CI relations in } \CI_P(\mathcal G)\}. 
        \]
        \item The \emph{local CI DAG model} for $\mathcal G$ is
        \[
        \mathcal M_L(\mathcal G) = \{ \mathbb P \in \mathbb D_X : \mathbb P \textrm{ satisfies all CI relations in } \CI_L(\mathcal G)\}. 
        \]
        \item The \emph{global CI DAG model} for $\mathcal G$ is
        \[
        \mathcal M_G(\mathcal G) = \{ \mathbb P \in \mathbb D_X : \mathbb P \textrm{ satisfies all CI relations in } \CI_G(\mathcal G)\}. 
        \]
    \end{enumerate}
\end{definition}

Similar to undirected CI graphical models, we have a valid notion of potential functions for DAGs.  
In this case, we gain advantage from the recursive nature of a directed acyclic graph, and the potential functions are simply conditional distributions. 

\begin{definition}
    \label{def: DAG factorization}
    A distribution $\mathbb P \in \mathbb D_X$ with density $f_X(x)$ \emph{factorizes} according to the DAG $\mathcal G = ([m], E)$ if 
    \[
    f_X(x) = \prod_{i=1}^m f_{X_i}(x_i \mid x_{\pa(i)}) \qquad \textrm{for all } x\in \mathcal X. 
    \]
    This yields a fourth CI DAG model 
    \[
    \mathcal M_F(\mathcal G) = \{\mathbb P \in \mathbb D_X : \mathbb P \textrm{ factorizes according to } \mathcal G\}. 
    \]
\end{definition}

Much like we did for undirected CI graphical models with Theorem~\ref{thm: undirected positive equality} and the Hammersley-Clifford Theorem (Theorem~\ref{thm: hammersley-clifford}), we will prove the following equality of models
\begin{equation}
\label{eqn: directed equalities}
\mathcal M_L(\mathcal G) = \mathcal M_G(\mathcal G) = \mathcal M_F(\mathcal G). 
\end{equation}
Note that, unlike undirected CI graphical models, we do not include the pairwise model $\mathcal M(\mathcal G)$ in the above chain of equalities.  
In fact, for directed acyclic graphs, we have the following observation. 

\begin{proposition}
    \label{prop: pairwise relation DAGs}
    Let $\mathcal G = ([m], E)$ be a DAG. 
    Then
    \[
    \mathcal M_L(\mathcal G) \subsetneq \mathcal M_P(\mathcal G). 
    \]
\end{proposition}

\begin{proof}
    We first prove the inclusion, and then show that the inclusion is strict by way of an example. 
    Let $\mathbb P\in \mathbb \mathcal M_L(\mathcal G)$, and suppose that $j\to i \notin E$ for some $j \in \nd(i)$.  
    In particular, $j\notin \pa(i)$, so $j\in \nd(i)\setminus \pa(i)$.  
    Since $X_i \independent X_{\nd(i)\setminus \pa(i)} \mid X_{\pa(i)}$ holds in $\mathbb P$, then writing $\nd(i) \setminus \pa(i) = A \cup B$ where $A = \{j\}$ and $B = \nd(i) \setminus (\pa(i)\cup \{j\})$ we obtain that $X_i \independent X_A \mid X_{\pa(i) \cup B}$ holds in $\mathbb P$ by the weak union property in Proposition~\ref{prop: CI axioms}. This CI relation simplifies notationally to $X_i \independent X_j \mid X_{\nd(i) \setminus \{j\}}$. 
    Since $i$ and $j$ were arbitrary, it follows that $\mathbb P \in \mathcal M_P(\mathcal G)$. 

    A classic example of a distribution $\mathbb P \in \mathcal M_P(\mathcal G)$ for a DAG $\mathcal G$ that does not belong to $\mathcal M_L(\mathcal G)$ is the following (see \cite[Example 3.26]{Lauritzen1996}):
    Suppose that $\mathcal G$ is the following DAG: 
    \begin{center}
    \begin{tikzpicture}[thick, scale=0.6]
    \node[circle, draw, inner sep=1pt, minimum width=1pt] (1) at (0,0)  {$1$};
    \node[circle, draw, inner sep=1pt, minimum width=1pt] (2) at (2,0) {$2$};
    \node[circle, draw, inner sep=1pt, minimum width=1pt] (4) at (1,-1.5) {$4$};
    \node[circle, draw, inner sep=1pt, minimum width=1pt] (3) at (1,-3) {$3$};
    
    \draw[->, very thick] (1) -- (2);
    \draw[->, very thick] (1) -- (4);
    \draw[->, very thick] (2) -- (4);
    \draw[->, very thick] (4) -- (3);
    \end{tikzpicture} 
    \end{center}
    Consider the discrete distribution on $X = (X_i)_{i\in \{1,2,3,4\}}$ where $X_i$ takes values in $\mathcal X_i = \{0,1\}$ for all $i$, $X_4$ is independent of $X_3$, $X_1 = X_2 = X_3$ and 
    \[
    f_{X_3}(1) = f_{X_3}(0) = f_{X_4}(1) = f_{X_4}(0) =  0.5. 
    \]
    The distribution $\mathbb P$ satisfies the pairwise Markov property with respect to $\mathcal G$ since 
    \[
    f_{X_3 \mid X_4, X_2, X_1}(x_3 \mid x_4, x_3, x_3) = f_{X_3 \mid X_4, X_2}(x_3 \mid x_4, x_3) = f_{X_3 \mid X_4, X_1}(x_3 \mid x_4, x_3). 
    \]
    However, it does not satisfy the local Markov property with respect to $\mathcal G$ since the equality of $X_1 = X_2 = X_3$ excludes the possibility that $X_3 \independent X_{\{1,2\}} \mid X_4$. 
\end{proof}

To prove the equalities in~\eqref{eqn: directed equalities} it is easiest to introduce a fourth Markov property that is combinatorially equivalent to the directed global Markov property. 
This Markov property will allow us to rely on our results for undirected CI graphical models. 
To define this additional Markov property we introduce an undirected graph constructed from a DAG. 

\begin{definition}
    \label{def: moral graph}
    Let $\mathcal G = ([m], E)$ be a DAG. 
    The \emph{moral graph} for $\mathcal G$ is the undirected graph $G^{\textrm{moral}} = ([m], E^{\textrm{moral}})$ where
    \[
    E^{\textrm{moral}} = \{ i - j : i \to j \in E \, \,  \textrm{ or } \, \,  i \leftarrow j\in E\} \cup \{ i - j : i \to k, j \to k \in E, \, k\in [m]\}. 
    \]
\end{definition}

\begin{example}
    The moral graph of the DAG from Example~\ref{ex: d-sep 2} is the following 
    \begin{center}
    \begin{tikzpicture}[thick, scale=0.6]
    \node[circle, draw, inner sep=1pt, minimum width=1pt] (1) at (0,0)  {$1$};
    \node[circle, draw, inner sep=1pt, minimum width=1pt] (2) at (2,0) {$2$};
    \node[circle, draw, inner sep=1pt, minimum width=1pt] (3) at (0,-2) {$3$};
    \node[circle, draw, inner sep=1pt, minimum width=1pt] (4) at (2,-2) {$4$};
    \node[circle, draw, inner sep=1pt, minimum width=1pt] (5) at (4,-2) {$5$};
    \node[circle, draw, inner sep=1pt, minimum width=1pt] (6) at (3,-3.5) {$6$};
    
    \draw[-, very thick] (1) -- (2);
    \draw[-, very thick] (1) -- (3);
    \draw[-, very thick] (2) -- (4);
    \draw[-, very thick] (3) -- (4);
    \draw[-, very thick] (4) -- (5);
    \draw[-, very thick] (4) -- (6);
    \draw[-, very thick] (5) -- (6);

    \draw[-, very thick] (2) -- (3);
    \end{tikzpicture}
    \end{center}
\end{example}

Our fourth Markov property for directed acyclic graphs utilizes the moral graph construction, applied to induced subgraphs on ancestral sets. 

\begin{definition}
    \label{def: moral MP}
    Let $\mathcal G = ([m],E)$ be a DAG.  
    The \emph{moral Markov property for } $\mathcal G$ is the set of CI relations
    \[
    \CI_M(\mathcal G) = \{X_A \independent X_B \mid X_C :  \textrm{$A$ and $B$ are separated given $C$ in $\mathcal G_{\an(A\cup B\cup C)}^\textrm{moral}$}\}.
    \]
\end{definition}

The moral Markov property is essentially a combinatorial rephrasing of the global Markov property, as the following lemma shows. 

\begin{lemma}
    \label{lem: moral MP}
    Let $\mathcal G = ([m],E)$ be a DAG and $A,B,C\subseteq[m]$. 
    Then $A$ and $B$ are d-separated given $C$ in $\mathcal G$ if and only if $A$ and $B$ are separated given $C$ in $\mathcal G_{\an(A\cup B\cup C)}^\textrm{moral}$.
\end{lemma}

\begin{proof}
    Suppose first that $A$ and $B$ are d-connected given $C$ in $\mathcal G$.  
    Then there exists a path $\rho = (v_i)_{i\in [s]}$ with $v_1\in A$ and $v_s\in B$ fulfilling the conditions (1) and (2) in Definition~\ref{def: d-separation}. 
    If the path $\rho$ is directed from either $v_1$ to $v_k$ or $v_k$ to $v_1$ or if there is a unique $k\in [s]$ such that $v_{k-1} \leftarrow v_k \rightarrow v_{k+1}$ occurs, then the same path (with all edges undirected) will exist in $\mathcal G_{\an(A\cup B\cup C)}^\textrm{moral}$.  Moreover, these three cases can only happen if $v_k\notin C$ for all $k\in [s]$, since $\rho$ is d-connecting.  
    Hence, in this case, $A$ and $B$ will be connected given $C$ in $\mathcal G_{\an(A\cup B\cup C)}^\textrm{moral}$. 
    Otherwise, there exist some $v_k$ along this path such that $v_{k-1} \to v_k \leftarrow v_{k-1}$ with $v_k\in \an(C)$.
    Note first that each node in the path $\rho$ belongs to $\an(A\cup B \cup C)$ since all other configurations along three-node subpaths must be one of $v_{k-1} \to v_k \to v_{k+1}$, $v_{k-1} \leftarrow v_k \to v_{k+1}$ or $v_{k-1} \leftarrow v_k \leftarrow v_{k+1}$. 
    Hence, the path $\rho$ with all edges undirected exists in $\mathcal G_{\an(A\cup B\cup C)}^\textrm{moral}$.  
    It is, however, possible that $\rho$ contains nodes in $C$, since $C\subseteq \an(C)$.  
    However, since $\rho$ is d-connecting these nodes must be some $v_k\in C$ for a configuration $v_{k-1} \to v_k \leftarrow v_{k-1}$ along $\rho$ in $\mathcal G$. 
    Hence, in the moral graph $\mathcal G_{\an(A\cup B\cup C)}^\textrm{moral}$, there will also be an edge $v_{k-1} - v_{k+1}$.  
    The existence of these edges then gives a path in $\mathcal G_{\an(A\cup B\cup C)}^\textrm{moral}$ from $v_1$ to $v_s$ that does not use nodes in $C$, showing that $A$ and $B$ are connected in $\mathcal G$ given $C$. 

    Suppose now that $A$ and $B$ are connected given $C$ in $\mathcal G_{\an(A\cup B\cup C)}^\textrm{moral}$. 
    Then there exists a path $\rho = (v_k)_{k\in[s]}$ in $\mathcal G_{\an(A\cup B\cup C)}^\textrm{moral}$ with $v_1\in A$, $v_s\in B$ and $v_k\notin C$ for all $k\in [s]$. 
    For each edge $v_k - v_{k+1}$ along $\rho$, if $v_k$ and $v_{k+1}$ are also adjacent in $\mathcal G_{\an(A\cup B \cup C)}$, replace $v_k - v_{k+1}$ with its directed version from $\mathcal G_{\an(A\cup B \cup C)}$. 
    We call this partially directed path $\tilde\rho$. 
    Note that an undirected edge $v_{k} - v_{k+1}$ on $\tilde\rho$ exists if and only if $v_k$ and $v_{k+1}$ are not adjacent in $\mathcal G_{\an(A\cup B \cup C)}$ and there exists $v_k'$ satisfying $v_k \to v_k' \leftarrow v_{k+1}$  in $\mathcal G_{\an(A\cup B \cup C)}$. 
    We update $\tilde\rho$ to insert each of these nodes $v_k'$. 
    Since each $v_k'\in \an(A\cup B\cup C)$, we know that each $v_k'$ has a descendant in either $A$, $B$ or $C$.  
    Extend the (now fully directed path) $\tilde \rho$ to a directed subgraph of $\mathcal G_{\an(A\cup B \cup C)}$ that includes the directed paths from each $v_k'$ to its descendants in $A$ or $B$ for every $v_k'$ having only descendants in $A$ or $B$.  
    This subgraph is also a subgraph of $\mathcal G$, and it clearly contains a $d$-connecting path between $v_1$ and $v_s$ given $C$.  
    Hence, $A$ and $B$ are d-connected given $C$ in $\mathcal G$. 
\end{proof}

With the help of Lemma~\ref{lem: moral MP}, we may now prove the following theorem, which is the DAG equivalent of Theorem~\ref{thm: undirected positive equality} and Theorem~\ref{thm: hammersley-clifford}. 

\begin{theorem}
    \label{thm: factorization DAG}
    Let $\mathcal G = ([m], E)$ be a DAG. Then 
    \[
    \mathcal M_L(\mathcal G) = \mathcal M_G(\mathcal G) = \mathcal M_F(\mathcal G). 
    \]
\end{theorem}

\begin{proof}
    Let $\mathbb P \in \mathcal M_F(\mathcal G)$ have density $f_X(x)$.  Then 
    \begin{equation}
        \label{eqn: dist fact}
        f_X(x) = \prod_{i=1}^mf_{X_i\mid X_{\pa(i)}(x_i \mid x_{\pa(i)})}. 
    \end{equation}
    Suppose that $A$ and $B$ are d-separated given $C$ in $\mathcal G$. 
    Marginalizing away nodes not in $\an(A\cup B \cup C)$ in~\eqref{eqn: dist fact} yields a factorization of the marginal distribution $\mathbb P_{\an(A\cup B \cup C)}$ with density
    \begin{equation}
        \label{eqn: an dist fact}
        f_{X_{\an(A\cup B \cup C)}}(x_{\an(A\cup B \cup C)}) = \prod_{i \in \an(A\cup B \cup C)}f_{X_i \mid X_{\pa(i)}}(x_i \mid x_{\pa(i)}). 
    \end{equation}
    Note, by Definition~\ref{def: potential function}, that the conditional densities $f_{X_i \mid X_{\pa(i)}}(x_i \mid x_{\pa(i)})$ are potential functions on the marginal sample space $\mathcal X_{\{i\}\cup \pa(i)}$.  
    Moreover, by Definition~\ref{def: moral graph}, each set $\{i\}\cup \pa(i)$ is a clique in $\mathcal G_{\an(A\cup B \cup C)}^{\textrm{moral}}$. 
    Hence, by Remark~\ref{rem: nonmaximal cliques} and the Hammersley-Clifford Theorem, we have that $\mathbb P_{\an(A\cup B \cup C)}$ belongs to the undirected CI graphical model $\mathcal M(\mathcal G_{\an(A\cup B \cup C)}^{\textrm{moral}})$. 
    Since $A$ and $B$ are d-separated given $C$ in $\mathcal G$, it follows from Lemma~\ref{lem: moral MP} that $A$ and $B$ are separated given $C$ in $\mathcal G_{\an(A\cup B \cup C)}^{\textrm{moral}}$.  
    Hence, by the definition of the undirected global Markov property (see Definition~\ref{def: undirected MPs}~(3)), it follows that $X_A \independent X_B \mid X_C$ holds in $\mathbb P_{\an(A\cup B \cup C)}$.
    This directly implies that $X_A \independent X_B \mid X_C$ holds in $\mathbb P$, so we conclude that $\mathbb P\in\mathcal M(\mathcal G)$. 
    Thus, $\mathcal M_F(\mathcal G)\subseteq \mathcal M_G(\mathcal G)$. 

    Suppose now that $\mathbb P\in \mathcal M_G(\mathcal G)$, and observe that $A = \{i\}$ and $B = \nd(i)\setminus \pa(i)$ are d-separated given $C = \pa(i)$ in $\mathcal G$. 
    Hence, $X_i \independent X_{\nd(i)\setminus \pa(i)} \mid X_{\pa(i)}$ holds in $\mathbb P$. 
    Therefore, $\mathcal M_G(\mathcal G) \subseteq \mathcal M_L(\mathcal G)$. 

    To complete the proof, it remains to show that $\mathcal M_L(\mathcal G)\subseteq \mathcal M_F(\mathcal G)$. 
    For this case, we fix a \emph{topogical ordering} of $\mathcal G$, which is an ordering $\pi = (\pi_1,\ldots, \pi_m)$ of the nodes $[m]$ of $\mathcal G$ with the property that there is no directed path from $\pi_i$ to $\pi_j$ whenever $i > j$. 
    We let $\textrm{pred}(\pi_i) = \{\pi_1,\ldots, \pi_{i-1}\}$, and note that $\textrm{pred}(\pi_i)\subseteq \nd(\pi_i)$ for all $\pi_i\in[m]$. 
    Since $\mathbb P\in \mathcal M_L(\mathcal G)$, we have that $X_{\pi_i} \independent X_{\nd(\pi_i)\setminus\pa(\pi_i)} \mid X_{\pa(\pi_i)}$.
    So by the decomposition property in Proposition~\ref{prop: CI axioms}, we know that $X_{\pi_i} \independent X_{\textrm{pred}(\pi_i)\setminus\pa(\pi_i)} \mid X_{\pa(\pi_i)}$ also holds in $\mathbb P$. 
    Applying the chain rule from basic probability, we conclude that
    \begin{equation*}
        \begin{split}
            f_X(x)
            &= \prod_{i=1}^mf_{X_{\pi_i} \mid X_{\{{\pi_1}, \ldots, {\pi_{i-1}}\}}}(x_{\pi_i} \mid x_{\pi_1}, \ldots, x_{\pi_{i-1}}),\\
            &= \prod_{i=1}^mf_{X_{\pi_i}\mid X_{\pa(\pi_i)}}(x_{\pi_i}\mid x_{\pa(\pi_i)}),\\
            &= \prod_{i=1}^mf_{X_{i}\mid X_{\pa(i)}}(x_{i}\mid x_{\pa(i)}).
        \end{split}
    \end{equation*}
    Hence, $\mathbb P\in \mathcal M_F(\mathcal G)$, which completes the proof. 
\end{proof}

Theorem~\ref{thm: factorization DAG} is sometimes referred to as the \emph{factorization theorem} for DAGs. 
One good thing about this theorem is that it allows us to define \emph{the} CI DAG model for a DAG $\mathcal G$. 

\begin{definition}
    \label{def: DAG model}
    Let $\mathcal G = ([m],E)$ be a DAG. The \emph{conditional independence DAG model} for $\mathcal G$, denoted $\mathcal M(\mathcal G)$ is 
    \[
    \mathcal M(\mathcal G) = \mathcal M_L(\mathcal G) = \mathcal M_G(\mathcal G) = \mathcal M_F(\mathcal G). 
    \]
    The \emph{CI DAG model family} is 
    \[
    \mathcal F_{\mathbb{DAG}} = \{\mathcal M(\mathcal G) : \mathcal G\in \mathbb{DAG}\}. 
    \]
\end{definition}

Note that, by Proposition~\ref{prop: pairwise relation DAGs}, we have that $\mathcal M(\mathcal G)\subsetneq \mathcal M_P(\mathcal G)$.  
In particular, when working with a distribution $\mathbb P \in \mathcal M(\mathcal G)$ we can always assume we have access to the pairwise CI relations defining the pairwise Markov property in Definition~\ref{def: DAG MPs}~(1).  
However, care must be taken when one defines a distribution $\mathbb P$ according to the directed pairwise Markov property constraints, since it is possible that such a distribution does not belong to the CI DAG model $\mathcal M(\mathcal G)$.  
In these instances, one needs to verify that the distribution does indeed belong to $\mathcal M(\mathcal G)$ before having access to CI relations specified by d-separation or the classic DAG factorization in Definition~\ref{def: DAG factorization}. 

Fortunately, many popular models are naturally specified according to the local Markov property or the DAG factorization, in which case the population distribution $\mathbb P$ for the model belongs to the CI DAG model $\mathcal M(\mathcal G)$. 
Classic examples include, for instance, the basic Bayesian hierarchical model, such as the following example from Gelman et al.'s \emph{Bayesian Data Analysis} \cite{gelman}. 

\begin{example}[Bayesian hierarchical models]
    \label{ex: bayesian hierarchy}
    Rats are used by several laboratories conducting a study on the effects of a certain treatment for tumor growth.  
    Each lab uses standard F344-type rats for their control group (which did not receive the treatment).  
    We would like to view the rats across all laboratories as independent and identically distributed specimens from the same population. 
    However, there is a possibility that the rats are subject to different standards in each laboratory.  
    To capture these subtle differences, one could model the probability that an individual rat in the control group for each laboratory develops tumors using a \emph{Bayesian Hierarchical Model}, 
    \[
    \begin{split}
    &X_{i,j} \mid \Theta_i = \theta_i \sim \textrm{Ber}(\theta_i) \\
    &\Theta_i \mid (A, B) = (\alpha, \beta) \sim \textrm{Beta}(\alpha, \beta)  \\
    &f_{A,B}(\alpha, \beta) \propto (\alpha + \beta)^{-5/2}, 
    \end{split}
    \]
    which models the probability a rat $j\in[m_i]$ in lab $i\in[n]$ develops a tumor (i.e. $X_{i,j} = 1$) with a Bernoulli distribution $\textrm{Ber}(\theta_i)$ where the Bernoulli probabilities $\theta_1,\ldots, \theta_n$ are treated as samples from a $\textrm{Beta}(\alpha, \beta)$-distribution and the beta parameters are distributed according to the uninformative hyperprior $f_{A,B}(\alpha, \beta) \propto (\alpha + \beta)^{-5/2}$. 
    The model imposes several conditional independence assumptions in order to facilitate posterior computations.  
    These assumptions are equivalent to assuming that the joint distribution of $(X_1,\ldots, X_n, \Theta_1,\ldots, \Theta_n, A, B)$ belongs to the CI DAG model $\mathcal M(\mathcal G)$ where $\mathcal G$ is the following DAG: 
    \begin{center}
    \begin{tikzpicture}[thick, scale=0.6]
    \node[circle, inner sep=1pt, minimum width=1pt] (x1) at (0,4)  {$X_{1,1}$};
    \node[circle, inner sep=1pt, minimum width=1pt] (x2) at (0,3) {$\vdots$};
    \node[circle, inner sep=1pt, minimum width=1pt] (x3) at (0,2) {$X_{1,m_1}$};

    \node[circle, inner sep=1pt, minimum width=1pt] (3) at (0,0) {$\vdots$};
    
    \node[circle, inner sep=1pt, minimum width=1pt] (x4) at (0,-2) {$X_{n,1}$};
    \node[circle, inner sep=1pt, minimum width=1pt] (x5) at (0,-3) {$\vdots$};
    \node[circle, inner sep=1pt, minimum width=1pt] (x6) at (0,-4) {$X_{n,m_n}$};

    \node[circle, inner sep=1pt, minimum width=1pt] (t1) at (-3,3) {$\Theta_1$};

    \node[circle, inner sep=1pt, minimum width=1pt] (t) at (-3,0) {$\vdots$};

    \node[circle, inner sep=1pt, minimum width=1pt] (t2) at (-3,-3) {$\Theta_n$};

    \node[circle, inner sep=1pt, minimum width=1pt] (ab) at (-6,0) {$(A,B)$};
    
    \draw[->, very thick] (ab) -- (t1);
    % \draw[->, very thick] (ab) -- (t);
    \draw[->, very thick] (ab) -- (t2);
    \draw[->, very thick] (t1) -- (x1);
    % \draw[->, very thick] (t1) -- (x2);
    \draw[->, very thick] (t1) -- (x3);
    \draw[->, very thick] (t2) -- (x4);
    % \draw[->, very thick] (t2) -- (x5);
    \draw[->, very thick] (t2) -- (x6);
    % \draw[->, very thick] (1) -- (4);
    % \draw[->, very thick] (2) -- (4);
    % \draw[->, very thick] (4) -- (3);
    \end{tikzpicture} 
    \end{center}
\end{example}

\begin{remark}[Bayesian networks]
    \label{rem: bayesian networks}
    The terminology \emph{DAG model} and \emph{Bayesian network} are sometimes used interchangeably.  
    However, Bayesian networks are traditionally more specific.  
    The logic of the terminology \emph{Bayesian network} is that a DAG is a purely combinatorial, static object, but by associating the nodes to random variables and assuming the local Markov property holds we translate this combinatorial object into a tool for Bayesian modeling. 
    In particular, the DAG structure allows prior distributions to be naturally built into a combinatorial model for the joint distribution (see Example~\ref{ex: bayesian hierarchy}). 
    Since efficient posterior computations often utilize simplifications via CI relations that are valid in the joint distribution, the global Markov property becomes a useful combinatorial tool for assessing how to do this. 
    From this perspective, when one works with a Bayesian network, they typically have a specific distribution $\mathbb P$ for the random variables $X = (X_i)_{i\in[m]}$ in mind (possibly up to some unknown parameters).  
    Hence, formally, a \emph{Bayesian network} is a pair $(\mathcal G, \mathbb P)$ where $\mathcal G$ is a DAG and $\mathbb P \in \mathcal M (\mathcal G)$. 
    The CI DAG model studied in this section captures all distributions that define Bayesian networks (see \cite[Definition 3.1 and Definition 3.5]{KF}). 
\end{remark}

\subsubsection{Completeness and faithfulness}
\label{subsubsec: completness and faithfulness}
This section collects a few subtle, but important, details about our CI graphical models.  
For starters, the notion of d-separation in the DAG $\mathcal G$ is used to \emph{specify} a set of CI relations that define the model $\mathcal M(\mathcal G)$.
We have seen in Proposition~\ref{prop: CI axioms} that when a distribution satisfies one CI relation, it follows that it satisfies (several) others, which in turn can combine to imply even more CI relations.  
It is natural to ask, \emph{``Is it possible that the specific set of CI relations specified by the d-separation relations in $\mathcal G$ is defining a set of distributions $\mathcal M(\mathcal G)$ that all satisfy some additional CI relations that are not combinatorially encoded in the DAG $\mathcal G$?''}  
The answer, fortunately, is `no' for our CI graphical models.  
\begin{definition}
    \label{def: completeness}
    A Markov property $\mathbb \CI(\mathcal G)$ for a graph $\mathcal G$ is \emph{complete} if for every CI relation $X_A\independent X_B\mid X_C \notin \CI(\mathcal G)$ the model 
    \[
    \mathcal M(\mathcal G) = \{\mathbb P \in \mathbb D_X : \mathbb P \textrm{ satisfies all CI relations in $\CI(\mathcal G)$}\}
    \]
    contains a distribution $\mathbb P$ for which $X_A\not\independent X_B \mid X_C$. 
    That is, the Markov property $\CI(\mathcal G)$ is complete if it contains all CI relations implied by containment in the model $\mathcal M(\mathcal G)$. 
\end{definition}

The following theorem states that the global Markov property is complete for both DAGs and undirected graphs. 

\begin{theorem}[Completeness of the global Markov properties]
    \label{thm: CI DAG completeness}
    Let $\mathcal G$ be a DAG or an undirected graph. 
    The global Markov property for $\mathcal G$ is complete.  That is, any CI relation satisfied by all distributions in the CI graphical model $\mathcal M(\mathcal G)$ is in $\CI_G(\mathcal G)$. 
\end{theorem}

For a proof of Theorem~\ref{thm: CI DAG completeness}, we refer the reader to \cite[Theorem 3.4]{KF} for the DAG case, and \cite[Theorem 4.3]{KF} for the undirected graph case. 
A key ingredient in the proof of Theorem~\ref{thm: CI DAG completeness} is a positive answer to a second natural question: \emph{Do there exist distributions in $\mathcal M(\mathcal G)$ that satisfy exactly the set of CI relations in the global Markov property for $\mathcal G$?}
Fortunately, the answer is again `yes' for both DAGs and undirected graphs. 

\begin{definition}
    \label{def: faithful}
    Let $\mathcal G$ be a DAG or an undirected graph, and recall that $\CI_G(\mathcal G)$ denotes the global Markov property for $\mathcal G$.  
    We say a distribution $\mathbb P$ is \emph{faithful} to $\mathcal G$ if the CI relation $X_A \independent X_B \mid X_C$ holds in $\mathbb P$ if and only if $X_A\independent X_B\mid X_C\in \CI_G(\mathcal G)$. 
\end{definition}

In other words, when a distribution $\mathbb P$ is faithful to a graph $\mathcal G$ then the separation statements in the graph perfectly encode the CI relations that hold in the distribution.  
In this case, $\mathcal G$ is a perfect combinatorial representation of the conditional independence structure of the distribution $\mathbb P$. 
It is a helpful exercise to prove the following theorem. 

\begin{theorem}[Existence of faithful distributions]
    \label{thm: faithfulness}
    Let $\mathcal G$ be a DAG or an undirected graph. 
    There exist distributions $\mathbb P\in \mathcal M(\mathcal G)$ that are faithful to $\mathcal G$. 
\end{theorem}

One consequence of Theorems~\ref{thm: CI DAG completeness} and~\ref{thm: faithfulness} is that we can distinguish our two CI graphical models families.

\begin{remark}[Distinguishing the families of CI models for undirected graphs and DAGs]
    \label{rem: different families}
    We have now defined two graphical models families $\mathcal F_{\mathbb{UG}}$ and $\mathcal F_{\mathbb{DAG}}$.
    It is worth noting that these two families are indeed distinct.  
    This can be done by way of example. 
    Let 
    \begin{center}
    \begin{tikzpicture}[thick, scale=0.6]
    \node[circle, draw, inner sep=1pt, minimum width=1pt] (1u) at (0,0)  {$1$};
    \node[circle, draw, inner sep=1pt, minimum width=1pt] (2u) at (2,0) {$2$};
    \node[circle, draw, inner sep=1pt, minimum width=1pt] (3u) at (0,-2) {$3$};
    \node[circle, draw, inner sep=1pt, minimum width=1pt] (4u) at (2,-2) {$4$};
    
    \draw[-, very thick] (1u) -- (2u);
    \draw[-, very thick] (1u) -- (3u);
    \draw[-, very thick] (2u) -- (4u);
    \draw[-, very thick] (3u) -- (4u);

    \node at (-1.5,-1) {$\mathcal G^{(1)} =$};

    \node[circle, draw, inner sep=1pt, minimum width=1pt] (1) at (0 + 8,0)  {$1$};
    \node[circle, draw, inner sep=1pt, minimum width=1pt] (2) at (2 + 8,0) {$2$};
    \node[circle, draw, inner sep=1pt, minimum width=1pt] (3) at (0 + 8,-2) {$3$};
    \node[circle, draw, inner sep=1pt, minimum width=1pt] (4) at (2 + 8,-2) {$4$};
    
    \draw[->, very thick] (1) -- (2);
    \draw[->, very thick] (1) -- (3);
    \draw[->, very thick] (2) -- (4);
    \draw[->, very thick] (3) -- (4);

    \node at (-1.5 + 8,-1) {$\mathcal G^{(2)} =$};
    
    \end{tikzpicture}
    \end{center}
    By Theorems~\ref{thm: hammersley-clifford} and~\ref{thm: factorization DAG}, respectively, we know that $\mathcal M(\mathcal G^{(1)})$ consists of all distributions $\mathbb P \in \mathbb D_X$ satisfying the CI relations
    \[
    X_1 \independent X_4 \mid X_2,X_3 \qquad \textrm{and} \qquad X_2 \independent X_3 \mid X_1,X_4,
    \]
    and $\mathcal M(\mathcal G^{(2)})$ is all $\mathbb P \in \mathbb D_X$ satisfying
    \[
    X_1 \independent X_4 \mid X_2,X_3 \qquad \textrm{and} \qquad X_2 \independent X_3 \mid X_1. 
    \]
    By Theorem~\ref{thm: faithfulness}, there exist distributions $\mathbb P^{(1)} \in \mathcal M(\mathcal G^{(1)})$ and $\mathbb P^{(2)} \in \mathcal M(\mathcal G^{(2)})$ that are faithful to their respective graphs. 
    Observing the existence of these distributions is enough to show that $\mathcal M(\mathcal G^{(1)}) \neq \mathcal M(\mathcal G^{(2)})$. 
    To make the stronger observation that $\mathcal F_{\mathbb{UG}}\neq \mathcal F_{\mathbb{DAG}}$, one can use the distributions $\mathbb P^{(1)}, \mathbb P^{(2)}$ to show that $\mathcal M(\mathcal G^{(1)})\neq \mathcal M(\mathcal G)$ for \emph{any} $\mathcal G \in \mathbb{DAG}$ and $\mathcal M(\mathcal G^{(2)}) \neq \mathcal M(\mathcal G)$ for \emph{any} $\mathcal G \in \mathbb{UG}$. 
    Working out the details of this proof is a healthy exercise. 

    One consequence of this observation is that the two different graphical models families each provide models that may not be covered by the other.  
    Hence, in addition to the different graphs offering different interpretations, they also combine to provide a more robust framework for modeling general conditional independence structures underlying different data scenarios. 
    To make use of this robustness naturally requires that we complete the graphical models program (see Section~\ref{sec: intro}) for \emph{both} families. 
    As we will see in the coming section, the solutions to the graphical models program depend heavily upon the chosen graphical model family $\mathcal F_\mathbb G$. 
\end{remark}

\subsubsection{Model distinguishability for CI DAG models}
\label{subsubsec: CI DAG MECs}
A key difference between the undirected CI graphical model family and the CI DAG model family is that two CI DAG models $\mathcal M(\mathcal G)$ and $\mathcal M(\mathcal H)$ can be identical even if $\mathcal G \neq \mathcal H$.  
This phenomenon is already clear for even small examples on three nodes. 
For instance, 
\begin{center}
    \begin{tikzpicture}[thick, scale=0.6]
    \node[circle, draw, inner sep=1pt, minimum width=1pt] (1u) at (0,0)  {$1$};
    \node[circle, draw, inner sep=1pt, minimum width=1pt] (2u) at (2,0.24) {$2$};
    \node[circle, draw, inner sep=1pt, minimum width=1pt] (3u) at (4,-0.2) {$3$};
    
    \draw[->, very thick] (1u) -- (2u);
    \draw[->, very thick] (2u) -- (3u);

    \node at (-1.5,0) {$\mathcal G =$};

    \node[circle, draw, inner sep=1pt, minimum width=1pt] (1) at (0 + 10,0)  {$1$};
    \node[circle, draw, inner sep=1pt, minimum width=1pt] (2) at (2 + 10,-0.15) {$2$};
    \node[circle, draw, inner sep=1pt, minimum width=1pt] (3) at (4 + 10,0.1) {$3$};
    
    \draw[<-, very thick] (1) -- (2);
    \draw[<-, very thick] (2) -- (3);

    \node at (-1.5 + 10,0) {$\mathcal H =$};
    
    \end{tikzpicture}
\end{center}
define the same CI DAG model (i.e., $\mathcal M(\mathcal G) = \mathcal M(\mathcal H)$) by Theorem~\ref{thm: factorization DAG}, since both graphs contain only the single d-separation $1 \independent 3 \mid 2$. 

This phenomenon can be problematic when one wishes to interpret the directions of the edges in the DAG causally. 
Specifically, if we claim that our data-generating distribution $\mathbb P$ belongs to $\mathcal M(\mathcal G)$, then it also belongs to $\mathcal M(\mathcal H)$ (and vice versa).  
Hence, deducing that $\mathbb P$ is a good fit for our data-generating distribution does not immediately imply that we have learned an ordered hierarchy of the dependence structure in the data-generating process. 
In other words, the causal dependencies are unidentifiable from the distribution $\mathbb P$ alone. 
While this may come as disappointing news, it should also be reassuring, since distinguishing the causal dependence structure from \emph{only} the CI relations satisfied by the distribution would go against the age-old adage, \emph{``Correlation does not imply causation.''}

This apparent lack of \emph{structural identifiability}, raises an important question.  

\begin{questype}{Question 1}
    If  the data-generating distribution $\mathbb P$ has underlying causal structure given by an unknown DAG $\mathcal G$, then how much of the causal structure $\mathcal G$ can be identified using only the CI relations satisfied by $\mathbb P$ (i.e., using only the assumption that $\mathbb P\in \mathcal M(\mathcal G)$)? 
\end{questype}

This question boils down to asking for a characterization of when two DAGs $\mathcal G, \mathcal H$ define the same CI DAG model $\mathcal M(\mathcal G) = \mathcal M(\mathcal H)$.
Thanks to Theorem~\ref{thm: factorization DAG}, this equivalence can be seen from two perspectives: (1) from the perspective of equality of the statistical models, and (2) from the purely combinatorial perspective of the DAGs and how they encode CI relations.
This leads to two definitions of equivalence.  
The first pertains to statistical models:

\begin{definition}[Model equivalence]
    \label{def: model equivalence}
    Let $\mathcal M, \mathcal M' \subseteq \mathbb D_X$ be two statistical models (i.e. sets of distributions).  We say that $\mathcal M$ and $\mathcal M'$ are \emph{model equivalent} if $\mathcal M = \mathcal M'$. 
\end{definition}

Our second notion of equivalence is purely combinatorial:

\begin{definition}[Markov equivalence of DAGs]
    \label{def: Markov equivalence}
    Let $\mathcal G, \mathcal H$ be two DAGs. We say that $\mathcal G$ and $\mathcal H$ are \emph{Markov equivalent} if they have the same set of d-separation relations; i.e., we have equality of the following ordered triples of subsets of $[m]$:
    \[
    \{ (A,B,C) : A, B, C\subseteq[m], A \perp_\mathcal G B \mid C\} = \{ (A,B,C) : A, B, C\subseteq[m], A \perp_\mathcal H B \mid C\}.
    \]
\end{definition}

For CI DAG models, these two notions of equivalence are the same. 

\begin{lemma}
    \label{lem: markov equals model}
    Let $\mathcal G$ and $\mathcal H$ be two DAGs.  Then $\mathcal M(\mathcal G)$ and $\mathcal M(\mathcal H)$ are model equivalent if and only if $\mathcal G$ and $\mathcal H$ are Markov equivalent. 
\end{lemma}

\begin{proof}
    Suppose first that $\mathcal G$ and $\mathcal H$ are Markov equivalent, and let $\mathbb P\in \mathcal M(\mathcal G)$.  
    By Definition~\ref{def: DAG model}, we have that $\mathcal M(\mathcal G) = \mathcal M_G(\mathcal G)$, which means that $\mathbb P$ entails $X_A \independent X_B \mid X_C$ whenever $A$ and $B$ are d-separated given $C$ in $\mathcal G$. 
    By Markov equivalence of $\mathcal G$ and $\mathcal H$, we know that the two graphs have exactly the same set of d-separation relations.  So $\mathbb P$ entails all CI relations defining the global Markov property for $\mathcal H$.  
    Hence, $\mathbb P \in \mathcal M_G(\mathcal H) = \mathcal M(\mathcal H)$.  
    By symmetry of the argument we conclude that $\mathcal M(\mathcal G) = \mathcal M(\mathcal H)$. 

    Now suppose that $\mathcal M(\mathcal G) = \mathcal M(\mathcal H)$, and let $X_A\perp_\mathcal G X_B \mid X_C$ be a d-separation relation in $\mathcal G$. 
    Since $\mathcal M_G(\mathcal G) = \mathcal M(\mathcal G)$, it follows that $X_A \independent X_B \mid X_C$ holds in all $\mathbb P\in \mathcal M(\mathcal G)$. 
    Since $\mathcal M(\mathcal G) = \mathcal M(\mathcal H)$, then $X_A \independent X_B \mid X_C$ holds in all $\mathbb P\in \mathcal M(\mathcal H)$.
    By Theorem~\ref{thm: faithfulness}, we know that $\mathcal M(\mathcal H)$ contains a distribution $\mathbb P$ that is faithful to $\mathcal H$.  
    Since this distribution $\mathbb P$ satisfies the  CI relation $X_A\independent X_B \mid X_C$, it follows that $X_A \perp_\mathcal H X_B \mid X_C$.  
    Hence, all d-separation relations that hold in $\mathcal G$ also hold in $\mathcal H$. 
    By symmetry of the argument, we conclude that $\mathcal G$ and $\mathcal H$ are Markov equivalent. 
\end{proof}

Lemma~\ref{lem: markov equals model} reduces the problem of characterizing model equivalence for CI DAG models to the combinatorial problem of characterizing DAGs that have the same set of d-separation relations. 
This resulting combinatorial characterization requires two special graph theory definitions. 

\begin{definition}
    \label{def: skeleton}
    Let $\mathcal G = ([m],E)$ be a DAG. The \emph{skeleton} of $\mathcal G$ is the undirected graph $\mathcal G^{\textrm{skel}} = ([m], E^{\textrm{skel}})$ where
    \[
    E^{\textrm{skel}} = \{i - j : i\to j \in E\}. 
    \]
    That is, $\mathcal G^{\textrm{skel}}$ is the graph obtained from $\mathcal G$ by forgetting the directions of all edges. 
\end{definition}

\begin{definition}
    \label{def: v-structure}
    Let $\mathcal G = ([m],E)$ be a DAG.  A triple $(i,j,k)$ of nodes $i,j,k\in[m]$ is called a \emph{v-structure} (in $\mathcal G$) if $i\to j, k\to j\in E$ but $i \to k, k\to i\notin E$. 
\end{definition}

The following theorem and corollary were originally proven by Frydenberg \cite{frydenberg} and Verma and Pearl \cite{vermapearl}. 

\begin{theorem}
    \label{thm: VP}
    Two DAGs $\mathcal G$ and $\mathcal H$ are Markov equivalent if and only if they have the same skeleton and v-structures.
\end{theorem}

Combining this with Lemma~\ref{lem: markov equals model} we obtain the following probabilistic statement.

\begin{corollary}
    \label{cor: VP}
    Let $\mathcal G,\mathcal H\in\mathbb{DAG}$.  Then $\mathcal M(\mathcal G) = \mathcal M(\mathcal H)$ if and only if $\mathcal G$ and $\mathcal H$ have the same skeleton and v-structures. 
\end{corollary}

Corollary~\ref{cor: VP} provides a concrete answer to our question above.  
Specifically, it states that, using only the CI relations satisfied by $\mathbb P\in \mathcal M(\mathcal G)$, the most we can hope to learn about the unknown causal DAG $\mathcal G$ is the locations of the edges representing direct causal effects (e.g. the skeleton of $\mathcal G$) and the direction of some subset of the edges that are fixed by the locations of the v-structures. 
The following gives an example. 

\begin{example}
    \label{ex: CPDAG}
    Let $\mathcal G = (V,E)$ be the following DAG on vertex set $V = \{1,2,3,4,5,6\}$. 
    \begin{center}
    \begin{tikzpicture}[thick, scale=0.6]
    \node[circle, draw, inner sep=1pt, minimum width=1pt] (1) at (0,0)  {$1$};
    \node[circle, draw, inner sep=1pt, minimum width=1pt] (2) at (2,0) {$2$};
    \node[circle, draw, inner sep=1pt, minimum width=1pt] (3) at (0,-2) {$3$};
    \node[circle, draw, inner sep=1pt, minimum width=1pt] (4) at (2,-2) {$4$};
    \node[circle, draw, inner sep=1pt, minimum width=1pt] (5) at (4,-2) {$5$};
    \node[circle, draw, inner sep=1pt, minimum width=1pt] (6) at (3,-3.5) {$6$};
    
    \draw[->, very thick] (1) -- (2);
    \draw[->, very thick] (1) -- (3);
    \draw[->, very thick] (2) -- (4);
    \draw[->, very thick] (3) -- (4);
    \draw[->, very thick] (4) -- (5);
    \draw[->, very thick] (4) -- (6);
    \draw[->, very thick] (5) -- (6);
    \end{tikzpicture}
    \end{center}
    This DAG contains precisely one v-structure $(2,4,3)$, and it has the following skeleton
    \begin{center}
    \begin{tikzpicture}[thick, scale=0.6]
    \node[circle, draw, inner sep=1pt, minimum width=1pt] (1) at (0,0)  {$1$};
    \node[circle, draw, inner sep=1pt, minimum width=1pt] (2) at (2,0) {$2$};
    \node[circle, draw, inner sep=1pt, minimum width=1pt] (3) at (0,-2) {$3$};
    \node[circle, draw, inner sep=1pt, minimum width=1pt] (4) at (2,-2) {$4$};
    \node[circle, draw, inner sep=1pt, minimum width=1pt] (5) at (4,-2) {$5$};
    \node[circle, draw, inner sep=1pt, minimum width=1pt] (6) at (3,-3.5) {$6$};
    
    \draw[-, very thick] (1) -- (2);
    \draw[-, very thick] (1) -- (3);
    \draw[-, very thick] (2) -- (4);
    \draw[-, very thick] (3) -- (4);
    \draw[-, very thick] (4) -- (5);
    \draw[-, very thick] (4) -- (6);
    \draw[-, very thick] (5) -- (6);
    \end{tikzpicture}
    \end{center}
    Any DAG that is Markov equivalent to $\mathcal G$ has this skeleton and precisely the v-structure $(2,4,3)$.  
    This implies that no Markov equivalent DAG can contain the edges $5\to 4$ or $6\to 4$. 
    In other words, all DAGs that are Markov equivalent to $\mathcal G$ contain the edges $4\to 5$ and $4\to 6$.  
    For all other edges in the skeleton of $\mathcal G$, we can find a DAG Markov equivalent to $\mathcal G$ where the edge points the opposite direction.  Hence, these two edges and the edges in the v-structure constitute all edges whose directions can be determined from only the CI relations encoded in the model $\mathcal M(\mathcal G)$. 
    Drawing these edges into the skeleton produces what is referred to as a \emph{complete partially directed acyclic graph} (CPDAG) \cite{HB12} or \emph{essential graph} \cite{AMP97} of $\mathcal G$. 
    \begin{center}
    \begin{tikzpicture}[thick, scale=0.6]
    \node[circle, draw, inner sep=1pt, minimum width=1pt] (1) at (0,0)  {$1$};
    \node[circle, draw, inner sep=1pt, minimum width=1pt] (2) at (2,0) {$2$};
    \node[circle, draw, inner sep=1pt, minimum width=1pt] (3) at (0,-2) {$3$};
    \node[circle, draw, inner sep=1pt, minimum width=1pt] (4) at (2,-2) {$4$};
    \node[circle, draw, inner sep=1pt, minimum width=1pt] (5) at (4,-2) {$5$};
    \node[circle, draw, inner sep=1pt, minimum width=1pt] (6) at (3,-3.5) {$6$};
    
    \draw[-, very thick] (1) -- (2);
    \draw[-, very thick] (1) -- (3);
    \draw[->, very thick] (2) -- (4);
    \draw[->, very thick] (3) -- (4);
    \draw[->, very thick] (4) -- (5);
    \draw[->, very thick] (4) -- (6);
    \draw[-, very thick] (5) -- (6);
    \end{tikzpicture}
    \end{center}
    The CPDAG represents the maximum amount of ``causal'' structure that can be learned from the CI relations of $\mathbb P\in \mathcal M(\mathcal G)$ when $\mathcal G$ is interpreted causally. 
    In other words, the CPDAG provides the answer to our question above. 
\end{example}

It is also natural to partition the space of all DAGs $\mathbb{DAG}$ into Markov equivalence classes (MECs). 
Specifically, the \emph{Markov equivalence class} of $\mathcal G \in \mathbb{DAG}$ is the set of all DAGs that are Markov equivalent to $\mathcal G$. 
It is a nice exercise to verify that the DAG in Example~\ref{ex: CPDAG} has an MEC consisting of six elements. 
Note that the CPDAG construction from Example~\ref{ex: CPDAG} is well-defined for any DAG \cite{AMP97}. 
Hence, the CPDAG can be interpreted as a graphical representation of an MEC of DAGs. 

As a final note, we can see now that the CI DAG model family $\mathcal F_{\mathbb{DAG}}$ contains exactly one model for every Markov equivalence class of DAGs. 
In particular, Corollary~\ref{cor: VP} provides a complete answer to the model distinguishability question in the graphical models program for the graphical model family $\mathcal F_{\mathbb{DAG}}$. 
This answer is also the foundation of many algorithms for the model selection question for $\mathcal F_{\mathbb{DAG}}$, since it reduces learning a DAG to testing CI relations that recover the skeleton and v-structures of the DAG underlying the data-generating distribution (for instance, the PC algorithm in \cite{spirtes}).

\subsection{Conditional independence mixed graph models}
\label{subsubsec: ancestral models}

Graphs that include bidirected edges $i\leftrightarrow j$ are especially useful when modeling variables that are related by the effects of a set of \emph{unobserved confounders} $U$.
We can think of this as building a dependence structure representing a marginal distribution $\mathbb P_{[m]}$ for $X_{[m]} =(X_i)_{i\in[m]}$ of a distribution $\mathbb P$ for $X = (X_i)_{i\in{[m]\cup U}}$. 
For instance, for the following graphs $\mathcal G$ and $\mathcal H$, if $\mathbb P\in\mathcal M(\mathcal G)$ then we would want $\mathbb P_{\{1,2\}}\in\mathcal M(\mathcal H)$: 
\begin{center}
\begin{tikzpicture}[thick, scale=0.6]
    \node[circle, draw, inner sep=1pt, minimum width=1pt] (1) at (0,0)  {$1$};
    \node[circle, draw, inner sep=1pt, minimum width=1pt] (2) at (2,1) {$3$};
    \node[circle, draw, inner sep=1pt, minimum width=1pt] (3) at (4,0) {$2$};
    
    \node[circle, draw, inner sep=1pt, minimum width=1pt] (4) at (8 + 4,0.5) {$1$};
    \node[circle, draw, inner sep=1pt, minimum width=1pt] (5) at (10 + 4,0.5) {$2$};
    
    \draw[<-, very thick] (1) -- (2);
    \draw[<-, very thick] (3) -- (2);
    
    \draw[<->, very thick] (4) -- (5);

    \node at (-1, 0.5) {$\mathcal G = $};
    \node at (11, 0.5) {$\mathcal H = $};

    \node at (2,-1) {$\mathbb P \in\mathcal M(\mathcal G)$};
    \node at (6 + 1,-1) {$\Longrightarrow$};
    \node at (9 + 3.5,-1) {$\mathbb P_{\{1,2\}}\in \mathcal M(\mathcal H)$};
    \end{tikzpicture}  
\end{center}

One may hope that we do not need to introduce new types of edges for this to work out in general; i.e., that we could simply use a \emph{directed} edge in the graph $\mathcal H$ above.  
However, given a DAG $\mathcal G = ([m]\cup U, E)$ there is not always a DAG $\mathcal H = ([m], E')$ such that $\mathbb P_{[m]}\in \mathcal M(\mathcal H)$ whenever $\mathbb P\in \mathcal M(\mathcal G)$. 

One family that is closed under marginalization in this way are the CI models defined for \emph{directed ancestral graphs} \cite{richardson}.  
These are the directed mixed graphs $\mathcal G = ([m], E)$ satisfying
\begin{enumerate}
    \item $\mathcal G$ contains no directed cycles, and
    \item if $i\leftrightarrow j\in E$ then there is no directed path in $\mathcal G$ between $i$ and $j$. 
\end{enumerate}
To assign a CI model $\mathcal M(\mathcal G)$ to an ancestral graph, we use the global Markov property for DAGs, but we replace d-separation with \emph{m-separation}.
Fortunately, m-separation is precisely the same thing as d-separation, but applied to directed ancestral graphs.  
That is, we simply add the following subgraphs to the list in item~(1) of Definition~\ref{def: d-separation}: 
\[
v_{k-1} \leftrightarrow v_k \rightarrow v_{k+1}, \qquad v_{k-1} \leftarrow v_k \leftrightarrow v_{k+1},
\]
and the following subgraphs to item~(2):
\[
v_{k-1} \leftrightarrow v_k \leftrightarrow v_{k+1}, \qquad v_{k-1} \leftrightarrow v_k \leftarrow v_{k+1}, \qquad v_{k-1} \rightarrow v_k \leftrightarrow v_{k+1}.
\]
Hence, the \emph{directed ancestral CI model} for $\mathcal G$ is 
\[
\mathcal M(\mathcal G) = \{\mathbb P\in \mathbb D_X : \mathbb P \textrm{ entails } X_A\independent X_B \mid X_C \textrm{ whenever } X_A\perp_\mathcal G X_B \mid X_C\}. 
\]
Letting $\mathbb{DAN}$ denote the set of all directed ancestral graphs, the \emph{directed ancestral CI graphical model family} is
\[
\mathcal F_{\mathbb{DAN}} = \{\mathcal M(\mathcal G) : \mathcal G\in \mathbb{DAN}\}. 
\]
The theory of directed ancestral graphs (and, more generally, \emph{ancestral graphs}) was developed by Richardson and Spirtes in \cite{richardson}, and a characterization of model equivalence was given in \cite{ali}.
In recent years, ancestral graphs have become popular in the field of causality since they allow for the incorporation of latent confounders, which are very natural to expect in a causal system. 
More generally, Sadeghi and Lauritzen extended the notion of m-separation to define CI graphical models for general mixed graphs. 
For more details, we refer the reader to their paper  \cite{SD14}.

\section{Parametric Graphical Models}
\label{subsec: parameteric GMs}
In most scenarios, we do not model our data based solely on conditional independence without also imposing some additional parametric assumptions. 
In other words, it is not common to model the unknown data-generating distribution using a distribution $\mathbb P$ where the only assumptions are that $\mathbb P\in \mathbb D_X$ and $\mathbb P$ satisfies some specified set of CI relations.  
This is for several reasons. 

Typically, adding some parametric assumptions on our model provides us with interpretable parameters that can be soundly estimated as functions of the observed data.  
This lets us make inferences about the data-generating scenario that translate into interpretable predictions. 
For DAGs, it is often desirable to have parameters with a natural causal interpretation, such as a value that represents the direct effect of one variable on another.  
Passing to such parametric models can also allow us to prove valuable properties about our estimators for these effects, such as consistency and unbiasedness. 
In the case of DAG models, a well-chosen parameterization can even resolve the lack of structural identifiability exhibited for CI DAG models in Section~\ref{subsubsec: CI DAG MECs}. 
This is particularly important when interpreting the edges of the DAG causally. 

For these reasons, it is common to work with \emph{submodels} of the conditional independence models $\mathcal M(\mathcal G)$ for undirected,  directed or mixed graphs that are produced by intersecting $\mathcal M(\mathcal G)$ with a parametric subfamily of $\mathbb D_X$. 

In this setting, each graph $\mathcal G$ in a family of graphs $\mathbb G$ is assigned a parameter space $\Theta\subseteq \mathbb R^d$ and a function $f$ that, for each $\theta\in \Theta$, parametrizes a distribution in the graphical model. %parameterizing the distributions in the model using the parameters $\theta\in \Theta$. 
Throughout, we use the map
\[
\Psi: \mathcal G \mapsto (f, \Theta)
\]
to denote this assignment, $\mathcal M_{f,\Theta}(\mathcal G)$ denotes the resulting \emph{parametric graphical model} for $\mathcal G$, and 
\[
\mathcal F_{\mathbb G, \Psi} = \{\mathcal M_{f,\Theta}(\mathcal G) : \mathcal G \in \mathbb G\}.
\]
is the \emph{parametric graphical model family}. 

Depending on our choices for the family of graphs $\mathbb G$ and the parameterization assignment $\Psi$, we obtain different families $\mathcal F_{\mathbb G, \Psi}$ that are suited to different types of data problems. 
Hence, for each $\mathcal F_{\mathbb G, \Psi}$, we require solutions to the graphical models program. 
We now highlight some key examples of parametric graphical model families and their status regarding the graphical models program. % and the specific properties they afford us in regard to inference. 

\subsection{Undirected parametric graphical models}
\label{subsubsec: UPGMs}
We have already seen one of the most popular examples of a parametric undirected graphical model family in Example~\ref{ex: UGGMs}. 
These are the \emph{undirected Gaussian graphical models}. 
Recall from Section~\ref{subsec: probability} that $\mathcal N_{0,m}\subseteq \mathbb D_X$ denotes the set of all multivariate normal distributions for variables $X = (X_i)_{i\in[m]}$ with mean vector $0$. 
The undirected Gaussian graphical model is obtained by intersecting the undirected CI graphical model $\mathcal M(\mathcal G)$ (Definition~\ref{def: UG model}) with the set $\mathcal N_{0,m}$:
\begin{definition}
    \label{def: undirected gaussian GM}
    Let $\mathcal G =([m], E)\in \mathbb{UG}$. The \emph{undirected Gaussian graphical model} for $\mathcal G$ is
    \[
    \mathcal M_{\textrm{Gauss}}(\mathcal G) = \mathcal M(\mathcal G)\cap \mathcal N_{0,m}.
    \]
\end{definition}

Note that, in practice, the restriction to mean $0$ normal distributions is mild since one often centers the data prior to building a model.  
In this case every normal distribution $\mathcal N(0, \Sigma)\in \mathcal N_{0,m}$ is completely determined by its covariance matrix.  
In Example~\ref{ex: UGGMs}, the undirected Gaussian graphical model for $\mathcal G$ is equivalently described as 
\begin{equation}
\label{eqn: UGGM equiv}
\mathcal M_{\textrm{Gauss}}(\mathcal G) = \{\mathcal N(0,\Sigma) : \Sigma^{-1} = K\in \mathcal L_\mathcal G\},
\end{equation}
where
\begin{equation}
\label{eqn: UG concentration space}
\mathcal L_\mathcal G = \{ K = (k_{ij})_{i,j=1}^m\in \textrm{PD}_m : k_{ij} = 0 \textrm{ whenever } i- j \notin E\}.  
\end{equation}
In other words, the undirected Gaussian graphical model is obtained, geometrically, by considering the set of all concentration matrices given by intersecting the positive definite cone $\textrm{PD}_m$ with the linear subspace in $\textrm{Sym}(\mathbb R^m)$ that sets all coordinates corresponding to nonedges of $\mathcal G$ equal to zero. 

\begin{remark}[The parameterizing map for $\mathcal M_{\textrm{Gauss}}(\mathcal G)$]
    \label{rem: geometric parametrization UGG}
With the help of the adjugate formula for computing a matrix inverse, we can slightly rephrase~\eqref{eqn: UGGM equiv} to make explicit the functional parameterization of the model.
Specifically, let
\begin{equation}
\label{eqn: UG cov param}
\varphi_\mathcal G : \mathcal L_\mathcal G \to \textrm{PD}_m; \qquad \varphi_\mathcal G(K) = \left(\frac{(-1)^{i+j}|K_{[m]\setminus\{j\}, [m]\setminus\{i\}}|}{|K|}\right)_{i,j=1}^m, 
\end{equation}
where $M_{A, B}$ denotes the submatrix of $M$ with rows and column indices, respectively in the sets $A$ and $B$, and $|M|$ denotes the matrix determinant.
The following rephrasing makes the role of this (strictly `geometric') parameterization map more apparent, a fact that will be valuable in Section~\ref{sec: algebra}:
\begin{equation}
\label{eqn: UGGM algebraic}
\mathcal M_{\textrm{Gauss}}(\mathcal G) = \{\mathcal N(0,\Sigma) : \Sigma\in \varphi_\mathcal G(\mathcal L_\mathcal G)\}.
\end{equation}
Using our notation for a general parametric graphical model family $\mathcal F_{\mathbb G, \Psi}$, we could take the assignment $\Psi$ for this family to be $\Psi: \mathcal G \mapsto (\varphi_\mathcal G, \mathcal L_\mathcal G)$. 
Hence, $\mathcal F^{\textrm{Gauss}}_{\mathbb{UG}} = \mathcal F_{\mathbb{UG}, \Psi}$. 
Notice that, while the function $\varphi_\mathcal G$ and parameter space $\mathcal L_\mathcal G$ depend on the graph, these choices have the same form for each model in the family.  
\end{remark}

This yields our first \emph{parametric graphical model family}
\[
\mathcal F_{\mathbb{UG}}^{\textrm{Gauss}} = \{\mathcal M_{\textrm{Gauss}}(\mathcal G) : \mathcal G \in \mathbb{UG}\}. 
\]
We briefly describe the status of the graphical models program for $\mathcal F_{\mathbb{UG}}^{\textrm{Gauss}}$. 

\paragraph{\emph{Graphical Models Program Question~1 (Model distinguishability):}}

The question of model distinguishability is completely settled for $\mathcal F_{\mathbb{UG}}^{\textrm{Gauss}}$. 
As for undirected CI graphical models, every undirected graph $\mathcal G$ defines a unique model $\mathcal M(\mathcal G)\in \mathcal F_{\mathbb{UG}}^{\textrm{Gauss}}$.
This is also true for undirected CI graphical models, so it may be tempting to say the result follows from the same result for the undirected CI graphical model family.  
However, care should be taken since $\mathcal M_{\textrm{Gauss}}(\mathcal G)$ is a subset of $\mathcal M (\mathcal G)$, and it is entirely possible that two subsets of two non-equal sets are equal.  
Fortunately, it is a straightforward exercise to verify this does not happen in this case. 

\paragraph{\emph{Graphical Models Program Question~2 (Parameter identifiability and inference):}}

The question of parameter identifiability can also be seen as settled. 
The model $\mathcal M_{\textrm{Gauss}}(\mathcal G)$ satisfies a strong property known as \emph{global identifiability} \cite{sullivant2023} since the map $\varphi_\mathcal G$ in Remark~\ref{rem: geometric parametrization UGG} is injective.  
The model-defining parameters are the concentration parameters $k_{ij}$ in the matrix $K$. 
For $\mathcal G\in \mathbb{UG}$ and a distribution $\mathcal N(0,\Sigma)\in\mathcal M_{\textrm{Gauss}}(\mathcal G)$ the parameters $k_{ij}$ can be uniquely recovered by simply applying Cramer's rule to $\Sigma$:
\[
k_{ij} = \frac{(-1)^{i+j}|\Sigma_{[m]\setminus j, [m]\setminus i}|}{|\Sigma|}.
\]
There are also methods available for parameter inference, such as \emph{iterative proportional scaling} \cite{Speed1986} or methods due to Lauritzen \cite{Lauritzen1996} that utilize the factorization of the model according to the Hammersley-Clifford Theorem (Theorem~\ref{thm: hammersley-clifford}). 

\paragraph{\emph{Graphical Models Program Question~3 (Model selection):}}

Here we use the terminology \emph{model selection} to refer to the process of selecting a model $\mathcal M_{\textrm{Gauss}}(\mathcal G)\in \mathcal F_{\mathbb{UG}}^{\textrm{Gauss}}$ that is fitting to the observed data.  
Note that this amounts to choosing a undirected graph $\mathcal G$ that fits the observed dependence structure.  
A model selection algorithm may also return both a model $\mathcal M_{\textrm{Gauss}}(\mathcal G)$ and a fitting candidate distribution $\mathbb P\in\mathcal M_{\textrm{Gauss}}(\mathcal G)$ for the observed data. 

Undirected Gaussian graphical models are particularly popular in high-dimensional statistics where testing conditional independence relations via statistical hypothesis tests becomes an unreliable way of estimating a dependence structure in the data. 
The parameterization of undirected Gaussian graphical models in Remark~\ref{rem: geometric parametrization UGG} allows us to construct a natural estimator for an undirected graph representation of the CI dependence structure in Gaussian data known as the \emph{graphical lasso estimator}:
\[
    \textrm{argmax}_{K\in \textrm{PD}_m}\left(\log \det(K) - \textrm{trace}(S K) - \lambda\sum_{1\leq i \leq j \leq m}|k_{ij}|\right),
\]
where $\lambda > 0$ is a penalization parameter, and $S$ is the sample covariance matrix. 
Making $\lambda$ large encourages the estimator to force more entries of the covariance matrix to $0$. 
This results in a sparser graph since the zeros in the concentration matrix correspond to the nonedges of $\mathcal G$. 

\subsubsection{Colored undirected Gaussian graphical models}
\label{subsubsec: colored UGs}
% \paragraph{\emph{Colored undirected Gaussian graphical models:}}

More recently, expansions of the undirected Gaussian graphical model family have appeared.  
The logic here is that $\mathcal M_{\textrm{Gauss}}(\mathcal G)$ is defined by a simple linear section of the positive definite cone $\mathcal L_\mathcal G$ given by setting the concentration coordinates $k_{ij}$ equal to zero whenever $i - j \notin E$. 
This is easily generalized by taking smaller linear subspaces defined by additional linear constraints $k_{ij} = k_{st}$ for pairs of edges $i - j, s - t\in E$ or $k_{ii} = k_{jj}$ for pairs of nodes $i,j\in[m]$. 
Combinatorially, the additional constraints can be represented coloring edges and nodes in $\mathcal G$ with the same color whenever their concentration parameters are set equal in the model. 
For example, 
\begin{center}
\begin{center}
    \begin{tikzpicture}[thick, scale=0.6]
    \node[circle, draw, fill=green!50, inner sep=1pt, minimum width=1pt] (1) at (0,0)  {$1$};
    \node[circle, draw, fill=green!50, inner sep=1pt, minimum width=1pt] (2) at (2,0) {$2$};
    \node[circle, draw, fill=yellow!50, inner sep=1pt, minimum width=1pt] (3) at (0,-2) {$3$};
    \node[circle, draw, fill=yellow!50, inner sep=1pt, minimum width=1pt] (4) at (2,-2) {$4$};
    
    \draw[-,red!50, very thick] (1) -- (2);
    \draw[-, blue!50, very thick] (1) -- (3);
    \draw[-, blue!50, very thick] (2) -- (4);
    \draw[-, red!50, very thick] (3) -- (4);
    \end{tikzpicture}
    \end{center}
\end{center}
encodes the additional constraints $k_{12} = k_{34}, k_{13} = k_{24}, k_{11} = k_{22}, k_{33} = k_{44}$. 
The result is an undirected Gaussian graphical model $\mathcal M(\mathcal G, c)$ for each graph $\mathcal G$ and each coloring of $c$, creating a much larger parametric graphical model family
\[
\mathcal F_{\textrm{c-}\mathbb{UG}}^{\textrm{Gauss}} = \{\mathcal M(\mathcal G, c): (\mathcal G,c) \in \textrm{c-}\mathbb{UG}\}, %\in \mathbb G, c \textrm{ a coloring of } \mathcal G\}.
\]
where $\textrm{c-}\mathbb{UG}$ denotes the set of all colored undirected graphs; i.e., pairs $(\mathcal G, c)$ where $\mathcal G\in \mathbb{UG}$ and $c$ is a coloring of $\mathcal G$. 
It is not hard to see that $\mathcal F_{\mathbb G}^{\textrm{Gauss}}\subset \mathcal F_{\textrm{c-}\mathbb{UG}}^{\textrm{Gauss}}$. 
In fact, much of the graphical models program for $\mathcal F_{\mathbb G}^{\textrm{Gauss}}$ outlined above lifts to this larger family \cite{HS08}. 

\subsubsection{Parameterizing discrete undirected graphical models}
\label{subsubsec: disc UGs}

A third family of parametric undirected graphical models arises when the variables $X_i$ in $X =(X_i)_{i\in[m]}$ are all discrete and the base measure is the counting measure.  
In this case, the parameterization of the model for $\mathcal G\in \mathbb{UG}$ comes directly from the factorization of the model obtained via the Hammersley-Clifford Theorem (Theorem~\ref{thm: hammersley-clifford}). Namely, the \emph{discrete undirected graphical model} for $\mathcal G$ has mass function given by
\[
f_X(x) = \frac{1}{Z(\theta)}\prod_{C\in \mathcal C(\mathcal G)}\theta_{x_C}^{(C)}, 
\]
where, for every maximal clique $C$ in $\mathcal G$ and every marginal outcome $x_C\in \mathcal X_C$, we have a parameter $\theta_{x_C}^{(C)} >0$. 
This model is an example of a \emph{hierarchical log-linear model}. 
We refer the reader to \cite{sullivant2023} for more details.

\subsection{Parametric DAG models}
\label{subsubsec: parametric DAGs}

It is equally as natural to construct parametric graphical models that are submodels CI DAG models. 
One of the most common ways to specify a parametric DAG model for $\mathcal G\in \mathbb{DAG}$ is to use \emph{noisy functional relations}, producing what is commonly called a \emph{structural equation model} (SEM). 

% \paragraph{\emph{Structural equation models:}}

For a DAG $\mathcal G = ([m],E)$, we can specify a SEM as follows: 
Let $\varepsilon_1,\ldots, \varepsilon_m$ be a collection of mutually independent error variables and $f_1, \ldots, f_m$ a collection of functions.
The $\varepsilon_i$ and $f_i$ may each be specified in terms of some parameters. 
We collect these parameters in a vector $\theta\in \Theta$, where $\Theta$ denotes the set of all allowed choices for $\theta$ (i.e.~the parameter space). 
The \emph{structural equation model} (SEM) for these choices is the set of all
distributions $\mathbb P$ for $X = (X_i)_{i\in[m]}$ satisfying the system of \emph{structural equations}
\begin{equation}
    \label{eqn: SEM}
    X_i = f_i(X_{\pa(i)}, \varepsilon_i) \qquad \textrm{for all } i\in [m]. 
\end{equation}
In other words, we obtain the \emph{parametric DAG model} for $\mathcal G$
\[
\mathcal M_{f,\varepsilon, \Theta}(\mathcal G) = \{ \mathbb P\in \mathbb D_X : \mathbb P \textrm{ satisfies~\eqref{eqn: SEM}} \textrm{ for some } \theta \in \Theta\}.
\]
Notice that the functions $f = (f_1,\ldots, f_m)$, errors $\varepsilon = (\varepsilon_1,\ldots, \varepsilon_m)$ and even the parameters $\Theta$ may be specific to the graph $\mathcal G$.  
A parametric family usually specifies a general form for each, via a map
\[
\Psi: \mathcal G \mapsto (f,\varepsilon, \Theta)
\]
that assigns each graph in $\mathbb{DAG}$ a triple describing its associated parameterization. 
Hence, we can denote the resulting \emph{parametric DAG model family} as
\[
\mathcal F_{\mathbb{DAG}, \Psi} = \{\mathcal M_{f,\varepsilon, \Theta}(\mathcal G) : \mathcal G \in \mathbb{DAG} \textrm{ where } \Psi(\mathcal G) = (f,\varepsilon, \Theta)\}.
\]

\begin{remark}
    \label{rem: subscripts}
    Note that while the subscripts here seem cumbersome, they are each important since the properties of the parametric graphical model family $\mathcal F_{\mathbb G, \Psi}$ are sensitive to each of $\mathbb G$ and $\Psi$. 
    Indeed, replacing $\mathbb G$ with a subset $\mathbb G' \subsetneq \mathbb G$ of DAGs, changing the $f$, distribution for $\varepsilon$ or even the parameter space $\Theta$ assigned to each graph can drastically change the solutions to the graphical models program. 
\end{remark}

Due to the recursive structure of the DAG, a distribution $\mathbb P\in \mathcal M_{f,\varepsilon, \Theta}(\mathcal G)$ amounts to a transformation of the error distribution for $\varepsilon$. 
It is also not hard to see that any distribution $\mathbb P\in \mathcal M_{f,\varepsilon,\Theta}(\mathcal G)$ satisfies the local Markov property for $\mathcal G$ given in Definition~\ref{def: DAG MPs}~(2) \cite[Theorem 1.4.1]{pearl2009}. 
In particular, we have that
\[
\mathcal M_{f,\varepsilon, \Theta}(\mathcal G)\subseteq \mathcal M(\mathcal G),
\]
where $\mathcal M(\mathcal G)$ is the CI DAG model in Definition~\ref{def: DAG model}.

Parametric DAG models specified by SEMs are a prominent tool in fields such as economics and the social sciences \cite{pearl2009}.  
More recently, they have become fundamental models of study in the modern theory of causality \cite{pearl2009, peters2017} largely due to the fact that the functional equations offer some measure of interpretability of the notion of a direct effect, while also allowing for a high degree of flexibility in their specification. 

When one would like to interpret the edges of the DAG $\mathcal G$ causally, working with a well-chosen parametric DAG model $\mathcal M_{f,\varepsilon, \Theta}(\mathcal G)$ can offer certain advantages. 
Simple choices of the functions $f$ and the errors $\varepsilon$ usually produce a family where the graphical models program can be thoroughly resolved.

\subsubsection{Linear Gaussian DAG models}
\label{subsubsec: linear Gaussian DAGs}
% \paragraph{\emph{Linear Gaussian DAG models:}}

A classic example of a parametric DAG model family are the \emph{linear Gaussian DAG models}. 
These models are produced by taking the functions $f_i$ to be linear and the errors $\varepsilon_i\sim\mathcal N(0,\omega_i)$ to be Gaussian with mean $0$ and variance $\omega_i>0$. 
(Note again that the mean $0$ assumption is relatively mild as long as we are willing to center our data.)
In this case, the structural equations in~\eqref{eqn: SEM} become
\begin{equation}
\label{eqn: linear SEM}
X_i = \sum_{i=1}^m\lambda_{ki}X_k + \varepsilon_i \qquad \textrm{for all } i\in[m],
\end{equation}
where $\lambda_{ki}\in\mathbb R$ with $\lambda_{ki}= 0$ when $k\to i\notin E$ and $\omega_i\in \mathbb R_{>0}$. 
Hence, each distribution in the model is parameterized via a vector
\begin{equation}
\label{eqn: Gauss DAG params}
(\lambda, \omega) = ((\lambda_{kj})_{k\to i\in E}, (\omega_i)_{i\in[m]}) \in \Theta = \mathbb R^{E}\times \mathbb R^m_{>0}.
\end{equation}
Collecting the parameters $\lambda$ and $\omega$ into, respectively, a matrix $\Lambda = (\lambda_{ki})_{k,i=1}^m$ and a diagonal matrix $\Omega = \textrm{diag}(\omega_1,\ldots, \omega_m)$ we may rewrite the equations in~\eqref{eqn: linear SEM} as
\begin{equation}
    \label{eqn: linear matrix SEM}
    X = \Lambda^TX + \varepsilon.
\end{equation}
From~\eqref{eqn: linear matrix SEM}, it is apparent that $X\sim \mathcal N(0, \Sigma)$ where 
\begin{equation}
\label{eqn: covariance matrix}
\Sigma = (1 - \Lambda)^{-T}\Omega(1 - \Lambda)^{-1}. 
\end{equation}

The resulting parametric DAG model is defined as follows.
\begin{definition}
    \label{def: linear Gaussian DAG models}
    Let $\mathcal G = ([m],E)\in \mathbb{DAG}$. 
    The \emph{linear Gaussian DAG model} for $\mathcal G$ is the set of distributions
    \[
    \mathcal M_{\textrm{Gauss}}(\mathcal G) = \{\mathcal N(0,\Sigma) : \Sigma  \textrm{ satisfies~\eqref{eqn: covariance matrix} for some } (\lambda, \omega)\in\mathbb R^{E}\times \mathbb R^m_{>0}\}. 
    \]
    The \emph{linear Gaussian DAG model family} is 
    \[
    \mathcal F_\mathbb{DAG}^{\textrm{Gauss}} = \{\mathcal M_{\textrm{Gauss}}(\mathcal G) : \mathcal G \in\mathbb{DAG}\}. 
    \]
\end{definition}

It can also be useful to have an explicit formula for the coordinates $\sigma_{ij}$ of the covariance matrix $\Sigma = (\sigma_{ij})_{i,j=1}^m$ of a distribution $\mathcal N(0,\Sigma)\in\mathcal M_{\textrm{Gauss}}(\mathcal G)$. 
Such a formula is obtained via certain subgraphs of $\mathcal G$ called \emph{treks}. 

\begin{definition}
    \label{def: trek}
    Let $\mathcal G = ([m], E)\in \mathbb{DAG}$ and $i,j\in[m]$. 
    A \emph{trek} from $i$ to $j$ in $\mathcal G$ is a pair of directed paths $T = ((v_k)_{k\in[s]}, (w_k)_{k\in[\ell]})$ satisfying $v_1 = w_1$, $v_s = i$ and $w_\ell = j$.  
    We let $\textrm{top}(T) = v_1 = w_1$, and $\mathcal T(i,j)$ denote the set of all treks from $i$ to $j$ in $\mathcal G$. 
    For $T= ((v_k)_{k\in[s]}, (w_k)_{k\in[\ell]})\in\mathcal T(i,j)$, we define the \emph{trek monomial}
    \[
    m_T = \omega_{\textrm{top}(T)}\prod_{k\in[s-1]}\lambda_{v_kv_{k+1}}\prod_{k\in[\ell-1]}\lambda_{w_kw_{k+1}}. 
    \]
\end{definition}

\begin{example}
    \label{ex: trek}
    Let $\mathcal G = ([5], E)$ be the following DAG:
    \begin{center}
    \begin{tikzpicture}[thick, scale=0.6]
    \node[circle, draw, inner sep=1pt, minimum width=1pt] (1) at (0,0)  {$1$};
    \node[circle, draw, inner sep=1pt, minimum width=1pt] (2) at (2.1,-1.5) {$2$};
    \node[circle, draw, inner sep=1pt, minimum width=1pt] (3) at (1.25,-3) {$3$};
    \node[circle, draw, inner sep=1pt, minimum width=1pt] (4) at (-2.1,-1.5) {$4$};
    \node[circle, draw, inner sep=1pt, minimum width=1pt] (5) at (-1.25,-3) {$5$};
    
    \draw[<-, very thick] (1) -- (2);
    \draw[<-, very thick] (2) -- (3);
    \draw[->, very thick] (3) -- (5);
    \draw[<-, very thick] (1) -- (4);
    \draw[->, very thick] (4) -- (5);
    \end{tikzpicture}
    \end{center}
    There are two treks in $\mathcal G$ between nodes $1$ and $5$. They have the following trek monomials: 
    \[
    \omega_3\lambda_{32}\lambda_{21}\lambda_{35}, \qquad \omega_4\lambda_{41}\lambda_{45}. 
    \]
    There are four treks between $1$ and itself, with trek monomials
    \[
    \omega_1, \qquad \omega_2\lambda_{21}^2, \qquad \omega_4\lambda_{41}^2, \qquad \omega_3\lambda_{32}^2\lambda_{21}^2. 
    \]
\end{example}

The following result is commonly called the \emph{trek rule}. 

\begin{proposition}[Trek rule]
    \label{prop: trek rule}
    Let $\mathcal G = ([m],E)\in \mathbb{DAG}$ with associated covariance matrix 
    \[
    \Sigma = (\sigma_{ij})_{i,j=1}^m = (1- \Lambda)^{-T}\Omega(1- \Lambda)^{-1}
    \]
    as in~\eqref{eqn: covariance matrix}. Then for all $i,j\in[m]$
    \[
    \sigma_{ij} = \sum_{T\in\mathcal T(i,j)}m_T. 
    \]
\end{proposition}

\begin{proof}
    Since $\mathcal G$ is acyclic, the matrix $(1- \Lambda)$ is nilpotent.  In particular, 
    \[
    (1-\Lambda)^{-1} = \frac{1}{1 - \Lambda} = \sum_{k = 0}^N\Lambda^k
    \]
    for some $N\geq 0$. 
    It can be checked that the $(i,j)$-th entry of the matrix power $\Lambda^k$ is the sum over the monomials $\prod_{t = 1}^{k}\lambda_{v_tv_{t +1}}$ for every directed path $(v_t)_{t\in[k+1]}$  from $i$ to $j$ in $\mathcal G$ of length $k$. 
    The result follows from simply analyzing the matrix multiplication $(1- \Lambda)^{-T}\Omega(1- \Lambda)^{-1}$ coordinatewise. 
\end{proof}

\begin{remark}
    \label{rem: trek map}
The trek rule in Proposition~\ref{prop: trek rule} allows us to give an explicit (geometric) parameterization of the set of covariance matrices for the distributions in $\mathcal M_{\textrm{Gauss}}(\mathcal G)$:  
\begin{equation}
    \label{eqn: trek map}
    \varphi_\mathcal G: \mathbb R^{E}\times \mathbb R^m_{>0} \to \textrm{Sym}(\mathbb R^m); \qquad \varphi_\mathcal G(\theta) = \left(\sum_{T\in\mathcal T(i,j)}m_T\right)_{i,j\in[m]}.
\end{equation}
Note that this map is the directed acyclic graph analog of the map for undirected graphs given in~\eqref{eqn: UG cov param}. 
In particular, we have that 
\[
\mathcal M_{\textrm{Gauss}}(\mathcal G) = \{\mathbb P\in \mathbb D_X : \mathbb P = \mathcal N(0,\Sigma) \textrm{ for some } \Sigma \in \varphi_\mathcal G(\mathbb R^E\times \mathbb R^m_{>0})\}. 
\]
\end{remark}

\begin{remark}
    \label{rem: treks and effects}
    One reason for the popularity of linear Gaussian DAG models is the ease of interpretability of the parameters of the model.  
    If one is willing to interpret the graph as a cause-effect network, then the parameter $\lambda_{ki}$ is a natural measure of the \emph{direct effect} of $X_k$ on $X_i$.  
    This gives a quantitative aspect to the qualitative representation of this direct effect via the edge $k\to i$ in $\mathcal G$. 
    Hence, when $\mathcal G$ is being interpreted causally, it is common to refer to the parameter $\lambda_{ki}$ as the direct effect of $X_k$ on $X_i$.  
    When a causal interpretation is not permitted, $\lambda_{ki}$ is often called a \emph{structural coefficient}. 
    
    Generalizing this, treks reveal how the DAG encodes the different reasons two variables $X_i$ and $X_j$ covary. Specifically, the covariance $\sigma_{ij}$ is sum over the products of these structural coefficients along any direct effect between $X_i$ and $X_j$ (e.g. $i\to j$), and \emph{indirect effect} between the two variables (e.g. a directed path from $i$ to $j$), as well as the effects of any common \emph{confounders} $X_k$ (e.g. treks from $i$ to $j$ with top node $k$). 
    The simplicity of the interpretation of the model parameters is one of the key reasons these models are now being widely used. 
\end{remark}

\subsubsection{Colored Gaussian DAG models}
\label{subsubsec: colored DAGs}
% \paragraph{\emph{Colored Gaussian DAG models:}}

Note that, similar to undirected Gaussian DAG models, we can produce a larger family of parametric DAG models by considering \emph{colored} DAGs. 
We let $\textrm{c-}\mathbb{DAG}$ denote the set of all colored DAGs. 

A \emph{colored DAG} is a pair $(\mathcal G, c)$ where $\mathcal G = ([m], E)$ is a DAG and $c$ is a coloring of its vertices and edges.  
To obtain the \emph{colored Gaussian DAG model} for $(\mathcal G, c)$, denoted $\mathcal M_{\textrm{Gauss}}(\mathcal G, c)$, we simply restrict the parameter space $\mathbb R^{E}\times \mathbb R^m_{>0}$ to 
\begin{equation}
\label{eqn: BPEC params}
\begin{split}
\Theta = \{(\lambda, \omega)\in \mathbb R^{E}\times \mathbb R^m_{>0} :& \lambda_{ij} = \lambda_{k\ell} \textrm{ if $i\to j$, $k\to \ell$ have the same color, and }\\
&\omega_i = \omega_j \textrm{ if $i, j$ have the same color}\}
\end{split}
\end{equation}
In particular, $\mathcal M_{\textrm{Gauss}}(\mathcal G, c) \subseteq \mathcal M_{\textrm{Gauss}}(\mathcal G)$ for every coloring $c$ of the DAG $\mathcal G$. 
% When $\mathbb G$ is the set of all colored DAGs $(\mathcal G, c)$, 
We denote the family of colored Gaussian DAG models as
\[
\mathcal F_{\textrm{c-}\mathbb{DAG}}^{\textrm{Gauss}} = \{ \mathcal M_{\textrm{Gauss}}(\mathcal G, c): (\mathcal G, c)\in \textrm{c-}\mathbb{DAG}\}. 
\]
Note that $F_{\mathbb{DAG}}^{\textrm{Gauss}} \subset \mathcal F_{\textrm{c-}\mathbb{DAG}}^{\textrm{Gauss}}$ since there is always a coloring of $\mathcal G$ in which every edge and vertex is assigned a different color.  

From a causal perspective, the colors have a natural interpretation.  Specifically, two edges $i \to j$ and $k\to \ell$ having the same color can be interpreted as $X_i$ and $X_k$ having similar effect on $X_j$ and $X_\ell$, respectively. 
The vertex colors simply represent that two of the error terms $\varepsilon_i\sim\mathcal N(0,\omega_i)$ have equal variances. 

Several families of colored Gaussian DAG models have appeared in the statistics literature. 
As we will see, this is because different colorings can induce different, desirable solutions to the questions in the graphical models program. 
For instance, Peters and B\"uhlmann \cite{PB14} studied the question of model distinguishability for the family of colored DAG models in which all edges have a \emph{distinct} color and all vertices have the \emph{same} color. 
In statistical language, these are the linear Gaussian DAG models with equal error variances. 
Makam et al. \cite{makam} studied a family of colored Gaussian DAG models called \emph{compatibly colored DAGs} to investigate the minimum sample size needed for the Maximum Likelihood Estimator (MLE) of the model parameters to exist.  

Some families of colored DAG models have been studied for the sake of embracing modeling interpretations of the colors.  
One example is the family of BPEC DAGs. 

\begin{definition}
    \label{def: BPEC DAG}
    A colored DAG $(\mathcal G, c)$ is called a \emph{Blocked Properly Edge-Colored} (BPEC) DAG if 
    \begin{itemize}
        \item all vertices of $\mathcal G$ have a distinct color, 
        \item for every edge $i\to j$ of $\mathcal G$ there is at least one other edge $k\to \ell$ in $\mathcal G$ with the same color, and 
        \item if $i\to j$ and $k\to \ell$ have the same color then $j= \ell$. 
    \end{itemize}
    We let $\mathbb{BPEC}$ denote the set of all BPEC DAGs, and  
    \[
    \mathcal F_{\mathbb{BPEC}} = \{\mathcal M_{\textrm{Gauss}}(\mathcal G, c) : (\mathcal G, c)\in \mathbb{BPEC}\}
    \]
    denote the parametric family of colored Gaussian DAG models for BPEC DAGs. 
\end{definition}
Thinking back to our causal interpretation of the edge colors, BPEC DAGs amount to modeling the hypothesis that, for all variables $X_j$, the direct effects $X_i$, $i\in\pa(j)$, of $X_j$ can be grouped into communities (or blocks) of variables that have similar effects on their common target $X_j$. 
This can be viewed as a causal model analogy of the Stochastic Block Models (SBMs), which are family of random graphs used for modelling community structures in network data (see for instance \cite{karwa}). 
The following depicts a colored DAG representing an equal variance model (left) and a BPEC DAG (right): 

\begin{center}
    \begin{tikzpicture}[thick, scale=0.6]
    \node[circle, fill=red!50, draw, inner sep=1pt, minimum width=1pt] (1) at (0,0)  {$1$};
    \node[circle, fill=red!50, draw, inner sep=1pt, minimum width=1pt] (2) at (2,0) {$2$};
    \node[circle, fill=red!50, draw, inner sep=1pt, minimum width=1pt] (3) at (0,-2) {$3$};
    \node[circle, fill=red!50, draw, inner sep=1pt, minimum width=1pt] (4) at (2,-2) {$4$};
    \node[circle, fill=red!50, draw, inner sep=1pt, minimum width=1pt] (5) at (0,-4) {$5$};
    
    \draw[->, blue!50, very thick] (1) -- (2);
    \draw[->, violet!50, very thick] (1) -- (3);
    \draw[->, cyan!50, very thick] (2) -- (4);
    \draw[<-, teal!50, very thick] (3) -- (4);
    \draw[->, orange!50, very thick] (3) -- (5);

    \node[circle, fill=red!50, draw, inner sep=1pt, minimum width=1pt] (b1) at (0 + 8,0)  {$1$};
    \node[circle, fill=yellow!50, draw, inner sep=1pt, minimum width=1pt] (b2) at (2 + 8,0) {$2$};
    \node[circle, fill=green!50, draw, inner sep=1pt, minimum width=1pt] (b3) at (0 + 8,-2) {$3$};
    \node[circle, fill=orange!50, draw, inner sep=1pt, minimum width=1pt] (b4) at (2 + 8,-2) {$4$};
    \node[circle, fill=cyan!50, draw, inner sep=1pt, minimum width=1pt] (b5) at (0 + 8,-4) {$5$};
    
    \draw[->, blue!50,  very thick] (b1) -- (b3);
    \draw[->, blue!50,  very thick] (b2) -- (b3);
    \draw[->, violet!50, very thick] (b1) -- (b4);
    \draw[->, teal!50, very thick] (b2) -- (b4);
    \draw[->, teal!50,  very thick] (b3) -- (b4);
    \draw[<-, blue!50,  very thick] (b3) -- (b5);
    \draw[->, violet!50,  very thick] (b5) -- (b4);

    \end{tikzpicture}
    \end{center}

\subsubsection{The graphical models program for linear Gaussian DAG models}
\label{subsubsec: GMP linear Gaussian DAGs}
Let's return to the (uncolored) Gaussian DAG model family $\mathcal F_\mathbb{DAG}^{\textrm{Gauss}}$.
The graphical models program for $\mathcal F_\mathbb{DAG}^{\textrm{Gauss}}$ can also be regarded as largely complete.  

\paragraph{\emph{Graphical Models Program Question~1 (Model distinguishability):}}

Just as for CI DAG models, two linear Gaussian DAG models $\mathcal M_{\textrm{Gauss}}(\mathcal G), \mathcal M_{\textrm{Gauss}}(\mathcal H)$ can be equal even if $\mathcal G \neq \mathcal H$. 
In fact, using the modified Cholesky decomposition of the covariance matrix $\Sigma$, one can prove the following.

\begin{proposition}
    \label{prop: Gauss model equivalence}
    Two linear Gaussian DAG models $\mathcal M_{\textrm{Gauss}}(\mathcal G), \mathcal M_{\textrm{Gauss}}(\mathcal H)$ are model equivalent if and only if $\mathcal G$ and $\mathcal H$ are Markov equivalent. 
\end{proposition}

\paragraph{\emph{Graphical Models Program Question~2 (Parameter identifiability and inference):}}

Similar to undirected Gaussian graphical models, the Gaussian DAG models also satisfy global identifiability (i.e. the parameterization map in~\eqref{eqn: trek map} is injective). 
For a given DAG $\mathcal G =([m],E)$, the individual parameter values $(\lambda, \omega)\in\mathbb R^{E}\times\mathbb R^m_{>0}$ specifying $\mathcal N(0,\Sigma)\in\mathcal M_{\textrm{Gauss}}(\mathcal G)$ can be recovered in several ways (see \cite[Theorem 5]{boege}), including
\begin{equation}
\label{eqn: Gaussian DAG identification}
\omega_i = \frac{|\Sigma_{\{i\}\cup\pa(i),\{i\}\cup\pa(i)}|}{|\Sigma_{\pa(i),\pa(i)}|} \qquad \textrm{and} \qquad \lambda_{ki} = \frac{|\Sigma_{\{k\}\cup\pa(i)\setminus\{k\},\{i\}\cup \pa(i)\setminus\{k\}}|}{|\Sigma_{\pa(i),\pa(i)}|},
\end{equation}
for every $i\in[m]$ and $k\to i\in E$. 
Standard parameter estimators, such as the maximum likelihood estimator (MLE), also admit easy-to-use, closed-form expressions.
It is an exercise in basic regression to derive the MLE for these model parameters. 

\paragraph{\emph{Graphical Models Program Question~3 (Model selection):}}

Graphical model selection methods are also well-developed for the linear Gaussian DAG model family. 
Here, the goal is to estimate a DAG $\mathcal G$ such that $\mathcal M_\textrm{Gauss}(\mathcal G)$ is the best fit to the observed data relative to all other models in the family $\mathcal F_\mathbb{DAG}^{\textrm{Gauss}}$. 
When the DAG $\mathcal G$ is interpreted causally, this is a process known as \emph{causal discovery}.  
When no causal assumptions are asserted, it is more generally referred to a \emph{structure learning}. 

One of the most basic structure learning algorithms is the \emph{Greedy Equivalence Search} (GES) \cite{chickering}, and it performs quite well for linear Gaussian DAG models in particular. 
The algorithm, in its most simplistic form, chooses a \emph{score function} that assigns a score to each DAG $\mathcal G\in\mathbb{DAG}$ based on the available data (treated as a random sample).  
A typical score to use is the \emph{Bayesian Information Criterion} (BIC), which is defined by
\[
\textrm{BIC}(\mathcal G; \mathbf{x}) = \ell(\hat\theta \mid \mathbf{x}) - \frac{\ln(n)\dim(\Theta)}{2},
\]
where $\ell(\hat\theta \mid \mathbf{x})$ denotes the log-likelihood function evaluated at the MLE $\hat\theta$ of the model parameters for $\mathcal M_{f,\varepsilon, \Theta}(\mathcal G)\in\mathcal F_{\mathbb{DAG}, \Psi}$, $n$ denotes the size of the sample $\mathbf{x}$ and $\dim(\Theta)$ denotes the number of free parameters. 
It its most basic form, the GES algorithm starts with a DAG with no edges, greedily adds edges to optimize the chosen score, then greedily removes edges to further optimize until no higher score can be achieved. 
The algorithm is known to be consistent under relatively mild assumptions on the choice of score function \cite{chickering}.  
It also highlights one reason we often pass to parametric DAG models, since common score functions often rely on parameter estimators such as the MLE. 

Research into structure learning algorithms for graphical model families is currently a very active field, with numerous new algorithms appearing each year.  
One substantial database of structure learning algorithms is called \texttt{Benchpress} \cite{benchpress}, which is an open-source benchmarking platform for new structure learning algorithms.

\subsubsection{Structural Identifiability}
\label{subsubsec: structural identifiability}

Interestingly, the combinatorial characterization of model equivalence for the linear Gaussian DAG model family $\mathcal F_{\mathbb{DAG}}^{\textrm{Gauss}}$ in Proposition~\ref{prop: Gauss model equivalence} is the \emph{same} as for the CI DAG model family (see Corollary~\ref{cor: VP}). 
It is important to notice that this is a phenomenon specific to the linear Gaussian DAG model family, and it does not extend to parametric DAG model families $\mathcal F_{\mathbb{DAG}, \Psi}$ for other choices of $f$, $\varepsilon$ or even simply changing the allowed parameters $\Theta$. 

In fact, one of the advantages of passing to parametric DAG model families is that the model equivalence classes $\mathcal F_{\mathbb{DAG}, \Psi}$ often \emph{refine} Markov equivalence classes. 
This means that more edges can be fixed in direction based solely on the probabilistic assumptions placed on the distribution, which can be of value in causal modeling. 
In particular, for a well-chosen parameterization $\Psi$ the parametric graphical model family $\mathcal F_{\mathbb G, \Psi}$ will be \emph{structurally identifiable}.

\begin{definition}
    \label{def: structural identifiability}
    Let $\mathcal F_{\mathbb G}$ be a graphical model family for a set of graphs $\mathbb G$. We say that $\mathcal F_{\mathbb G}$ is \emph{structurally identifiable} if almost all distributions in $\mathcal M_{f,\varepsilon, \Theta}(\mathcal G)$ do not belong to $\mathcal M_{f,\varepsilon,\Theta}(\mathcal H)$ for any $\mathcal H\in \mathbb G$ where $\mathcal H \neq \mathcal G$. 
\end{definition}

Here, the phrasing \emph{almost all distributions} is intentionally left vague, since a rigorous meaning of the phrase depends on the choice of distributions constituting the models in $\mathcal F_{\mathbb G}$ (i.e., it depends on the parameterization if the family is parametric). 
For a parametric family, $\mathcal F_{\mathbb G, \Psi}$, the distributions in each $\mathcal M_{f,\Theta}(\mathcal G)\in \mathcal F_{\mathbb G, \Psi}$ can often be represented in some measurable space that allows us to assign $\mathcal M_{f,\Theta}(\mathcal G)$ a dimension (see Section~\ref{subsec: rational realizations} below). 
In this case, \emph{almost all distributions} is rigorously phrased as \emph{all distributions in $\mathcal M_{f,\Theta}(\mathcal G)$ aside from some lower-dimensional subset.}
Hence, a `typical' distribution in $\mathcal M_{f,\Theta}(\mathcal G)$ will belong to no other model in the family $\mathcal F_{\mathbb G, \Psi}$, and $\mathcal G$ is the unique graph in $\mathbb G$ whose edges represent the dependence structure underlying the distribution.

\begin{example}[Undirected Gaussian graphical models]
    \label{ex: UGGMs structurally identifiable}
    It is a nice exercise to show that the family of undirected Gaussian graphical models $\mathcal F_{\mathbb{UG}}^{\textrm{Gauss}}$ discussed in Section~\ref{subsubsec: UPGMs} is structurally identifiable. 
\end{example}

For parametric DAG models, the following choices for $f$, $\varepsilon$ and $\Theta$ have recently been observed to yield structurally identifiable families.

\begin{trailer}{Structurally identifiable parametric DAG model families}
    The parametric graphical model family $\mathcal F_{\mathbb{DAG}, \Psi} =\{\mathcal M_{f,\varepsilon, \Theta}(\mathcal G) : \mathcal G \in \mathbb{DAG}\}$ is structurally identifiable if
    \begin{enumerate}
        \item $f_i$ are linear and $\varepsilon_i$ are non-Gaussian. \hfill (LiNGAM models \cite{shimizu2006})
        \item $f_i$ are linear and $\varepsilon_i\sim \mathcal N(0,\omega)$ are Gaussian with equal error variances. \hfill (\cite{PB14})
    \end{enumerate}
\end{trailer}

The second example here is particularly interesting, since it demonstrates an intriguing phenomenon. 
As discussed in Section~\ref{subsubsec: colored DAGs}, the linear Gaussian DAG model with equal error variances for a DAG $\mathcal G$ is a submodel of the linear Gaussian DAG model $\mathcal M_{\textrm{Gauss}}(\mathcal G)$.  
Proposition~\ref{prop: Gauss model equivalence} tells us that the linear Gaussian DAG model family is \emph{not} structurally identifiable.
However, we see that restricting each model to a submodel by constraining the parameter space yields a structurally identifiable family. 
In particular, given a non-structurally identifiable family $\mathcal F_{\mathbb G, \Psi}$, it can be enough to replace the parameter spaces of each graph $\Theta$ with some alternative $\Theta'$ to produce a structurally identifiable family $\mathcal F_{\mathbb G, \Psi'}$.

Structural identifiability results such as these are desirable for several reasons.  
First, they give the best possible answer to Question~1 of the graphical models program (model distinguishability).  
In addition to this, they can provide cost-saving alternatives for causal modeling.  
The gold standard for learning the causal relations in a system is to, of course, perform Randomized Controlled Trials (RCTs).  
While many companies now have means by which to do this ethically, conducting and coordinating multiple large-scale RCTs still requires a lot of resources (both human and monetary). 
Structurally identifiable models provide a cheap alternative when they can safely be used.  

A proof of structural identifiability (or more generally a characterization of model equivalence) for a given parametric family often produces useful results pertaining to parameter identifiability/estimation and model selection as a byproduct. 
Hence, there are several good reasons to investigate the following general question. 

\begin{questype}{Question 2}
    Given a parametric graphical model family $\mathcal F_{\mathbb{G}, \Psi} =\{\mathcal M_{f, \Theta}(\mathcal G) : \mathcal G \in \mathbb{G}\}$, how can we prove (or disprove) that $\mathcal F_{\mathbb{G}, \Psi}$ is structurally identifiable?
\end{questype}

The goal of the next section is to present the basics of a theory from algebraic geometry that provide one answer to this question.

%---Algebra
\section{Algebraic Implicitization}
\label{sec: algebra}

Parametric graphical model families come in many shapes and sizes. 
As we saw in Section~\ref{subsec: parameteric GMs}, the model equivalence classes for a family $\mathcal F_{\mathbb G, \Psi}$ depend on the choices of $\mathbb G$ and $\Psi$.  
Understanding these equivalence classes is fundamental to working with the family $\mathcal F_{\mathbb G, \Psi}$ since we want to avoid risky misinterpretations when modeling real data. 

From the perspective of causality, where one may want to interpret an edge $i\to j$ as representing a directed effect, the best case scenario is that the family is structurally identifiable (Definition~\ref{def: structural identifiability}). 
As discussed in Section~\ref{subsubsec: structural identifiability}, the families of parametric DAG models that are known to be structurally identifiable are few. 
In particular, we are in need of some general techniques for solving the \emph{model distinguishability problem} (see Section~\ref{sec: intro}). 

Our goal in the remainder of these notes is to describe some emerging techniques for solving the model distinguishability problem for parametric graphical model families.  
Interestingly, these techniques can make use of solutions to the \emph{parameter identifiability problem} (Section~\ref{subsubsec: candidate polynomials}). 
They also give us methodology for addressing the third question in the graphical models program; i.e., \emph{model selection} (Remark~\ref{rem: other uses}). 

The techniques we present here apply to parametric graphical model families that admit a \emph{realization as the image of a rational map}.  
While perhaps somewhat unorthodox, it turns out many parametric graphical models satisfy this condition.  
In particular, we will end our discussions by showing how these techniques, known collectively as \emph{algebraic implicitization}, yield proofs of structural identifiability for families for which no (non-algebraic) proof is known. 

Algebraic implicitization is a fundamental practice within algebraic geometry. 
Much of the purely algebro-geometric content of the following sections is described at a level accessible to a bachelor's student in mathematics in the classic book \emph{Ideals, Varieties and Algorithms} by Cox, Little and O'Shea \cite{cox}.

\subsection{Rational realizations of parametric statistical models}
\label{subsec: rational realizations}
Let's start by formalizing what exactly we mean by a \emph{rational realization}. 
% We will then highlight a few examples. 
To define a rational realization of a statistical model, we need to first define a rational map. 

\begin{definition}
    \label{def: rational map}
    Let $k$ be an infinite field, such as $\mathbb Q, \mathbb R, \mathbb C$.
    A function $\varphi : k^d\setminus W \to k^N$ is a \emph{rational map} if 
    \[
    \varphi(\theta_1,\ldots, \theta_d) = \left(\frac{f_1(\theta_1,\ldots, \theta_d)}{g_1(\theta_1,\ldots, \theta_d)}, \ldots, \frac{f_N(\theta_1,\ldots, \theta_d)}{f_N(\theta_1,\ldots, \theta_d)}\right)
    \]
    for polynomials $f_1(\theta_1,\ldots, \theta_d),\ldots, f_N(\theta_1,\ldots, \theta_d)$,  $g_1(\theta_1,\ldots, \theta_N), \ldots, g_N(\theta_1,\ldots, \theta_d)$ in the variables $\theta_1,\ldots, \theta_d$ with coefficients in $k$, where 
    \[
    W = \{(\theta_1,\ldots, \theta_d)\in \mathbb R^d : w(\theta_1,\ldots, \theta_d) = 0\},
    \]
    for the polynomial $w(\theta_1,\ldots, \theta_d) = \prod_{i=1}^Ng_i(\theta_1,\ldots, \theta_d)$.
\end{definition}

\begin{remark}[Fields and polynomial rings]
    \label{rem: fields}
    In the remainder of these notes, $k$ will always denote an \emph{infinite} field, and $k[z_1,\ldots, z_N]$ denotes the set (formally, \emph{ring}) of all polynomials in the variables $z_1,\ldots, z_N$ with coefficients in $k$. 
    In most cases relevant to statistics, $k = \mathbb R$, but we can also choose $k = \mathbb Q$,  $k = \mathbb C$ or any other infinite field.  
    The most popular choice in classical algebraic geometry is $k = \mathbb C$, since this ensures any polynomial in $k[z_1,\ldots, z_N]$ has a solution in 
    \[
    k^N = \{ (a_1,\ldots, a_N) : a_i \in k \textrm{ for all } i\in [n]\},
    \]
    which we call \emph{$N$-dimensional affine space over $k$}.
    Fortunately, the techniques we outline here will be applicable for the more statistical choice $k = \mathbb R$. 
\end{remark}

Note that Definition~\ref{def: rational map} suggestively denotes points in the domain of $\varphi$ as $(\theta_1,\ldots, \theta_d)$. 
In Section~\ref{subsec: parameteric GMs}, we used $\theta = (\theta_1,\ldots, \theta_d)$ to denote a vector of parameters living in our parameter space $\Theta\subseteq \mathbb R^d$ for the models in $\mathcal F_{\mathbb G, \Psi}$. 
Hence, in our statistical setting, we will use rational maps where the polynomial $w(\theta_1,\ldots, \theta_d)$ does not evaluate to $0$ at any point in our parameter space $\Theta\subseteq \mathbb R^d$. 

It turns out that many parametric statistical models admit a rational realization in the following sense. 

\begin{definition}[Rational realization]
    \label{def: rational realization}
    Let $\mathcal M$ be a statistical model parameterized by $\Theta \subseteq k^d$.
    We say that $\mathcal M$ admits a \emph{rational realization} if there exists a rational map $\varphi: k^d\setminus W \to k^N$ such that
    \begin{itemize}
        \item  $W\cap \Theta = \emptyset$,
        \item for all $\theta\in \Theta$, the image $\varphi(\theta) = (z_1,\ldots, z_N)\in  k^N$ is a vector of alternative parameters defining the same distribution $\mathbb P\in \mathcal M$ as $\theta$, and
        \item each $(z_1,\ldots, z_N)\in\varphi(\Theta)$ specifies a \emph{unique} distribution in $\mathcal M$. 
    \end{itemize}
    In this case, we call $\mathcal M^\ast = \varphi(\Theta)$ a \emph{rational realization} of $\mathcal M$ (according to $\varphi$). 
    If $\mathcal F$ is a family of statistical models, each admitting a rational realization $\mathcal M^\ast$, we let 
    \[
    \mathcal F^\ast = \{\mathcal M^\ast : \mathcal M\in \mathcal F\}.
    \]
\end{definition}

\begin{remark}
    \label{rem: reparameterization}
    If a model $\mathcal M$ admits a rational realization, the map $\varphi$ is simply specifying a reparameterization of the model $\mathcal M$ such that each vector of parameters corresponds to a unique distribution in the model. 
    In practice, we choose this map so that the new set of parameters are advantageous in the sense that the parameters $(z_1,\ldots, z_N)\in \varphi(\Theta)$ are also easier to estimate directly via descriptive statistics.
    An example are the linear Gaussian DAG models, for which the map $\varphi$ is given in~\eqref{eqn: trek map}. 
\end{remark}

When a model admits a rational realization, we can employ tools from geometry to say more about the model. 
For example, we noted in Sections~\ref{subsubsec: UPGMs} and~\ref{subsubsec: parametric DAGs} that both the undirected Gaussian graphical model and Gaussian DAG model families satisfy a useful probability called \emph{global identifiability}.  
It is easy to precisely define this property in the language of rational realizations. 
\begin{definition}[Global rational identifiability]
    \label{def: global identifiability}
    Let $\mathcal M$ be a statistical model parameterized by $\Theta\subseteq k^d$ with rational realization $\mathcal M^\ast = \varphi(\Theta)\subseteq k^N$. 
    We say that
    \begin{enumerate}
        \item $\mathcal M$ satisfies \emph{global (parameter) identifiability} if the map $\varphi$ is injective on $\Theta$. 
        \item $\mathcal M$ satisfies \emph{global rational identifiability} if there is a rational function $\psi: k^N\setminus W \to k^d$ such that $\psi(\varphi(\theta)) = \theta$ for all $\theta\in\Theta$. 
    \end{enumerate}
\end{definition}

Let's look at some examples of models with rational realizations. 
\begin{figure}[t]
\begin{center}
\includegraphics[width=0.5\textwidth]{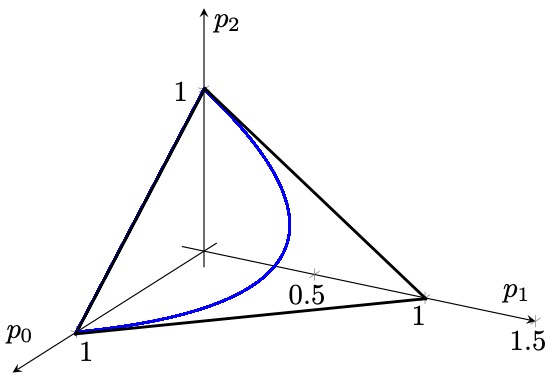}
\end{center}
\caption{The rational realization of the statistical model $\{\textrm{Bin}(2,\theta): \theta\in(0,1)\}$. The boundary of the open probability simplex $\Delta_{\{0,1,2\}}^\circ$ is drawn in black, and the points in the model are drawn in blue.}
\label{fig: binomial model}       
\end{figure}

\begin{example}
    \label{ex: bindist}
    Let $X\sim \textrm{Bin}(2, \theta)$ be a binomial distributed random variable with $\theta \in \Theta = (0,1)$ and sample space $\mathcal X = \{0,1,2\}$.
    The mass function $f_X(x) = \binom{2}{x}\theta^x(1 - \theta)^{2- x}$ for $x\in\mathcal X$ can be encoded in a vector $f_\theta = (f_X(0), f_X(1), f_X(2))\in \mathbb R^3$. 
    The vector $f_\theta$ clearly parameterizes the distribution for $X$ since $\textrm{Pr}(X = i) = f_\theta(i)$. 
    Moreover, $f_\theta$ lives in the open probability simplex 
    \[
    \Delta^\circ_\mathcal X = \{p = (p_0,p_1,p_2)\in \mathbb R^3 : p_0 + p_1 + p_2 =1, p_i>0 \textrm{ for all } i\in \{0,1,2\}\}. 
    \]
    In a classic modeling scenario, we may assume our data-generating distribution to be modeled by a binomial distribution with $n = 2$ and $\theta$ unknown.  
    Hence, we are interested in the distributions for $X$ for which $f_\theta$ belongs to the image $\varphi(\Theta)$ of the rational map
    \[
    \varphi: \mathbb R \to \mathbb R^3; \qquad \varphi(\theta) = ((1 - \theta)^2, 2\theta(1-\theta), \theta^2). 
    \]
    In particular, $\varphi(\Theta)$ is a rational realization of the binomial family $\{\textrm{Bin}(2, \theta) : \theta \in \Theta\}$. 
    Figure~\ref{fig: binomial model} depicts $\varphi(\Theta)$ as the blue curve living inside the probability simplex $\Delta_{\{0,1,2\}}^\circ$. 
\end{example}

Our next examples are families of parametric graphical models. % 

\begin{example}
    \label{ex: gauss UGM geometric}
    Consider the undirected Gaussian graphical model $\mathcal M_{\textrm{Gauss}}(\mathcal G)$ for $\mathcal G = ([m], E)\in\mathbb{UG}$ from Definition~\ref{def: undirected gaussian GM}.  
    Here, our parameter space is the set of concentration matrices $\mathcal L_\mathcal G\subset \textrm{Sym}(\mathbb R^m)$ defined in~\eqref{eqn: UG concentration space}.
    We saw in Remark~\ref{rem: geometric parametrization UGG} that 
    \[
    \mathcal M_{\textrm{Gauss}}(\mathcal G) = \{\mathcal N(0, \Sigma) : \Sigma \in \varphi_\mathcal G(\mathcal L_\mathcal G)\}, 
    \]
    where $\varphi_\mathcal G$ is the rational map given in~\eqref{eqn: UG cov param}.  Note that $\varphi_\mathcal G$ is defined on $\mathcal L_G$ since $|K|\neq 0$ for all $K\in \mathcal L_\mathcal G$, a set of positive definite matrices.
    Since every distribution in $\mathcal M_{\textrm{Gauss}}(\mathcal G)$ is determined by its covariance matrix, the set $\mathcal{M}^\ast_{\textrm{Gauss}}(\mathcal G) = \varphi_\mathcal G(\Theta)\subset \textrm{Sym}(\mathbb R^m)$ is a rational realization of the model $\mathcal M(\mathcal G)$. 
\end{example}

\begin{example}
    \label{ex: gauss DAG geometric}
    Consider the Gaussian DAG model $\mathcal M_{\textrm{Gauss}}(\mathcal G)$ for $\mathcal G = ([m], E)\in \mathbb{DAG}$ from Definition~\ref{def: linear Gaussian DAG models}.  
    In this case, our parameter space is $\mathbb R^{E}\times \mathbb R^m_{>0}$ (by~\eqref{eqn: Gauss DAG params}). 
    According to Remark~\ref{rem: trek map}, we have that 
    \[
    \mathcal M_{\textrm{Gauss}}(\mathcal G) = \{\mathcal N(0,\Sigma) : \Sigma \in \varphi_\mathcal G(\mathbb R^{E}\times \mathbb R^m_{>0})\}, 
    \]
    where $\varphi_\mathcal G$ is the rational map in~\eqref{eqn: trek map} arising from the trek rule in Proposition~\ref{prop: trek rule}. 
    Hence, $\mathcal{M}^\ast_{\textrm{Gauss}}(\mathcal G) = \varphi_\mathcal G(\mathbb R^{E}\times \mathbb R^m_{>0})$ is a rational realization of $\mathcal M_{\textrm{Gauss}}(\mathcal G)$.
\end{example}

\begin{remark}
    \label{rem: it works}
Note that in each of these three examples, the rational map $\varphi$ gives us the desired properties of our rational realization stated in Remark~\ref{rem: reparameterization}.  In Example~\ref{ex: bindist}, the reparameterization in terms of the mass function vector $f_\theta = (p_0,p_1,p_2)$ is amenable to direct estimation via descriptive statistics by simply counting how many observations in the sample are equal to $0,1$ and $2$, respectively. In Examples~\ref{ex: gauss UGM geometric} and~\ref{ex: gauss DAG geometric}, the map $\varphi$ reparameterizes the distribution in terms of the covariance matrix, which is easily estimated via the sample covariance matrix.  Moreover, in all three examples, the reparameterization $\varphi(\theta)$ for $\mathbb P$ in the model specifies $\mathbb P$ uniquely. 
\end{remark}

\subsection{Polynomial constraints for parametric models}
\label{subsec: polynomial constraints}
For a family of parametric graphical models with rational realizations $\mathcal F^\ast_{\mathbb G, \Psi}$, solving the model distinguishability problem amounts to characterizing the graphs $\mathcal G\neq \mathcal H$ for which $\mathcal{M}^\ast_{f,\Theta}(\mathcal G)$ and $\mathcal{M}^\ast_{f, \Theta}(\mathcal H)$ are non-equal subsets of $k^N$. 
This can be a difficult task using only the parameterization of the models.  
We need is a way to describe these models that makes it easier to see when they are different. 
The idea of \emph{algebraic implicitization} is to find a description of each rational realization $\mathcal{M}^\ast_{f,\Theta}(\mathcal G)\in \mathcal F_{\mathbb G, \Psi}^\ast$ as the solution set to a system of equations.
Manipulating these equations, instead of a parameterization, can make the model distinguishability task more manageable. 

To describe this method, we use the language of algebraic geometry, where the fundamental objects of study are called \emph{(affine) varieties}. 
An affine variety is simply the solution set to a system of polynomial equations, generalizing the familiar linear spaces studied in linear algebra.

\begin{definition}
    \label{def: affine variety}
    The \emph{(affine) variety} generated by $f_1,\ldots, f_s \in k[z_1,\ldots, z_N]$ is the set
    \[
    V(f_1,\ldots, f_s) = \{ a = (a_1,\ldots, a_N) \in k^N : f_i(a) = 0 \textrm{ for all } i\in[s]\}. 
    \]
\end{definition}

A simple (non-statistical) example can help us to see how affine varieties may be of use for our statistical problem. 
The following example comes from \cite{cox}.

\begin{example}
\label{ex: circle}
    Consider the rational map $\varphi: \mathbb R \to \mathbb R^2$ where
    \[
    \varphi(\theta) = \left( \frac{1 - \theta^2}{1 + \theta^2}, \frac{2\theta}{1 + \theta^2}\right) = (x,y)\in \mathbb R^2. 
    \]
    It may not be immediately clear, but as we let $\theta$ vary over $\mathbb R$ this parameterization traces out the unit cirle in $\mathbb R^2$ minus the point $(0,-1)$. 
    \begin{center}
        \includegraphics[width=0.5\textwidth]{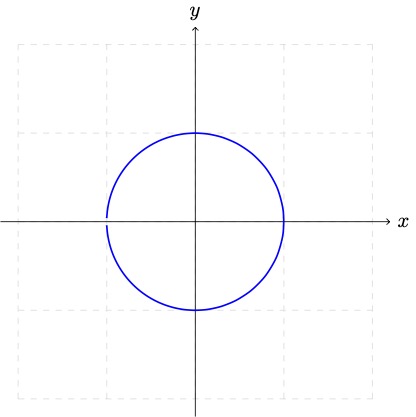}
    \end{center}
    Hence, almost every point in the image $\varphi(\mathbb R)$ is a solution to the polynomial equation $x^2 + y^2 = 1$. 
    In other words, the affine variety $V(x^2 + y^2 - 1)$ minus the point $(0,-1)$ is an alternative (implicit) representation of $\varphi(\mathbb R)$. 
    Since, a typical point in $\varphi(\mathbb R)$ belongs to the affine variety $V(x^2 + y^2 - 1)$, it is easier to describe the set $\varphi(\mathbb R)$ as the solution set to the equation $x^2 + y^2 = 1$ (with the understanding that we don't care about the atypical solution $(0,-1)$). 
\end{example}

Our goal is to apply the same logic as in Example~\ref{ex: circle} to $\mathcal{M}^\ast_{f,\Theta}(\mathcal G) \in \mathcal F_{\mathbb G, \Psi}^\ast$. 
While it may be hard to distinguish the sets $\mathcal{M}^\ast_{f,\Theta}(\mathcal G)$ and $\mathcal{M}^\ast_{f, \Theta}(\mathcal H)$ using only their parameterizations, we may have more luck using implicit representations $V(f_1,\ldots, f_s)$ and $V(g_1,\ldots, g_t)$ since they will tell us constraints $f_1(z) = 0,\ldots, f_s(z) = 0$ and $g_1(z) = 0,\ldots, g_t(z) = 0$ that, respectively, define the models. 

\begin{remark}
    \label{rem: polynomial constraints}
    In Section~\ref{subsec: CI models}, we used conditional independence constraints to describe CI graphical models, Here, the plan is to identify and use \emph{polynomial} constraints instead. 
\end{remark}

In Example~\ref{ex: circle}, the polynomial constraint $x^2 + y^2 -1$ defining the image $\varphi(\mathbb R)$ has the property that its affine variety is the \emph{smallest} affine variety containing the image of the parameter space $\mathbb R$ under the rational map $\varphi$.  
Only a few atypical points were added to obtain the full variety $V(x^2 + y^2 - 1)\supset \varphi(\mathbb R)$.  
To generalize this to our statistical models, we consider the \emph{implicitization problem} in algebraic geometry: 

\begin{trailer}{The Implicitization Problem}
    Let $k$ be an infinite field, such as $\mathbb Q, \mathbb R$ or $\mathbb C$.  For a rational map 
    \[
    \varphi: k^d \setminus W\to k^N; \qquad \varphi(\theta) = \left(\frac{f_1(\theta)}{g_1(\theta)}, \ldots, \frac{f_N(\theta)}{g_N(\theta)}\right) = (z_1,\ldots, z_N),
    \]
    let $\overline{\varphi(k^d\setminus W)}$ denote the smallest affine variety containing $\varphi(k^d\setminus W)$. 
    Find polynomials $h_1,\ldots, h_s \in k[z_1,\ldots, z_N]$ such that $V(h_1,\ldots, h_s) = \overline{\varphi(k^d \setminus W)}$.
\end{trailer}

\begin{remark}[Assuming the model is Zariski dense]
    \label{rem: model dimension and Zariski closure}
    In actuality, we are interested in finding $h_1,\ldots, h_s \in k[z_1,\ldots, z_N]$ such that $V(h_1,\ldots, h_s)$ is the smallest affine variety containing $\mathcal{M}^\ast = \varphi (\Theta)$, which can be a proper subset of $\varphi(k^d\setminus W)$. 
    However, if $\mathcal M^\ast$ and $\overline{\varphi(k^d\setminus \Theta)}$ have the same dimension, then $\overline{\varphi(k^d\setminus W)}$ \emph{is} the smallest variety containing $\mathcal M^\ast$.  
    Hence, it is enough to answer the implicitization problem for $\varphi(k^d\setminus W)$ to obtain polynomial constraints $h_1(z) = 0, \ldots, h_s(z) = 0$ that define our statistical model $\mathcal M$. 
    Many parametric graphical models satisfy this dimension condition, which, in the language of algebraic geometry, implies that $\mathcal M^\ast$ is \emph{Zariski dense} in $\overline{\varphi(k^d\setminus W)}$.
\end{remark}

\begin{definition}
    \label{def: dense family} 
    A statistical model $\mathcal M$ is \emph{dense} if $\mathcal M^\ast = \varphi(\Theta)$ has the same dimension as $\overline{\varphi(k^d\setminus W)}$.  A family $\mathcal F$ is \emph{dense} if every $\mathcal M\in\mathcal F$ is dense. 
\end{definition}

For a dense statistical model $\mathcal M$, we are interested in finding polynomials $h_1,\ldots, h_s$ that solve the implicitization problem. 
A good place to start is to consider \emph{all} the polynomials that evaluate to $0$ on $\mathcal{M}^\ast$.  

\begin{definition}
    \label{def: vanishing ideal}
    Let $\mathcal{M}^\ast\subseteq k^N$. 
    The \emph{vanishing ideal} of $\mathcal{M}^\ast$ is the set of polynomials
    \[
    I(\mathcal M^\ast) = \{ f\in k[z_1,\ldots, z_N] : f(a) = 0 \textrm{ for all } a\in \mathcal M^\ast\}. 
    \]
\end{definition}

The vanishing ideal is indeed an \emph{ideal} in the polynomial ring $k[z_1,\ldots, z_N]$; meaning it is a collection of polynomials $I\subseteq k[z_1,\ldots, z_N]$ satisfying:
\begin{enumerate}
    \item $0 \in I$, 
    \item if $g\in I$ and $h \in k[z_1,\ldots, z_N]$ then $hg\in I$, and 
    \item if $g,h\in I$ then $g + h\in I$. 
\end{enumerate}
For readers unfamiliar with algebra, it is a nice exercise to verify that the vanishing ideal $I(\mathcal M^\ast)$ is an ideal according to this definition. 

The vanishing ideal of the set $\mathcal M^\ast$ is useful since it helps us translate the problem of finding the smallest affine variety containing $\mathcal M^\ast$ (a geometric problem), into a language where we work only with polynomials -- which are easier to manipulate algorithmically. 
In particular, we have the following lemma. 
\begin{lemma}
    \label{zariski closure}
    Let $\mathcal M^\ast\subseteq k^N$ for some infinite field $k$. The subset
    \[
    \overline{\mathcal M^\ast} = V(I(\mathcal M^\ast)) = \{a\in k^N : f(a) = 0 \textrm{ for all } f\in \mathcal M^\ast\}
    \]
    is the smallest affine variety containing $\mathcal M^\ast$.  
\end{lemma}

When $\mathcal M$ is dense, we have that $\overline{\mathcal M^\ast} = \overline{\varphi(k^d\setminus \Theta)}$, so we are looking for polynomials $h_1,\ldots, h_s\in k[z_1,\ldots, z_N]$ such that  $V(h_1,\ldots, h_s) = V(I(\mathcal M^\ast))$.
A keen reader may notice that, by our definition of affine variety, such polynomials may not even exist!  
This is because the vanishing ideal $I(\mathcal M^\ast)$ is an \emph{infinite} set of polynomials in $k[z_1,\ldots, z_N]$, and we have defined affine varieties to be sets specified by finitely many polynomials $h_1,\ldots, h_s$. 
Fortunately, the Hilbert Basis Theorem resolves this issue. 

\begin{theorem}[Hilbert Basis Theorem]
    For every ideal $I \subseteq k[z_1,\ldots, z_N]$ there exist polynomials $h_1,\ldots, h_s\in k[z_1,\ldots, z_N]$ such that 
    \[
    I = \left\{\sum_{i=1}^sg_ih_i : g_1,\ldots, g_s\in k[z_1,\ldots, z_N]\right\}. 
    \]
    The (finite) set of polynomials $\{h_1,\ldots, h_s\}$ is called a \emph{basis} for $I$, and we write 
    \[
    I = \langle h_1,\ldots, h_s\rangle.
    \]
\end{theorem}

In other words, our main goal is to find a basis for the vanishing ideal $I(\mathcal M^\ast)$; i.e., polynomials $h_1,\ldots, h_s$ such that $\langle h_1,\ldots, h_s\rangle = I(\mathcal M^\ast)$. 
This makes our problem completely algebraic, and hence easier to attack algorithmically.

\begin{trailer}{Our Goal: Find polynomial constraints defining a statistical model}
    \label{goal}
    Let $\mathcal M$ be a dense statistical model. We want to find a basis $\{h_1,\ldots, h_s\}\subset k[z_1,\ldots, z_N]$ for the vanishing ideal $I(\mathcal M^\ast)$ since these polynomials give us constraints $h_1(z) = 0, \ldots, h_s(z) = 0$ that define the model $\mathcal M$. 
\end{trailer}
Let's have a look at a couple examples to see why this makes sense. 

\begin{example}
    \label{ex: binomial model revisited}
    Consider the family of binomial distributions $\{\textrm{Bin}(2,\theta) : \theta\in (0,1)\}$ from Example~\ref{ex: bindist}, with rational realization $\mathcal M^\ast$ depicted in Figure~\ref{fig: binomial model}.  
    This model has vanishing ideal generated by two polynomials in the polynomial ring $\mathbb R[p_0,p_1,p_2]$:
    \[
    I(\mathcal M^\ast) = \langle p_0 +p_1 + p_2 -1, p_1^2 - 4p_0p_2\rangle. 
    \]
    As noted in Example~\ref{ex: bindist}, the variable $p_i$ records the probability $\textrm{Pr}(X = i)$, and $\mathcal M^\ast$ is a subset of the open probability simplex $\Delta_\mathcal X^\circ\subset \mathbb R^3$. 
    Hence, the first generator $p_0 + p_1 + p_2 -1$ is simply a constraint on $\mathcal M^\ast$ that records the fact that the sum of the entries of the mass function is equal to $1$.  
    The second constraint is specific to the model.
    It selects exactly which points in the probability simplex correspond to distributions in the model.

    The only other constraints beyond $p_0 +p_1 + p_2 -1 = 0$ and $p_1^2 - 4p_0p_2 = 0$ needed to define $\mathcal M^\ast$ are the inequalities $p_0,p_1,p_2>0$, which are simply stating that the probabilities $\textrm{Pr}(X = i)$ are all positive. 
    Hence, a complete set of constraints defining the model $\mathcal M^\ast$ is
    \[
    p_0 +p_1 + p_2  = 1,\quad  p_1^2 - 4p_0p_2 = 0,\quad  p_1>0,\quad  p_2>0,\quad  p_3>0.
    \]
\end{example}

\begin{example}
    \label{ex: UG path example}
    Let us now consider the vanishing ideal for an undirected Gaussian graphical model $\mathcal M_{\textrm{Gauss}}(\mathcal G)$ where $\mathcal G$ is the following graph on $3$ nodes: 
    \begin{center}
    \begin{tikzpicture}[thick, scale=0.6]
    \node[circle, draw, inner sep=1pt, minimum width=1pt] (1u) at (0,0)  {$1$};
    \node[circle, draw, inner sep=1pt, minimum width=1pt] (2u) at (2,0.24) {$2$};
    \node[circle, draw, inner sep=1pt, minimum width=1pt] (3u) at (4,-0.2) {$3$};
    
    \draw[-, very thick] (1u) -- (2u);
    \draw[-, very thick] (2u) -- (3u);
    
    \end{tikzpicture}
    \end{center}
    In this case, the rational realization of our model consists of $3\times 3$ covariance matrices $\Sigma = (\sigma_{ij})_{i,j = 1}^3$ (see Example~\ref{ex: gauss UGM geometric}). 
    Hence, our polynomial constraints should live in the ring $\mathbb R[\sigma_{11}, \sigma_{22}, \sigma_{33}, \sigma_{12}, \sigma_{13}, \sigma_{23}]$. 
    The vanishing ideal for this model has only a single generator
    \begin{equation}
    \label{eqn: UG vanishing ideal}
    I(\mathcal M_{\textrm{Gauss}}^\ast(\mathcal G)) = \langle\sigma_{12}\sigma_{23} - \sigma_{22}\sigma_{13}\rangle.
    \end{equation}
    This generator has a familiar statistical interpretation.  
    Namely, it is the determinant of the submatrix $|\Sigma_{\{1,2\}, \{2,3\}}|$. 
    With the help of the well-known fact: 
    \begin{equation}\label{eqn: Gauss CI}
    \textrm{If} \quad X\sim\mathcal{N}(0,\Sigma) \quad \textrm{then} \quad X_A\independent X_B \mid X_C \Longleftrightarrow |\Sigma_{A\cup C, B\cup C}| = 0,
    \end{equation}
    we see that the single constraint in our basis is capturing the global Markov property for the undirected graph $\mathcal G$; i.e., the CI relation $X_1 \independent X_3 \mid X_2$. 
    The only other constraints on our model are the assumptions that $\Sigma$ is a positive definite matrix, which are algebraically encoded via the polynomial inequalities asserting that the principle minors of $\Sigma$ are all positive.  
    Hence, a complete set of constraints defining the model $\mathcal M_{\textrm{Gauss}}(\mathcal G)$ is
    \[
    \sigma_{12}\sigma_{23} - \sigma_{22}\sigma_{13} = 0,\quad \sigma_{11} > 0,\quad \sigma_{11}\sigma_{22} - \sigma_{12}^2> 0,\quad |\Sigma|>0. 
    \]
\end{example}

\begin{remark}
    \label{rem: vanishing ideals for distinguishing models}
    A set of polynomial equalities and inequalities that define our model is called an \emph{implicit description}. 
    One motivation for computing implicit descriptions is to use the constraints as a tool for solving the model distinguishability problem for parametric graphical models. 
    In  Example~\ref{ex: UG path example}, we see that the rational realization of \emph{any} Gaussian graphical model will include the same inequalities (since they are simply asserting that the covariance matrix is positive definite).  
    Hence, it is enough to look at the constraints in the basis of the vanishing ideals (so long as the specificities of the graphical model family are not imposing any additional polynomial \emph{inequalities}). 
\end{remark}

\begin{remark}[Polynomial constraints as tools for model selection]
    \label{rem: other uses}
    Example~\ref{ex: UG path example} reveals one way the polynomials obtained via algebraic implicitization can be used in the model selection problem of the graphical models program.  
    To test if the DAG model in Example~\ref{ex: UG path example} is a good fit to the data, we can perform a statistical hypothesis test for the null hypothesis $H_0: X_1 \independent X_3 \mid X_2$. 
    We see from this example that this amounts to deciding if we can reject the hypothesis that the sample covariance matrix evaluates to $0$ on the polynomial $\sigma_{12}\sigma_{23} - \sigma_{22}\sigma_{13}$. 
    More generally, the polynomials in the vanishing ideal of a rational realization of a statistical model translate into hypothesis tests for model selection using U-statistics (see \cite{sturma}). 
\end{remark}

Given these examples, the natural question to ask is, \emph{``How do we actually compute a basis for a vanishing ideal of a statistical model?''}

\subsection{Gröbner bases for computing model-defining constraints}
\label{subsubsec: gröbner bases}

In this section, we present a computational method for finding a basis of the vanishing ideal of a statistical model.  
While the computational methods only return a basis for a single model $\mathcal M_{f, \Theta}(\mathcal G)$, it is a valuable tool for gaining an understanding of what constraints may hold more generally for the models in the family $\mathcal F_{\mathbb G, \Psi}$. 
It is also a tool for finding constraints to be used in statistical hypothesis tests for model fit, as described in Remark~\ref{rem: other uses}. 

Suppose $\mathcal M$ is a dense statistical model with parameter space $\Theta$ and rational realization given by the map
\begin{equation}
\label{eqn: elimination map}
\varphi: k^d\setminus W \to k^N; \qquad \varphi(\theta) = \left(\frac{f_1(\theta)}{g_1(\theta)},\ldots, \frac{f_N(\theta)}{g_N(\theta)}\right) = (z_1,\ldots, z_N). 
\end{equation}
We are interested in explicitly computing a basis for the vanishing ideal $I(\mathcal M^\ast)$ or, equivalently, a basis for $I(\varphi(k^d\setminus W))$. 
Fortunately, we have a good starting point, since we have a basis for a very-much related ideal: 
\[
I_\varphi = \langle g_1(\theta)z_1 - f_1(\theta), \ldots, g_N(\theta)z_N - f_N(\theta), g_1(\theta)\cdots g_N(\theta)y - 1\rangle.
\]
The variety $V(I_\varphi)$ is the \emph{graph} of the parameterization $\varphi$, since it is generated by the set of polynomials 
encoding how we obtain the values of each $x_i$ for a given $\theta \in k^d\setminus W$: 
\[
x_i = \frac{f_i(\theta)}{g_i(\theta)}, \qquad i = 1,\ldots, N,
\]
The ideal has one additional generator, $g_1(\theta)\cdots g_N(\theta)y - 1$, which ensures that our denominators $g_1(\theta),\ldots, g_N(\theta)$ aren't tossing in any additional zeros that we don't want.  This works since it corresponds to the rational function
\[
y = \frac{1}{g_1(\theta)\cdots g_N(\theta)}, 
\]
which cannot have any solutions with $y = 0$.

\begin{example}
    \label{ex: binomial graph ideal}
    For the binomial model in Example~\ref{ex: bindist} parameterized by the rational map
    \[
    \varphi: \mathbb R \to \mathbb R^2; \qquad \varphi(\theta) = ((1 - \theta)^2, 2\theta(1 - \theta), \theta^2) = (p_0,p_1,p_2), 
    \]
    we obtain the ideal 
    \[
    I_\varphi = \langle p_0 - (1 - \theta)^2, p_1 - 2\theta(1 - \theta), p_2 - \theta^2 \rangle \subset \mathbb R[p_0,p_1,p_2]. 
    \]
\end{example}

Notice that a point $(\theta_1,\ldots, \theta_d,z_1,\ldots, z_N)\in V(I_\varphi)$ if and only if $(z_1,\ldots, z_N)\in \mathcal M^\ast$ and $(z_1,\ldots, z_N) = \varphi(\theta_1,\ldots, \theta_d)$. 
So, in a sense, this is \emph{almost} our vanishing ideal $I(\mathcal M^\ast)$. 
The main problem is that the ideal $I_\varphi$ lives in a polynomial ring with too many variables:  $k[\theta_1,\ldots, \theta_d, y, z_1,\ldots, z_N]$, while our vanishing ideal lives in $k[z_1,\ldots, z_N]$. 
So we need to \emph{eliminate} the extra variables $\theta_1,\ldots, \theta_d, y$.

It turns out that if we eliminate these extra variables carefully, the result will be exactly the vanishing ideal $I(\mathcal M^\ast)$. 
To do this, we need to change the basis for the ideal $I_\varphi$, to a special type of basis called a \emph{Gröbner basis}.  

A Gröbner basis for an ideal $I\subset k[z_1,\ldots, z_N]$ is always defined with respect to an ordering of the monomials in the polynomial ring $k[z_1,\ldots, z_N]$, called a term order. 

\begin{definition}
    Let $k[z_1,\ldots, z_N]$ be a polynomial ring over a field $k$. A \emph{term order} $\succ$ on $k[z_1,\ldots, z_N]$ is a total ordering of the monomials $z^a = z_1^{a_1}\cdots z_N^{a_N}$ for $a = (a_1,\ldots, a_N) \in \mathbb Z_{\geq 0}^N$ that satisfies
    \begin{enumerate}
        \item $z^a \prec z^b$ implies that $z^a z^c \prec z^b z^c$ for all $c\in \mathbb Z_{\geq 0}^N$, and 
        \item $1 = z^0 \prec z^{a}$ for all $a\in \mathbb Z_{\geq 0}^N$. 
    \end{enumerate}
\end{definition}

There are many term orders to choose from. 
A common example is the \emph{lexicographic term order}, which asserts that $z^{a} \prec z^b$ if and only if the left-most nonzero entry in the vector $b - a$ is positive.  For instance, if we take $z_1 \succ z_2 \succ z_3$ then 
\[
z_1^3 \succ z_1^2z_2^3 \succ z_1^2z_2^2\succ z_2 \succ z_3^3.
\]

Notice that every polynomial $f\in k[z_1,\ldots, z_N]$ has a largest term with respect to a given term order $\succ$. 

\begin{definition}
    \label{def: leading term and initial ideal}
    Let $f = \sum_{a\in \mathbb Z_{\geq 0}^N}c_az^{a}\in k[z_1,\ldots, z_N]$ and $\succ$ be a term order on $k[z_1,\ldots, z_N]$. 
    The \emph{initial term} of $f$ with respect to $\succ$ is the term 
    \[
    \textrm{in}_\succ(f) := c_az^{a},
    \]
    for which $z^{a}\succ z^b$ for all other $z^b$ having $c_b\neq 0$ in $f$. 
\end{definition}

For example, the polynomial $f = z_1z_2 + z_2^2$ has initial term $\textrm{in}_\succ(f) = z_1z_2$ with respect to the lexicographic term order described above. 

A Gröbner basis for an ideal is a special type of basis that is detected using the initial terms of its polynomials. 

\begin{definition}
    \label{def: GB}
    Let $I \subseteq k[z_1,\ldots, z_N]$ be an ideal, and let $\succ$ be a term order on $k[z_1,\ldots, z_N]$. A \emph{Gröbner basis} for $I$ with respect to $\succ$ is a finite set $G = \{g_1,\ldots, g_s\}\subset I$ satisfying
    \[
    \langle \textrm{in}_\succ(f) : f\in I\rangle = \langle \textrm{in}_\succ(g_1), \ldots, \textrm{in}_\succ(g_s)\rangle. 
    \]
\end{definition}
That is,  a basis for $I$ is a Gröbner basis with respect to $\succ$ if the ideal generated by the initial terms of the polynomials in the basis equals the ideal generated by initial terms of \emph{all} polynomials in the ideal. 
While it is not entirely obvious, an important observation is that Gröbner bases always exist. 

\begin{lemma}
    Every ideal $I\subseteq k[z_1,\ldots, z_N]$ has a Gröbner basis with respect to any term order $\succ$ on the polynomial ring $k[z_1,\ldots, z_N]$. 
\end{lemma}

Aside, from existence, there are algorithms for computing Gröbner bases.  
The classical algorithm is known as Buchberger's Algorithm. 
We skip the details of these algorithms here, but a very accessible and detailed discussion of these methods is given in \cite{cox}. 
% The same reference also provides a completely thorough and rigorous development of the algebra and geometry we discussed above. 
We note, however, that there is easy-to-use software for computing Gröbner bases for yourself. 
One language that can be used online is called \texttt{Macaulay2} \cite{M2}  (\url{https://www.unimelb-macaulay2.cloud.edu.au/#home}). 
The following code snippet is an example of computing a Gröbner basis for the ideal $I_\varphi$ for the graphical model in Example~\ref{ex: binomial graph ideal}. 

\begin{programcode}{Computing a Gröbner Basis for $I_\varphi$ in Macaulay2}
\begin{verbatim}
    S = QQ[t, p0, p1, p2, MonomialOrder => Lex]; 
    I = ideal(p0 - (1 - t)^2, p1 - t*(1 - t), p2 - t^2);
    gens gb I
\end{verbatim}
\end{programcode}

\begin{example}
    \label{ex: binomial M2}
The output of the code above is the Gröbner basis for the ideal $I_\varphi$ in Example~\ref{ex: binomial graph ideal} with respect to the lexicographic term order on $\mathbb Q[\theta, p_0,p_1,p_2]$. Since all polynomials generating the ideal $I_\varphi$ have only rational coefficients, this Gröbner basis is also the Gröbner basis for the same ideal in $\mathbb R[\theta,p_0,p_1,p_2]$.   It looks like this
\[
p_1^2 + 4p_1p_2 + 4p_2^2 - 4p_2, \quad p_0 + p_1 + p_2 - 1, \quad 2\theta - p_1 - 2p_2.
\]
Note that exactly one polynomial in this Gröbner basis for $I_\varphi \subseteq \mathbb R[\theta, p_0, p_1, p_2]$ uses the variable $\theta$.  
If we simply forget this polynomial we can generate an ideal in the polynomial ring $\mathbb R[p_0, p_1, p_2]$: 
\[
I_1 = \langle p_1^2 + 4p_1p_2 + 4p_2^2 - 4p_2, p_0 + p_1 + p_2 - 1\rangle \subseteq \mathbb R[p_0, p_1, p_2]. 
\]
Now let
\begin{equation}
    \begin{split}
        h_1 &=p_1^2 + 4p_1p_2 + 4p_2^2 - 4p_2, \\
        h_2 &= p_0 + p_1 + p_2 - 1
    \end{split}
\end{equation}
denote our generators for the ideal $I_1$, and
\begin{equation}
    \begin{split}
        g_1 &=p_1^2 -4p_0p_2, \\
        g_2 &= p_0 + p_1 + p_2 - 1
    \end{split}
\end{equation}
denote the generators for the vanishing ideal of the model $I(\mathcal M^\ast)$ from Example~\ref{ex: binomial model revisited}.
We see that
\[
g_1 = h_1 - 4p_2h_2, \qquad h_2 = g_2 \qquad h_1 = g_1 + 4p_2g_2. 
\]
The first two equations show that $I_1 \subseteq \langle h_1, h_2\rangle =I(\mathcal M^\ast)$, while the last two equations show that $I(\mathcal M^\ast)\subseteq \langle g_1, g_2\rangle = I_1$.
Hence, $I_1 = I(\mathcal M^\ast)$, and we have recovered the vanishing ideal of the model by simply forgetting the polynomials in our Gröbner basis for $I_\varphi$ that use the variables we do not want. 
\end{example}

This technique generalizes. 

\begin{definition}
    \label{def: elimination ideal}
    Let $I$ be an ideal in the polynomial ring $k[z_1,\ldots, z_N]$. 
    For $\ell\in[N]$, the \emph{$\ell$-th elimination ideal} of $I$ is the ideal 
    \[
    I_\ell = I\cap k[z_{\ell +1}, \ldots, z_N].
    \]
\end{definition}
Since an ideal $I\subseteq k[z_1,\ldots, z_N]$ is a set of polynomials, this definition is simply saying that the $\ell$-th elimination ideal is the subset of polynomials in $I$ using only the variables $z_{\ell + 1}, \ldots, z_N$.  
It is an exercise to show that $I_\ell$ is indeed an ideal in the polynomial ring $k[z_{\ell +1}, \ldots, z_N]$.
However, it is not clear that a basis for the ideal $I_\ell$ can be obtained from an arbitrary basis for $I$ by simply forgetting the polynomials using the variables $z_1,\ldots, z_\ell$ (as we did above).  
The magic of Gröbner bases is that they allow us to compute a basis for the ideal $I_\ell$ in precisely this naive way.  

\begin{theorem}[The Elimination Theorem]
    \label{thm: elimination}
    Let $I$ be an ideal in the polynomial ring $k[z_1,\ldots, z_N]$, and let $G$ be a Gröbner basis for $I$ with respect to the lexicographic term order where $z_1\succ \cdots \succ z_N$. 
    Then, for $\ell\in[N]$, the set 
    \[
    G_\ell = G\cap k[z_{\ell + 1}, \ldots, z_N]
    \]
    is a Gröbner basis for the $\ell$-th elimination ideal $I_\ell$.
\end{theorem}

When the ideal in Theorem~\ref{thm: elimination} is the ideal $I_\varphi$ for some rational map $\varphi$ as in~\eqref{eqn: elimination map}, the elimination ideal that forgets the parameters $\theta_1,\ldots, \theta_d$ and the extra variable $y$ is exactly the vanishing ideal of our statistical model. 

\begin{theorem}[Rational Implicitization]
    \label{thm: rational implicitization}
    Let $k$ be an infinite field, let $\varphi: k^d \setminus W \to k^N$ be a rational map, and consider the ideal $I^\varphi$ living in the polynomial ring $k[\theta_1,\ldots, \theta_d, y, z_1,\ldots, z_N]$ equipped with the lexicographic term order $\succ$ with $\theta_1 \succ \cdots \succ\theta_d \succ y \succ z_1 \succ\cdots\succ z_N$. 
    Then $V((I_\varphi)_{d + 1})= \overline{\varphi(k^d\setminus W)}$, and $I(\varphi(k^d\setminus W))$ has basis $G_{d+1}$, where $G$ is a Gröbner basis for $I_\varphi$ with respect to $\succ$. 
\end{theorem}

This result means that we can use computational algebra software, such as Macaulay2 \cite{M2}, to explicitly compute a basis of model-defining polynomial constraints for a given parametric graphical model. 

\begin{example}
    \label{ex: UG implicitization}
    Let $\mathcal G$ be the undirected graph from Example~\ref{ex: UG path example}.  
    We can use Macaulay2 and Theorem~\ref{thm: rational implicitization} to verify the claim in~\eqref{eqn: UG vanishing ideal} that the vanishing ideal $I(\mathcal M_{\textrm{Gauss}}^\ast(\mathcal G))$ has basis
    \[
    \sigma_{12}\sigma_{23}- \sigma_{22}\sigma_{13}. 
    \]
    To do this, recall the parameterization $\varphi_\mathcal G$ of the undirected Gaussian graphical model $\mathcal M_{\textrm{Gauss}}(\mathcal G)$ given in~\eqref{eqn: UG cov param}. 
    Note that the ideal $I_{\varphi_\mathcal G}$ lives in the polynomial ring $\mathbb R[k_{11}, k_{22}, k_{33}, k_{12}, k_{23},y, \sigma_{11}, \sigma_{22},\sigma_{33}, \sigma_{12},\sigma_{13}, \sigma_{23}]$, where the $k_{ij}$ denote the concentration parameters and $\sigma_{ij}$ are the covariance parameters.  
    $\mathcal M^\ast_{\textrm{Gauss}}(\mathcal G)$ lives in the space of $3\times 3$ covariance matrices $\Sigma = (\sigma_{ij})_{i,j=1}^3$, so we want to eliminate the $k_{ij}$. 

    Following, the story outlined above, we start by defining the ideal 
    \[
    I_{\varphi_\mathcal G} = \langle \{|K|\sigma_{ij} - (-1)^{i+j}|K_{[3]\setminus\{j\},[3]\setminus \{i\}}| : 1 \leq i \leq j\leq 3\}\cup \{y|K| - 1\} \rangle.
    \]
    We can then use Macaulay2 to compute a Gröbner basis for this ideal with respect to the lexicographic term order. 
    \begin{verbatim}
    
S = QQ[k11, k22, k33, k12, k23, y, s11, s22, s33, s12, s13, s23, 
       MonomialOrder=> Lex];

K = matrix{{k11, k12, 0},
           {k12, k22, k23},
           {0,k23, k33}};
           
g = det(K);
f11 = det(submatrix(K, {1,2},{1,2}));
f22 = det(submatrix(K, {0,2},{0,2}));
f33 = det(submatrix(K, {0,1},{0,1}));
f12 = det(submatrix(K, {1,2},{0,2}));
f13 = det(submatrix(K, {1,2},{0,1}));
f23 = det(submatrix(K, {0,2},{0,1}));

I = ideal(g*s11 - f11, g*s22 - f22, g*s33 - f33, 
          g*s12 + f12, g*s13 - f13, g*s23 + f23, 
          g*y - 1);

gens gb I
    \end{verbatim}

The result is a Gröbner basis for the ideal $I_{\varphi_\mathcal G}$ with respect to the lexicographic term order.  It contains sixteen polynomials:
\[
\begin{array}{cc}
\mathit{s22}\,\mathit{s13}-\mathit{s12}\,\mathit{s23}&
\mathit{k23}\,\mathit{s33}\,\mathit{s12}-\mathit{k23}\,\mathit{s13}\,\mathit{s23}+\mathit{s13}\\
y-\mathit{s11}\,\mathit{s22}\,\mathit{s33}+\mathit{s11}\,\mathit{s23}^{2}+\mathit{s33}\,\mathit{s12}^{2}-\mathit{s12}\,\mathit{s13}\,\mathit{s23}&\mathit{k23}\,\mathit{s22}\,\mathit{s33}-\mathit{k23}\,\mathit{s23}^{2}+\mathit{s23}\\
\mathit{k12}\,\mathit{s11}\,\mathit{s23}-\mathit{k12}\,\mathit{s12}\,\mathit{s13}+\mathit{s13}&\mathit{k12}\,\mathit{s11}\,\mathit{s22}-\mathit{k12}\,\mathit{s12}^{2}+\mathit{s12}\\
\mathit{k33}\,\mathit{s23}+\mathit{k23}\,\mathit{s22}&\mathit{k33}\,\mathit{s13}+\mathit{k23}\,\mathit{s12}\\
\mathit{k33}\,\mathit{s33}+\mathit{k23}\,\mathit{s23}-1&\mathit{k22}\,\mathit{s23}+\mathit{k12}\,\mathit{s13}+\mathit{k23}\,\mathit{s33}\\
\mathit{k22}\,\mathit{s13}-\mathit{k12}\,\mathit{k23}\,\mathit{s11}\,\mathit{s33}+\mathit{k12}\,\mathit{k23}\,\mathit{s13}^{2}&\mathit{k22}\,\mathit{s12}+\mathit{k12}\,\mathit{s11}+\mathit{k23}\,\mathit{s13}\\
\mathit{k22}\,\mathit{s22}+\mathit{k12}\,\mathit{s12}+\mathit{k23}\,\mathit{s23}-1&\mathit{k11}\,\mathit{s13}+\mathit{k12}\,\mathit{s23}\\
\mathit{k11}\,\mathit{s12}+\mathit{k12}\,\mathit{s22}&\mathit{k11}\,\mathit{s11}+\mathit{k12}\,\mathit{s12}-1
\end{array}
\]

We see that exactly one polynomial uses only the variables $\sigma_{ij}$ (which are denoted \texttt{sij} in the above list). 
By Theorem~\ref{thm: rational implicitization}, this polynomial is the sole generator for the vanishing ideal of the model $\mathcal M_{\textrm{Gauss}}(\mathcal G)$.  
Indeed, this is the first polynomial in the above list
$
\sigma_{22}\sigma_{13} - \sigma_{12}\sigma_{23},
$
so we have that 
\[
I(\mathcal M_{\textrm{Gauss}}^\ast(\mathcal G)) = \langle \sigma_{22}\sigma_{13} - \sigma_{12}\sigma_{23}\rangle,
\]
which verifies the claim in~\eqref{eqn: UG vanishing ideal}.

\end{example}

Example~\ref{ex: UG implicitization} is particularly nice, since the basis we recovered with our computational methods has a natural statistical interpretation via~\eqref{eqn: Gauss CI}; namely it is encoding the global Markov property for the graphical model.  
Sometimes, we will get lucky and the polynomials produced by these computational methods will be easy to interpret. 
They may even reveal a general result for the constraints defining each model in the family.  
However, this luck is contingent on the term order being well-chosen. 
Indeed, changing the term order changes the Gröbner basis, and one basis may be more statistically interpretable than another. 

A second caveat is that computational methods become infeasible as the number of parameters grows. 
So we use these methods to get a sense of the constraints defining low-dimensional models in the family, and then prove results based on these observations. 

\begin{remark}[On the study of toric models]
    \label{rem: toric}
    One family of varieties where proving a specific form of a Gröbner basis is more achievable are (affine) \emph{toric varieties}, which are models parameterized by rational functions $\varphi$ in which each rational function is a \emph{monomial}; i.e.,  $\frac{f_i(\theta)}{g_j(\theta)} = \theta_1^{a_{1,1}}\cdots \theta_d^{a_{d,1}}$ for all $i = 1,\ldots, N$. 
    For these models, the vanishing ideal is always generated by \emph{binomials} (i.e., polynomials of the form $z^a - z^b$), and these binomials are related in the combinatorial structure of an associated \emph{polytope}, which is the convex hull of the exponent vectors $(a_{1,i},\ldots, a_{d,i})$ for $i \in[N]$. 
    
    This can solve both problems we just described. 
    The binomials often have a nice interpretation due to the combinatorics of the polytope, and this combinatorial structure makes it easier to prove a general form for a basis of the vanishing ideal. 
    % This combinatorial encoding can make it easier to compute a general form for the basis of the vanishing ideal, giving us useful equations for tasks such as model selection and model distinguishability. 
    However, many toric models are not obviously toric from their standard parameterizations, which is one reason there are many papers in the field of \emph{algebraic statistics} that describe when  certain statistical models have toric vanishing ideals \cite{coons23, DG, DS, geiger, gorgen22, maraj26, misra21, misra22, nicklasson}. 
\end{remark}

In some cases, there is an easier way to extract polynomials in the vanishing ideal for \emph{any} model in the family, and these polynomials are often directly statistically interpretable. 
This method is a different version of algebraic implicitization, which we discuss in the next section.

\subsection{Generically birational implicitization: An algebraic Markov property}
\label{subsubsec: candidate polynomials}

The computational methods described in the previous section give us a useful tool for finding a set of polynomial constraints that define a \emph{single} dense statistical model.  
However, if one intends to use Gröbner bases to compute these constraints for an arbitrary model $\mathcal{M}_{f, \Theta}(\mathcal G)$ in some graphical model family $\mathcal F_{\mathbb G, \Psi}$, then we would have to go beyond Macaulay2 and prove that a certain set of polynomials is indeed a Gröbner basis for the vanishing ideal of each model in the family. 
Accomplishing this task can be more or less difficult, depending on the family $\mathcal F_{\mathbb G, \Psi}$.

However, for many parametric families of graphical models there is an alternative (algebraic) approach to finding an implicit representation of the models. 
This technique amounts to finding an algebraic analogue of a Markov property. 

Our next techique is easiest to understand if we start with a dense graphical model family $\mathcal F_{\mathbb G, \Psi}$ that satisfies global rational (parameter) identifiability (see Definition~\ref{def: global identifiability}).
A parametric graphical model with rational realization $\mathcal M_{f,\Theta}^\ast(\mathcal G)\subseteq k^N$ satisfies global rational identifiability if for all $(z_1,\ldots, z_N)\in \mathcal M_{f, \Theta}^\ast(\mathcal G)$ and $i\in[d]$
\[
\theta_i = \frac{f_i(z_1,\ldots, z_N)}{g_i(z_1,\ldots, z_N)} 
\]
for some polynomials $f_i,g_i\in k[z_1,\ldots, z_N]$. 
In particular, if the parameter space $\Theta$ for $\mathcal M_{f, \Theta}^\ast(\mathcal G)$ satisfies some polynomial equation
\begin{equation}
\label{eqn: parameter constraint}
h(\theta_1,\ldots, \theta_d) = 0
\end{equation}
then substituting $\frac{f_i(z_1,\ldots, z_N)}{g_i(z_1,\ldots, z_N)}$ for every $\theta_i$ in $h(\theta_1,\ldots, \theta_d)$ and multiplying by a sufficiently large power $N$ of $g_1(z_1,\ldots, z_N)\cdots g_d(z_1,\ldots, z_N)$ produces a polynomial in $k[z_1,\ldots, z_N]$ that must belong to the vanishing ideal $I(\mathcal M_{f, \Theta}^\ast(\mathcal G))$: 
\begin{equation}
\label{eqn: clearing denominators}
g = (g_1(z_1,\ldots, z_N)\cdots g_d(z_1,\ldots, z_N))^Nh\left(\frac{f_i(z_1,\ldots, z_N)}{g_i(z_1,\ldots, z_N)}\right) \in I(\mathcal M_{f, \Theta}^\ast(\mathcal G)).
\end{equation}
The polynomial $g$ is simply a translation of the constraint $h(\theta_1,\ldots, \theta_d) = 0$ on our parameter space $\Theta\subseteq k^d$ to a constraint on the model.  
As the following example shows, polynomials obtained in this way amount to an algebraic generalization of the Markov properties for CI graphical models discussed in Sections~\ref{subsubsec: UG models} and~\ref{subsubsec: DAG models}. 

\begin{example}
    \label{ex: markov polynomials}
    Let $\mathcal G\in \mathbb{DAG}$, and consider the linear Gaussian DAG model $\mathcal M_{\textrm{Gauss}}(\mathcal G)$ from Section~\ref{subsubsec: parametric DAGs}.  For this model, the only polynomial equations satisfied by the parameters $(\lambda, \omega)$ are 
\[
\lambda_{ij} = 0 \qquad \textrm{ for $i\to j \notin E$.}
\]
Since these models satisfy global rational identifiability via the rational functions in~\eqref{eqn: Gaussian DAG identification}, we can make the substitutions described above to recover 
\[
\begin{split}
0 &= \lambda_{ij}\\
0 &=\frac{|\Sigma_{\{i\}\cup \pa(j),\{j\}\cup \pa(j)}|}{|\Sigma_{\pa(j), \pa(j)}|} \\
0 &= |\Sigma_{\pa(j), \pa(j)}|\cdot\frac{|\Sigma_{\{i\}\cup \pa(j),\{j\}\cup \pa(j)}|}{|\Sigma_{\pa(j), \pa(j)}|} \\
0 &= |\Sigma_{\{i\}\cup \pa(j),\{j\}\cup \pa(j)}| 
\end{split}
\]
Since $\lambda_{ij} = 0$ and $\lambda_{ij} = \frac{|\Sigma_{\{i\}\cup \pa(j),\{j\}\cup \pa(j)}|}{|\Sigma_{\pa(j), \pa(j)}|}$ hold for every parameter choice $(\lambda, \omega) \in \Theta$ and the corresponding $\Sigma = \varphi(\lambda, \omega)\in \mathcal M_{\textrm{Gauss}}^\ast(\mathcal G)$, it follows that the final (polynomial) equation must be true for every $\Sigma \in \mathcal M_{\textrm{Gauss}}^\ast(\mathcal G)$.  In other words, 
\[
|\Sigma_{\{i\}\cup \pa(j),\{j\}\cup \pa(j)}| \in I(\mathcal M_{\textrm{Gauss}}^\ast(\mathcal G)).
\]
By~\eqref{eqn: Gauss CI}, we know that $|\Sigma_{\{i\}\cup \pa(j),\{j\}\cup \pa(j)}| = 0$ if and only if $X_i \independent X_j \mid X_{\pa(j)}$ whenever $X\sim \mathcal N(0,\Sigma)$.  
Moreover, since all distributions in this model are Gaussian, it can be shown that satisfying these CI relations for all $j\in[m]$ and $i\in \nd(j)\setminus\pa(j)$ is equivalent to satisfying the local Markov property for $\mathcal G$. 
In particular, the polynomials recovered via the method described above are (algebraically) recovering the local Markov property for DAGs in the Gaussian setting. 
\end{example}

\begin{example}
    \label{ex: birationality for UGs}
    It is a nice exercise to follow the same procedure as in Example~\ref{ex: markov polynomials} to identify a family of polynomials belonging to the vanishing ideal of an \emph{undirected} Gaussian graphical model $\mathcal M_{\textrm{Gauss}}(\mathcal G)$.  
    In this case, we obtain a family of polynomials belonging to the vanishing ideal $I(\mathcal M^\ast_{\textrm{Gauss}}(\mathcal G))$ for an undirected graph $\mathcal G \in \mathbb{UG}$.  What Markov property do these polynomials recover?
\end{example}

In Examples~\ref{ex: markov polynomials} and~\ref{ex: birationality for UGs}, we saw that the only constraints on the parameter space were constraints that translate into a Markov property for the given graph.  
However, we have also seen examples of parametric graphical model families defined on parameter spaces that satisfy additional polynomial constraints. 
One example are the colored Gaussian DAG models from Section~\ref{subsubsec: colored DAGs}.

\begin{example}
    \label{ex: colored DAGs}
    Consider the following pair of colored DAGs on $3$ nodes:
    \begin{center}
    \begin{tikzpicture}[thick, scale=0.6]
    \node[circle, fill=red!50, draw, inner sep=1pt, minimum width=1pt] (1u) at (0,0)  {$1$};
    \node[circle, fill=green!50, draw, inner sep=1pt, minimum width=1pt] (2u) at (2,0.24) {$2$};
    \node[circle, fill=orange!50, draw, inner sep=1pt, minimum width=1pt] (3u) at (4,-0.2) {$3$};
    
    \draw[->, blue!50, very thick] (1u) -- (2u);
    \draw[->, blue!50, very thick] (2u) -- (3u);

    \node at (-2,0) {$(\mathcal G,c) =$};

    \node[circle, fill=red!50, draw, inner sep=1pt, minimum width=1pt] (1) at (0 + 10,0)  {$1$};
    \node[circle, fill=green!50, draw, inner sep=1pt, minimum width=1pt] (2) at (2 + 10,-0.15) {$2$};
    \node[circle, fill=orange!50, draw, inner sep=1pt, minimum width=1pt] (3) at (4 + 10,0.1) {$3$};
    
    \draw[<-, blue!50, very thick] (1) -- (2);
    \draw[<-, blue!50, very thick] (2) -- (3);

    \node at (-2 + 10,0) {$(\mathcal H,c') =$};
    
    \end{tikzpicture}
    \end{center}
    According to the definition of a colored Gaussian DAG model given in Section~\ref{subsubsec: colored DAGs}, the parameter space of the model $\mathcal M_{\textrm{Gauss}}(\mathcal G, c)$ satisfies the polynomial equation 
    \[
    \lambda_{12} - \lambda_{23} = 0, 
    \]
    while the parameter space of the model $\mathcal M_{\textrm{Gauss}}(\mathcal H, c')$ satisfies a different polynomial equation
    \[
    \lambda_{32} - \lambda_{21} = 0.
    \]
    We saw in Section~\ref{subsubsec: GMP linear Gaussian DAGs} that linear Gaussian DAG models are globally rationally identifiable. 
    In particular, given $\Sigma \in \mathcal M_{\textrm{Gauss}}^\ast(\mathcal G, c)$, the parameter values $(\lambda, \omega)$ satisfying $\varphi_\mathcal G(\lambda, \omega) = \Sigma$ are recovered via the rational functions given in~\eqref{eqn: Gaussian DAG identification}.
    Substituting these formulas into our polynomial equations in the parameters, we obtain the, respective, rational functions
    \begin{equation*}
    \begin{split}
    \frac{|\Sigma_{\pa(2),\{2\}\cup \pa(2)\setminus\{1\}}|}{|\Sigma_{\pa(2),\pa(2)}|} - \frac{|\Sigma_{\pa(3),\{3\}\cup \pa(3)\setminus\{2\}}|}{|\Sigma_{\pa(3),\pa(3)}|} &= \frac{\sigma_{12}}{\sigma_{11}} - \frac{\sigma_{23}}{\sigma_{22}}, \textrm{ and}\\
    \frac{|\Sigma_{\pa(2),\{2\}\cup \pa(2)\setminus\{3\}}|}{|\Sigma_{\pa(2),\pa(2)}|} - \frac{|\Sigma_{\pa(1),\{1\}\cup \pa(1)\setminus\{2\}}|}{|\Sigma_{\pa(1),\pa(1)}|} &= \frac{\sigma_{23}}{\sigma_{33}} - \frac{\sigma_{12}}{\sigma_{22}}.
    \end{split}
    \end{equation*}
    We then produce polynomials by multiplying these rational functions by the product of the denominators of their individual terms.  
    It is not hard to see that the resulting polynomials must belong to the, respective, vanishing ideals since they are derived by simply phrasing the defining equations of the parameter space in terms of the functions of $\sigma_{ij}$ that recover those parameter values. 
    In particular, we have
    \begin{equation*}
        \begin{split}
            &\sigma_{12}\sigma_{22} - \sigma_{11}\sigma_{23}\in I(\mathcal M_{\textrm{Gauss}}^\ast(\mathcal G, c)), \textrm{ and}\\
            &\sigma_{22}\sigma_{23} - \sigma_{12}\sigma_{33}\in I(\mathcal M_{\textrm{Gauss}}^\ast(\mathcal H, c')).
        \end{split}
    \end{equation*}
    It is a healthy exercise to write code in Macaulay2 to verify that each of these polynomials does not belong to the vanishing ideal of the other model.
\end{example}

This technique for recovering polynomials in the vanishing ideal $I(\mathcal M^\ast_{f, \Theta}(\mathcal G))$ amounts to recovering a Markov property of sorts for the model. 
For instance, it is not hard to see that the small example for colored Gaussian DAGs models generalizes. 

\begin{trailer}{A Markov property for colored Gaussian DAG models}
If $(\mathcal G, c)$ is a colored DAG then the following polynomials belong to the vanishing ideal $I(\mathcal M_{\textrm{Gauss}}^\ast(\mathcal G, c))$:
\begin{enumerate}
    \item $
    |\Sigma_{\{i\}\cup\pa(j),\{j\}\cup \pa(j)}|
    $
    whenever $i\to j \notin E$,
    
    \item $
    |\Sigma_{\pa(j),\{j\}\cup \pa(j)\setminus\{i\}}||\Sigma_{\pa(\ell),\pa(\ell)}| - |\Sigma_{\pa(\ell),\{\ell\}\cup \pa(\ell)\setminus\{k\}}||\Sigma_{\pa(j),\pa(j)}|
    $
    whenever the edges $i\to j, k\to \ell$ have the same color, and 
    
    \item $
    |\Sigma_{\{i\}\cup\pa(i),\{i\}\cup\pa(i)}||\Sigma_{\pa(j),\pa(j)}| - |\Sigma_{\{j\}\cup\pa(j),\{j\}\cup\pa(j)}||\Sigma_{\pa(i),\pa(i)}|
    $
    whenever the vertices $i,j\in[m]$ have the same color.
    
\end{enumerate}
\end{trailer}
It is a nice exercise to prove this Markov property and give a statistical interpretation of each polynomial in terms of conditional covariances of $X\sim \mathcal N(0,\Sigma)$. 
Since the Markov property for colored Gaussian DAG models is a generalization of Example~\ref{ex: markov polynomials}, these polynomials can be viewed as an \emph{algebraic analogue} of the Markov properties for CI graphical models. 

The method is not limited to Gaussian models, but applies more generally.  
So we summarize the strategy here. 
We then give some conditions under which the strategy recovers an algebraic Markov property (e.g. polynomial constraints defining the model) in Theorem~\ref{thm: birational MP}.

\begin{trailer}{Generically birational implicitization}
Let $\mathcal M$ be a statistical model with parameter space $\Theta\subseteq k^d$ and rational realization $\mathcal M^\ast = \varphi(\Theta)\subseteq k^N$. 
Let $\psi: Y \to \Theta$ be a rational map  where $\psi(z) = (f_1(z)/g_1(z),\ldots, f_d(z)/g_d(z))$ for $z\in Y$ and $Y$ is a dense subset of $\mathcal M^\ast$. 
\begin{enumerate}
    \item[] For a polynomial $h(\theta_1, \ldots, \theta_d)$ evaluating to $0$ on all $\theta \in \Theta$: 
    \begin{enumerate}
        \item[(1)] Construct the rational function $r_h(z) = h(f_1(z)/g_1(z),\ldots, f_d(z)/g_d(z))$.
        \item[(2)] Find a positive integer $N$ such that $(g_1(z)\cdots g_d(z))^Nr(z)$ is a polynomial in $z_1\ldots, z_N$.
        \item[(3)] Return the polynomial $p_h(z) = (g_1(z)\cdots g_d(z))^Nr(z)$.
    \end{enumerate}
\end{enumerate}
    
\end{trailer}

As we saw in our examples above, the right choice of rational map $\psi$ will return polynomials $p_h$ in the vanishing ideal of the model $\mathcal M$, and this set of polynomials amounts to an algebraic Markov property for the parametric graphical model.  
This will work if, for \emph{almost all} distributions in the model, the rational map $\psi$ is recovering the parameters defining the distribution.  
One condition under which there is a good candidate for a $\psi$ with this property is known as \emph{generic identifiability}.

\begin{definition}[Generic identifiability]
    \label{def: generic identifiability}
    Let $\mathcal M$ be a statistical model with parameter space $\Theta \subseteq k^d$ and rational realization $\mathcal M^\ast = \varphi(\Theta)\subset k^N$. 
    We say that $\mathcal M$ satisfies \emph{generic identifiability} if there exists a rational map $\psi$ and a variety $U\subset k^d$ such that $\psi (\varphi(\theta)) = \theta$ for all $\theta \in \Theta\setminus U$ and $\Theta \cap U$ has strictly lower dimension than $\Theta$. 
\end{definition}

If $\mathcal M$ is generically identifiable and the smallest affine variety containing $\varphi(\Theta\cap U)$ intersects $\mathcal M^\ast$ in a strictly lower-dimensional subset of $\mathcal M^\ast$, then $\psi$ acts as an inverse on $\mathcal M^\ast$ minus the measure zero set $\mathcal M^\ast \cap \overline{\varphi(\Theta\cap U)}$.  
Since the smallest affine variety containing $\mathcal M^\ast\setminus \overline{\varphi(\Theta\cap U)}$ will add back in this missing piece, then any polynomial constraints arising from $\psi$ will naturally belong to the vanishing ideal $I(\mathcal M^\ast)$.  
In other words, the polynomials $p_h$ obtained from generically birational implicitization will be valid constraints on the model, yielding an `algebraic' Markov property for $\mathcal M$. 

\begin{theorem}[generically birational implicitization]
    \label{thm: birational MP}
    Let $\mathcal M$ be a generically identifiable statistical model with rational realization $\mathcal M^\ast = \varphi(\Theta)$.  
    If $\overline{\varphi(\Theta\cap U))}\cap \mathcal M^\ast$ has strictly lower dimension than $\mathcal M^\ast$, then all polynomials $p_h$ returned by  generically birational implicitization belong to the vanishing ideal $I(\mathcal M^\ast)$. 
\end{theorem}

\begin{remark}
The method discussed here is a version of ``implicitization up to saturation" in the language of algebraic geometry. 
For a statistical model fulfilling the assumptions of Theorem~\ref{thm: birational MP} with rational realization $\mathcal M^\ast$, we have computed a basis $\{h_1,\ldots, h_s\}$ for an ideal $I\subseteq k[z_1,\ldots, z_N]$ such that there is a multiplicatively closed set of polynomials $S\subseteq k[z_1,\ldots, z_N]$ satisfying
\[
I(\mathcal M^\ast) = I : S,
\]
where 
\[
I: S = \{f\in k[z_1,\ldots, z_N] : gf\in I \textrm{ for some } g\in S\}
\]
is called the \emph{saturation} of $I$ by $S$. 
The set $S$ is produced by taking products of the denominators $g_i(z)$ in our rational functions identifying the parameters $\theta_i = f_i(z)/g_i(z)$. 
The ideal $I:S$ is undoing (as much as possible) the multiplication by elements of $S$ that we did in~\eqref{eqn: clearing denominators} to produce the polynomials $h_1,\ldots, h_s$. %, to extract the vanishing ideal from the ideal $I = \langle h_1,\ldots, h_s\rangle$. 
Geometrically, $V(I:S)$ is all solutions to the system of equations $h_1(z) = 0,\ldots, h_s(z) = 0$ that are \emph{not} solutions to $g_1(z)\cdots g_d(z)$. 
The solutions we are throwing away are exactly those that make the product of denominators $g_1(z)\cdots g_d(z)$ evaluate to $0$. 
These solutions cannot correspond to distributions in our model $\mathcal M$.  
Hence, $\{h_1,\ldots, h_s\}$ is a sufficient set of constraints to define our model when we toss out a `non-probabilistic' subset of $k^N$. 

This algebraic statement is proven in \cite[Lemma 10]{boege}. 
The method can also be extended to models where we need to understand the constraints on $\mathcal M^\ast$ coming from polynomial \emph{inequalities} defining $\Theta$, which replace~\eqref{eqn: parameter constraint} with $h(\theta_1,\ldots, \theta_d)> 0$.  
This extension is described in rigorous detail in \cite{boegeSolus}, but it works just as above.  
In our examples, the only polynomial inequalities defining our parameter space are $\omega_i >0$, which translate to model-defining constraints $|\Sigma_{[i],[i]}|>0$.  
We did not consider these here, since they are equivalent to asserting that the covariance matrix is positive definite (a triviality). 
\end{remark}

Unlike the Gröbner basis we could compute for an individual model in Section~\ref{subsubsec: gröbner bases}, the polynomials recovered via Theorem~\ref{thm: birational MP} have a closed-form expression for \emph{any} model $\mathcal M_{f, \Theta}(\mathcal G)$ in our parametric graphical model family $\mathcal F_{\mathbb G, \Psi}$.   
As we will see in the coming sections, the polynomials constituting the \emph{algebraic Markov property} obtained via Theorem~\ref{thm: birational MP} are useful in solving model distinguishability problems, just as Markov properties were useful for the same problem in the case of CI DAG models.

\subsection{An algebraic strategy for proving structural identifiability}
\label{subsubsec: algebraic strategy}

In this section, we describe how to use the properties of ideals and varieties described in Section~\ref{subsec: polynomial constraints}, and the technique of generically birationally implicitization in Theorem~\ref{thm: birational MP}, to prove model distinguishability results in the graphical models program. 
In this section, we assume $\mathcal F_{\mathbb G, \Psi}$ is a dense family of parametric graphical models (see Definition~\ref{def: dense family}). 
Our goal is to outline an algebraic strategy for proving that $\mathcal M_{f, \Theta}(\mathcal G) \neq  \mathcal M_{f, \Theta}(\mathcal H)$ whenever $\mathcal G \neq \mathcal H$. 
This gives a general technique for proving structural identifiability of the family $\mathcal F_{\mathbb G, \Psi}$ (see Definition~\ref{def: structural identifiability}). 
Consequently, the technique will settle the model distinguishability question in the graphical models program for $\mathcal F_{\mathbb G, \Psi}$, while also expanding the very short list of structurally identifiable graphical models families given in Section~\ref{subsubsec: structural identifiability}. 

The strategy is rooted in the following observation.

\begin{proposition}
    \label{prop: algebraic model equvialence}
    Let $\mathcal M_{f, \Theta}(\mathcal G), \mathcal M_{f, \Theta}(\mathcal H)$ be two models in the parametric graphical model family $\mathcal F_{\mathbb G, \Psi}$. 
    \begin{enumerate}
        \item If $\mathcal M_{f, \Theta}^\ast(\mathcal G)$ and $\mathcal M_{f, \Theta}^\ast(\mathcal H)$ have different dimensions, then they are not model equivalent. 
        \item If the models have the same dimension then they are model equivalent if and only if $I(\mathcal M^\ast_{f, \Theta}(\mathcal G)) = I(\mathcal M^\ast_{f, \Theta}(\mathcal H))$
    \end{enumerate}
\end{proposition}

Proposition~\ref{prop: algebraic model equvialence}~(1) is true since the model with the larger dimension necessarily contains distributions not belonging to the lower-dimensional model. 
In many instances, it is not too difficult to compute the dimension of a parametric graphical model. 
This means the main tool for model distinguishability is Proposition~\ref{prop: algebraic model equvialence}~(2), 
which reduces the model distinguishability task to identifying a polynomial $g\in I(\mathcal M^\ast_{f, \Theta}(\mathcal G))$ such that $g\notin I(\mathcal M^\ast_{f, \Theta}(\mathcal H))$ for any two graphs $\mathcal G, \mathcal H \in \mathbb G$. 
One strategy, based on Proposition~\ref{prop: algebraic model equvialence}, is the following: 

\begin{trailer}{Strategy 1: Algebraically proving structural identifiability}
Let $\mathcal F_{\mathbb G, \Psi}$ be a dense family of parametric graphical models. 
    \begin{enumerate}
        \item[0.] Determine the dimension of $\mathcal M^\ast_{f, \Theta}(\mathcal G)$ for all $\mathcal G \in \mathbb G$. 
        \item[1.] Find a polynomial $g\in I(\mathcal M^\ast_{f, \Theta}(\mathcal G))$ for every $\mathcal G\in \mathbb G$ that could distinguish the model in the sense that $g$ might not belong to $I(\mathcal M^\ast_{f, \Theta}(\mathcal H))$ for $\mathcal H\neq \mathcal G$. 
        \item[2.] Use the polynomials from the previous step to argue that a variable $z_i$ is used by some polynomial in every basis for $I(\mathcal M^\ast_{f, \Theta}(\mathcal G))$ but no polynomial in any basis for $I(\mathcal M^\ast_{f, \Theta}(\mathcal H))$ uses $z_i$ when $\mathcal H \neq \mathcal G$. 
    \end{enumerate}
\end{trailer}

The logic of the approach outlined here is that the two ideals 
\[
I(\mathcal M^\ast_{f, \Theta}(\mathcal G)), I(\mathcal M^\ast_{f, \Theta}(\mathcal H))\subseteq k[z_1,\ldots, z_N]
\]
cannot be equal if every basis of one ideal requires a variable $z_i$ that does not appear in any basis for the other ideal.  
This makes sense, since if the two ideals were equal any basis for one would also be a basis for the other. 

In order to prove something about the polynomials belonging to the bases of these ideals, we need to know some polynomials that belong to the ideals. 
So, step~(1) is finding something for us to work with in step~(2). 
Finding these polynomials amounts to developing techniques for extracting polynomials from a vanishing ideal.

The development of such techniques is an active area of research in the fields of computational algebra and applied algebra, which we can generally refer to as \emph{the polynomial constraint extraction program}. 

\begin{trailer}{The polynomial constraint extraction program}
    Let $I\subseteq k[z_1,\ldots, z_N]$ be an ideal. Develop techniques for identifying families $F =\{f_1,\ldots, f_s\}$ of polynomials belonging to $I$. This includes identifying $F \subseteq I$ such that 
    \begin{enumerate}
        \item $I = \langle F\rangle$ \hfill (Theorem~\ref{thm: rational implicitization})
        \item $I = \langle F\rangle : S$ \hfill (Theorem~\ref{thm: birational MP})
        \item $F$ is a vector space basis for the polynomials in $I$ of a fixed degree. 
        \item $F\subseteq I$, but $F\not\subseteq J$ for another ideal $J$ of interest. 
    \end{enumerate}
\end{trailer}

In these notes, we discussed techniques for (1) and (2), which we called rational implicitization and generically birational implicitization, respectively. 
In general, many recent research activities in computational algebra, applied algebra, and algebraic statistics can be viewed as solving various instances of the polynomial constraint extraction program, including techniques for (3) and (4) not discussed here (e.g., \cite{cummings, GKSolus}).

%---Algebraic Implicitization in Action
\subsection{New structural identifiability results via algebraic implicitization}
\label{sec: example}
We end our discussions by showing how the polynomial constraint extraction program ties into the graphical models program.  
We will use the technique of generically birational implicitization (Theorem~\ref{thm: birational MP}) on a family of globally rationally identifiable graphical models $\mathcal F_{\mathbb G, \Psi}$ in order to apply Strategy~1 and solve the model distinguishability problem for $\mathcal F_{\mathbb G, \Psi}$. 

To highlight the value of the algebraic perspective, we do this for a family $\mathcal F_{\mathbb G, \Psi}$ for which there is no known probabilistic solution to the model distinguishability problem.  
Our example is the family of BPEC DAG models $\mathcal F_{\mathbb{BPEC}}$ defined in Definition~\ref{def: BPEC DAG}. 
The following theorem first appeared in \cite{boege}.

\begin{theorem}
    \label{thm: BPEC structural identifiability}
    The parametric graphical model family $\mathcal F_{\mathbb{BPEC}}$ is structurally identifiable. 
\end{theorem}

\paragraph{\emph{Step~(0) of Strategy~1:}}
Step~(0) is to deduce the dimension of $\mathcal M^\ast_{\textrm{Gauss}}(\mathcal G, c)$ for an arbitrary BPEC DAG $\mathcal G$. 
Note that the parameter space $\Theta$ for the model (see~\eqref{eqn: BPEC params}) is clearly $(c_V + c_E)$-dimensional, where $c_V$ and $c_E$ are, respectively, the number of colors used by $c$ to color the vertices and edges of $\mathcal G$.  
Moreover, the transformation $\varphi_\mathcal G(\lambda, \omega) = \Sigma$ is a diffeomorphism with a diffeomorphic inverse, from which it follows that $\mathcal M^\ast_{\textrm{Gauss}}(\mathcal G, c) = \varphi_\mathcal G(\Theta)$ is also $(c_V + c_E)$-dimensional.  
This can be seen via the differentiability of the Cholesky decomposition of a positive definite matrix.  
A detailed proof is in \cite[Theorem 7]{boege}. 

\begin{lemma}
    \label{lem: BPEC dimension}
    Let $(\mathcal G,c)$ be a BPEC DAG that is colored using $c_V$ colors for the vertices and $c_E$ colors for the edges. Then $\mathcal M^\ast_{\textrm{Gauss}}(\mathcal G, c)$ has dimension $c_V + c_E$. 
\end{lemma}

By Proposition~\ref{prop: algebraic model equvialence}~(1), we know that two models of different dimension cannot be model equivalent. 
Hence, Lemma~\ref{lem: BPEC dimension} implies that we only need to show that $\mathcal M_{\textrm{Gauss}}(\mathcal G, c) \neq \mathcal M_{\textrm{Gauss}}(\mathcal H, c')$ when $(\mathcal G, c)$ and $(\mathcal H, c')$ use the same number of colors. 

To prove two models of the same dimension are non-equal, it helps to know that each model contains distributions that exactly satisfy the dependence structure encoded by the graph. 
In other words, we will use the fact that every BPEC DAG model contains faithful distributions (see Definition~\ref{def: faithful}). 

\begin{lemma}
    \label{lem: BPEC faithful}
    Every $\mathcal M_{\textrm{Gauss}}(\mathcal G, c)\in \mathcal F_{\mathbb{BPEC}}$ contains distributions that are faithful to $\mathcal G$.
\end{lemma}

\begin{remark}[Faithless graphical models]
    \label{rem: faithlessness}
    Lemma~\ref{lem: BPEC faithful} fails to hold for general colored DAG models. Examples of colored DAG models $\mathcal M_{\textrm{Gauss}}(\mathcal G, c)$ that do not contain distributions faithful to $\mathcal G$ appear in \cite[Example~4]{boege}. 
    It would be interesting to characterize the colored DAGs $(\mathcal G, c)$ whose models do not contain distributions faithful to $\mathcal G$. 
\end{remark}

Lemma~\ref{lem: BPEC faithful} implies that if two BPEC models are model equivalent then their DAGs are Markov equivalent; i.e.~if $\mathcal M_{\textrm{Gauss}}(\mathcal G, c) = \mathcal M_{\textrm{Gauss}}(\mathcal H, c')$ then $\mathcal G$ and $\mathcal H$ are Markov equivalent.  
By Corollary~\ref{cor: VP}, this implies that $\mathcal G$ and $\mathcal H$ must have the same skeleton and v-structures. 
In particular, we have reduced the problem of structural identifiability for $\mathcal F_{\mathbb{BPEC}}$ to showing that $\mathcal M_{\textrm{Gauss}}(\mathcal G, c) \neq \mathcal M_{\textrm{Gauss}}(\mathcal H, c')$ whenever $\mathcal G$ and $\mathcal H$ have the same skeleton and v-structures. 

\paragraph{\emph{Step~(1) of Strategy~1:}}
Step~(1) then asks for polynomials that could potentially help us distinguish the vanishing ideals of two models. 
Since BPEC models are colored DAG models, we can use the polynomials in our Markov property for colored DAGs (in Section~\ref{subsubsec: candidate polynomials}), which we obtained via our solution to the implicitization up to saturation problem in the polynomial constraint extraction program. 
Specifically, we will use the polynomials from item~(2) in our Markov property for colored DAGs:
\[
g_{ij,k\ell} = |\Sigma_{\pa(j),\{j\}\cup \pa(j)\setminus\{i\}}||\Sigma_{\pa(\ell),\pa(\ell)}| - |\Sigma_{\pa(\ell),\{\ell\}\cup \pa(\ell)\setminus\{k\}}||\Sigma_{\pa(j),\pa(j)}|
\]
for $i\to j$ and $k\to \ell$ the same color in $(\mathcal G, c)$. 

\paragraph{\emph{Step~(2) of Strategy~1:}}
Now that we have our candidate polynomials $g_{ij,k\ell}\in I(\mathcal M^\ast_{\textrm{Gauss}}(\mathcal G,c))$, we can execute step~(2) of Strategy~1. 
That is, we would like to use $g_{ij,k\ell}$ to show that \emph{every} basis for $I(\mathcal M^\ast_{\textrm{Gauss}}(\mathcal G,c))$ uses a certain variable $\sigma_{st}$, but \emph{no} basis for $I(\mathcal M^\ast_{\textrm{Gauss}}(\mathcal H,c'))$ contains polynomials using $\sigma_{st}$. 

To show this, we first make some simplifying assumptions, noting that \cite{boege} explains how the remaining cases follow from the one we consider here. 
In the following, suppose that $(\mathcal G, c),(\mathcal H, c')$ are a pair of BPEC DAGs for which
\begin{enumerate}
    \item $\mathcal G$ and $\mathcal H$ are Markov equivalent with the same set of vertices $[m]$,
    \item $\mathcal G$ and $\mathcal H$ are connected, 
    \item there exists $j\in [m]$ such that $j$ is a \emph{sink} in $\mathcal G$ (i.e., $j$ has no children), but $j$ is not a sink in $\mathcal H$. 
\end{enumerate}
Item~(3) implies that there exist $\ell\in[m]$ such that $\ell\to j$ is an edge in $\mathcal G$, but the edge is reversed in $\mathcal H$ (i.e. $\ell \leftarrow j$ is an edge in $\mathcal H$, but not $\ell\to j$). 
By the definition of a BPEC DAG (Definition~\ref{def: BPEC DAG}) there exists a second edge $\ell\leftarrow i$ in $\mathcal H$ that is the same color as $\ell\leftarrow j$.  
Hence, we consider our polynomial $g_{i\ell, j\ell}$ for this pair of edges
\begin{equation}
\begin{split}
g_{i\ell,j\ell} &= |\Sigma_{\pa(\ell),\{\ell\}\cup \pa(\ell)\setminus\{i\}}||\Sigma_{\pa(\ell),\pa(\ell)}| - |\Sigma_{\pa(\ell),\{\ell\}\cup \pa(\ell)\setminus\{j\}}||\Sigma_{\pa(\ell),\pa(\ell)}|,\\
&= |\Sigma_{\pa(\ell),\pa(\ell)}|\left(|\Sigma_{\pa(\ell),\{\ell\}\cup \pa(\ell)\setminus\{i\}}| - |\Sigma_{\pa(\ell),\{\ell\}\cup \pa(\ell)\setminus\{j\}}|\right).
\end{split}
\end{equation}
We know that $g_{i\ell,j\ell}\in I(\mathcal M_{\textrm{Gauss}}^\ast(\mathcal H, c'))$. 
We also know that the principle minor $|\Sigma_{\pa(\ell),\pa(\ell)}|\neq 0$ for all $\Sigma \in \mathcal M^\ast_{\textrm{Gauss}}(\mathcal H, c')$, since all matrices in this set are positive definite. 
This implies that 
\[
h_{i\ell, j\ell} = |\Sigma_{\pa(\ell),\{\ell\}\cup \pa(\ell)\setminus\{i\}}| - |\Sigma_{\pa(\ell),\{\ell\}\cup \pa(\ell)\setminus\{j\}}| \in I(\mathcal M_{\textrm{Gauss}}^\ast(\mathcal H, c')).
\]
Notice that the variable $\sigma_{jj}$ appears in the minor $|\Sigma_{\pa(\ell),\{\ell\}\cup \pa(\ell)\setminus\{i\}}|$, but not in $|\Sigma_{\pa(\ell),\{\ell\}\cup \pa(\ell)\setminus\{j\}}|$. 
Hence, taking the appropriate Laplace expansion to compute the minor $|\Sigma_{\pa(\ell),\{\ell\}\cup \pa(\ell)\setminus\{i\}}|$, we see that we can re-write $h_{i\ell, j\ell}$ as
\[
h_{i\ell, j\ell} = \sigma_{jj}|\Sigma_{\pa(\ell)\setminus\{j\}, \{\ell\}\cup\pa(\ell)\setminus\{i,j\}}| + F,
\]
where $F$ is a polynomial in the ring $\mathbb R[\sigma_{ij} : 1\leq i < j \leq m]$ (i.e., it does not use $\sigma_{jj}$).

Suppose now that there is a basis $\{g_1,\ldots, g_s\}$ for the ideal $I(\mathcal M_{\textrm{Gauss}}^\ast(\mathcal H, c'))\subseteq \mathbb R[\sigma_{ij} : 1 \leq i \leq j\leq m]$ such that $\{g_1,\ldots, g_s\}\subseteq \mathbb R[\sigma_{ij} : 1 \leq i < j\leq m]$.  
It follows that there exist polynomials $h_1,\ldots, h_s, h_1',\ldots, h_s'\in \mathbb R[\sigma_{ij} : 1 \leq i \leq j\leq m]$ such that 
\[
h_{i\ell, j\ell} = \sigma_{jj} \sum_{k=1}^sh_kf_k + \sum_{k=1}^sh_k'f_k. 
\]
Comparing the two expressions for $h_{i\ell, j\ell}$, it follows that $|\Sigma_{\pa(\ell)\setminus\{j\}, \{\ell\}\cup\pa(\ell)\setminus\{i,j\}}| = \sum_{k=1}^sh_kf_k$.  
Since $f_1,\ldots, f_s$ is a basis for $I(\mathcal M^\ast_{\textrm{Gauss}}(\mathcal H, c'))$, this means that
\[
|\Sigma_{\pa(\ell)\setminus\{j\}, \{\ell\}\cup\pa(\ell)\setminus\{i,j\}}|  \in I(\mathcal M^\ast_{\textrm{Gauss}}(\mathcal H, c')).
\]
Statistically, this means that $X_\ell \independent X_i \mid X_{\pa(\ell)\setminus\{i,j\}}$ for all distributions $\mathcal N(0,\Sigma) \in \mathcal M_{\textrm{Gauss}}(\mathcal H, c')$. 
However, by Lemma~\ref{lem: BPEC faithful}, the model $\mathcal M_{\textrm{Gauss}}(\mathcal H, c')$ contains distributions that are faithful to $\mathcal H$. 
Since there is an edge $i\to \ell$ in $\mathcal H$, the existence of these faithful distribution implies that the CI relation $X_\ell \independent X_i \mid X_{\pa(\ell)\setminus\{i,j\}}$ cannot hold for all distributions in $\mathcal M_{\textrm{Gauss}}(\mathcal H, c')$.
Since we have arrived at this contradictory conclusion, we conclude that every basis of the ideal $I(\mathcal M^\ast_{\textrm{Gauss}}(\mathcal H, c'))$ must contain some polynomial using the variable $\sigma_{jj}$. 

On the other hand, it is a relatively straightforward exercise in the algebraic theory we have outlined in these notes to show that, since $j$ is a sink in $\mathcal G$, then no basis of the ideal $I(\mathcal M^\ast_{\textrm{Gauss}}(\mathcal G, c))$ contains a polynomial using the variable $\sigma_{jj}$. 
Hence, it follows that $I(\mathcal M^\ast_{\textrm{Gauss}}(\mathcal H, c'))\neq I(\mathcal M^\ast_{\textrm{Gauss}}(\mathcal G, c))$.
By Proposition~\ref{prop: algebraic model equvialence}, it follows that $\mathcal M_{\textrm{Gauss}}(\mathcal H, c')\neq \mathcal M_{\textrm{Gauss}}(\mathcal G, c)$, which proves Theorem~\ref{thm: BPEC structural identifiability}.

\begin{remark}
\label{rem: edge colorings}
    Presently, the only known proof of Theorem~\ref{thm: BPEC structural identifiability} uses the algebraic implicitization techniques outlined here. 
    In \cite{boege}, the authors applied the same techniques to prove that the family of colored Gaussian DAG models for colored DAGs in which all edges are the same color and all vertices are colored differently is also structurally identifiable. 
    Hence, algebraic implicitization techniques allow us to add these two families to the short list of structurally identifiable DAG model families listed in Section~\ref{subsubsec: structural identifiability}.  
    Note that three of the four families on this (expanded) list are families of colored Gaussian DAG models, since linear Gaussian DAG models with equal error variances are colored Gaussian DAG models for which the DAG is colored so that all vertices have the same color and all edges are colored differently. 
\end{remark}

\begin{remark}
\label{rem: vertex colorings}
    Another example that uses algebraic implicitization techniques to solve the model distinguishability question appeared recently in \cite{WD23}.  
    Here, the authors classified model equivalence classes for colored DAG models where the all edges in the DAG are colored differently, but no restrictions are placed on the vertex colors. 
    Although not stated directly, their method uses the concept of implicitization up to saturation.   
\end{remark}

Based on Remarks~\ref{rem: edge colorings} and~\ref{rem: vertex colorings}, it would be interesting to know if algebraic implicitization techniques offer a solution to the model distinguishability problem for the class of all colored Gaussian DAG models $\mathcal F_{\textrm{c-}\mathbb{DAG}}^{\textrm{Gauss}}$.

\begin{problem}
    \label{prod: colored characterization}
    Solve the model distinguishability problem for the family of colored Gaussian DAG models $\mathcal F_{\textrm{c-}\mathbb{DAG}}^{\textrm{Gauss}}$. 
    That is, characterize when two colored DAGs $(\mathcal G, c)$ and $(\mathcal H, c')$ satisfy $\mathcal M_{\textrm{Gauss}}(\mathcal G, c) = \mathcal M_{\textrm{Gauss}}(\mathcal H, c)$.
\end{problem}

The same problem could also be posed for other parametric families of graphical models, including mixed graph families.  
For instance, the technique of generically birational implicitization also applies to parametric families of directed ancestral graphical models (see the appendix of \cite{boege}), as well as the relatively new family of graphical Lyapunov models \cite{dlyapunov}. 
    
As a final note, we remark that algebraic implicitization techniques have also made recent appearances in the study of other families of graphical models \cite{homoloidal, faithlessness, dlyapunov}, including hypothesis tests for model selection \cite{barnhill, strieder}.

\end{document}